\documentclass[12pt]{article} 
\usepackage{amssymb}
\usepackage{makecell}
\usepackage{graphicx}
\usepackage{subcaption}
\usepackage{caption} 
\usepackage{array}
\usepackage{multirow}
\usepackage{amsmath}
\usepackage{multirow}
\usepackage{mathrsfs}
\usepackage{amsfonts}
\usepackage{indentfirst}
\usepackage{colortbl,dcolumn}
\usepackage{color}

\usepackage{amsmath}
\allowdisplaybreaks[4]
\usepackage{psfrag}

\usepackage{array}
\usepackage{booktabs}
\usepackage{cite}
\usepackage[hidelinks]{hyperref}
\usepackage{amsthm}
\allowdisplaybreaks
\numberwithin{equation}{section}
\usepackage[top=1in, bottom=1in, left=0.8in, right=0.8in]{geometry}
\renewenvironment{proof}%
  {\begin{trivlist}%
   \item[\hskip\labelsep]
   \ignorespaces}%
  {\qed\end{trivlist}}

\usepackage{caption}
\usepackage{appendix} 
\usepackage{graphicx}
\usepackage{subcaption}
\usepackage{subcaption}
\usepackage{float}
\usepackage{makecell}
\newcommand{\E}{\mathbb{E}}

\newtheorem{thm}{Theorem}[section]

\newtheorem{lem}[thm]{Lemma}
\newtheorem{prop}[thm]{Proposition}

\newtheorem{rem}[thm]{Remark}

\newtheorem{assumption}[thm]{Assumption}

\begin{document}
\bibliographystyle{plain}
\title{ Multimodal sampling via Schr\"odinger-F\"ollmer samplers with temperatures
\footnotemark[2] {}}
	\author{
          \ 
	Xiaojie Wang,
    Xiaoyan Zhang
		\\
		\footnotesize School of Mathematics and Statistics,   HNP-LAMA, Central South University, Changsha 410083, Hunan, China
		%
		%
		%
	}

\maketitle
\footnotetext[2]{This work was supported by Natural Science Foundation of China (12471394, 12371417, 12071488) and Hunan Basic Science Research Center for Mathematical Analysis (2024JC2002). 
                 E-mail addresses:
 x.j.wang7@csu.edu.cn, 232103003@csu.edu.cn.
 }

\begin{abstract}
Generating samples from complex and high-dimensional distributions is ubiquitous in various scientific fields of statistical physics, Bayesian inference, scientific computing and machine learning.
Very recently,  Huang et al. (IEEE Trans. Inform. Theory, 2025)  proposed new Schr\"odinger-F\"ollmer samplers (SFS), based on the Euler discretization of the Schr\"odinger-F\"ollmer diffusion evolving on the unit interval $[0, 1]$. 
There, a convergence rate of order $\mathcal{O}(\sqrt{h})$ in the $L^2$-Wasserstein distance was obtained for the Euler discretization with a uniform time step-size $h>0$.
%
%
By introducing a temperature parameter, different samplers are proposed in this paper, based on the Euler discretization of the Schr\"odinger-F\"ollmer process with temperatures.
As revealed by numerical experiments, high temperatures are vital, particularly in sampling from multimodal distributions.
Further, a novel approach of error analysis is developed for the time discretization and an enhanced convergence rate of order $\mathcal{O}(h)$ is obtained in the $L^2$-Wasserstein distance, under certain smoothness conditions on the drift. This significantly improves the existing order-half convergence in the aforementioned paper.
Unlike Langevin samplers, SFS is gradient-free, works in a unit interval $[0, 1]$ and does not require any ergodicity. 
Numerical experiments confirm the convergence rate and show that, the SFS substantially outperforms vanilla Langevin samplers, particularly in sampling from multimodal distributions.

\end{abstract}

\section{Introduction}
\noindent
Sampling from complex and high-dimensional unnormalised distributions of the form
\begin{align}
\label{SFS-W2-eq:targe-distribution}
    \mu(\mathrm{\,d} x) 
    \propto 
     e^
     {-V(x)} 
    \mathrm{\,d} x,
    \quad
    x \in \mathbb{R}^d,
    \
    d \gg 1,
\end{align}
where 
$V \colon\mathbb{R}^d \rightarrow \mathbb{R}$ is a potential function,
turns out to be a fundamental task in various fields ranging from Bayesian inference, statistical physics, machine learning to scientific computing. 
%
A popular sampling algorithm in a vast amount of literature is 
%
the overdamped Langevin Monte Carlo (LMC) algorithm \cite{DA2017,DAME2017,mou2022improved,li2022sqrt,Li2025Asharp,Altschuler2024ShiftedCI,Yang2025error}.
To understand the main idea of the overdamped LMC, let us look at
the overdamped Langevin stochastic differential equation (SDE):
\begin{align}
\label{SFS-W2-eq:Langevin-SDE}
    \mathrm{\,d} X_t
    =
    -\nabla V(X_t) 
    \mathrm{\,d} t
    +
    \sqrt{2} 
    \mathrm{\, d} W_t, 
    \quad X_0=x_0, \quad t>0,
\end{align}
where $\left\{W_t\right\}_{t \geq 0}$ is a standard $d$-dimensional Brownian motion process.
Under certain convexity type conditions imposed on the drift $-\nabla V(\cdot)$, the overdamped Langevin SDE is ergodic, 
admitting the target distribution 
$
\mu(\mathrm{\,d} x) 
    \propto 
     e^
     {-V(x)} 
    \mathrm{\,d} x
$ 
as its unique invariant distribution.
%
Therefore, a practical way to sample from the target distribution  turns to a long-time approximation of the Langevin SDE. Indeed, the  overdamped LMC is nothing but the  Euler discretization scheme  of \eqref{SFS-W2-eq:Langevin-SDE}:
\begin{align}
\label{SFS-W2-eq:ULA}
    \bar{X}_{n+1}
    =
    \bar{X}_{n}
    -h \nabla V ( \bar{X}_{n} )
    +
    \sqrt{2} \Delta W_n,
\end{align}
where $h>0$ is the uniform time step-size and $\Delta W_n := W_{t_{n+1}}- W_{t_{n}}$. In recent years, an extensive study has been devoted to the non-asymptotic convergence of the overdamped LMC.
The early non-asymptotic error analysis of LMC was carried out under

(i) {\bf gradient Lipschitz condition}: there exists a constant $ L_0>0 $ such that 
\begin{align}
\label{eq:intro-griadiant-Lipschitz}
\|\nabla V(x)-\nabla V(y)\| 
\leq 
L_0\|x-y\|, 
\quad \forall x, y \in \mathbb{R}^d; \quad \:\text{and}
\end{align}

(ii) {\bf strong convexity condition}: there exists a constant 
${\color{black}L_1>0}$ such that
\begin{align}\label{eq:intro-convexity-condition}
    \langle x-y, \nabla V(x)-\nabla V(y)\rangle 
    \geq 
    L_1\|x-y\|^2, 
    \quad \forall x, y \in \mathbb{R}^d. 
\end{align}

Under these conditions, the non-asymptotic error between the target distribution and the law of the overdamped LMC algorithm under various metrics, such as the Wasserstein distance, the total variation distance and the Kullback-Leibler divergence has been systemically studied in the literature (see, e.g., \cite{DA2017,DAME2017,DAME2019} and references therein).
%

To improve sampling efficiency, some variants of Langevin sampling algorithms {have been} introduced. A typical alternative is based on discretizations of the following underdamped Langevin dynamics in the state space $\mathbb{R}^{2d}$:
\begin{align}
    \label{SFS-W2-eq:under-SDE}
     \left\{\begin{array}{l}
     \mathrm{d} X_t 
     =
     M_t 
     \mathrm{\,d} t \\
     \mathrm{\,d} 
     M_t
     =-
     \nabla V\left(X_t\right) \mathrm{\,d} t
     -
     \gamma M_t 
     \mathrm{\,d} t
     +
     \sqrt{2 \gamma} 
     \mathrm{\,d} W_t,
\end{array}\right.
\end{align}
where $\gamma >0$ is the friction coefficient  and $\left\{W_t\right\}_{t \geq 0}$ is a standard $d$-dimensional Brownian motion. It is known that, under some assumptions on $\nabla V(\cdot)$, the underdamped Langevin SDE is ergodic and admits a unique stationary distribution
$
\mu(\mathrm{d} x, \mathrm{d} m) 
    \propto 
     e^
     { -V(x) - \frac12 \| m \|^2 } 
    \mathrm{\,d} x \mathrm{\, d } m .
$
Evidently, the $x$-marginal distribution of the stationary distribution is exactly the target distribution. 
%
On a uniform mesh with a uniform step-size $h>0$, the Euler scheme applied to the dynamics \eqref{SFS-W2-eq:under-SDE} reads:
\begin{align}
\label{SFS-W2-eq:under-EM}
    \left\{\begin{array}{l}
     \tilde{X}_{n+1} 
     =
     \tilde{X}_n
     +
     h \tilde{M}_n \\
     \tilde{M}_{n+1}
     =
     \tilde{M}_n
     -
     h \nabla V(\tilde{X}_n)
     -
     h \gamma \tilde{M}_n 
     +
     \sqrt{2 \gamma} 
     \Delta W_n,
\end{array}\right.
\end{align}
which is termed as the underdamped LMC.
Under the gradient Lipschitz condition \eqref{eq:intro-griadiant-Lipschitz} and the strong convexity condition \eqref{eq:intro-convexity-condition}, a large amount of research was devoted to the non-asymptotic error analysis of the underdamped LMC under various metrics \cite{cheng2018underdamped,DARL2020,CZVS2022}.
Other efficient samplers based on higher-order (splitting) numerical schemes 
in \cite{MP2021,LBNC2013,camrud2023second} for \eqref{SFS-W2-eq:under-SDE}
%
have been recently proposed and examined under conditions \eqref{eq:intro-griadiant-Lipschitz}-\eqref{eq:intro-convexity-condition}.
%

Nevertheless, the strong convexity condition is too restrictive in practice and researchers try to do the error analysis under non-convex conditions, including the contractivity at infinity condition \cite{PWW25,schuh2024convergence,DAME2017,EA2016,cheng2018sharp} and the log-Sobolev inequality condition \cite{mou2022improved,Li2025Asharp,Yang2025error}. Such conditions can be viewed as weak convexity conditions, used to ensure the 
ergodicity of the Langevin SDEs.
%
%
%
An interesting and natural question thus arises:
\\

{\it Instead of the Langevin sampling over infinite time, can one do efficient sampling over the unit interval $[0, 1]$ without ergodicity, to avoid (weak) convexity conditions?}
\\

%
In this direction, some researchers have made some efforts (see, e.g., \cite{Albergo2025Stochastic,AMVE2023,zhang2025stochastic,MRjiao,DYJ2023,Ruzayqat2023Unbiased}). 
%
Given the target distribution $\mu \in \mathcal{P}(\mathbb{R}^d )$ and a {  simple initial distribution }
$\nu \in \mathcal{P}(\mathbb{R}^d ) $ (e.g., Gaussian or {degenerate Dirac distribution }$\delta_0$), the aforementioned question can be addressed by finding and (numerically) solving an SDE over $[0,1]$ connecting these two distributions:
\begin{align*}
    \mathrm{d} X_t 
    = b ( X_t, t ) 
    \mathrm{\,d} t 
    + 
    \sigma ( X_t, t )
    \mathrm{\,d} W_t, 
    \quad
    X_0 \sim \nu,
    \quad
    X_1 \sim \mu,
    \quad
    t \in (0, 1].
\end{align*}
Unfortunately, finding the closed-form of the drift and diffusion coefficients (i.e., $b, \sigma $) is usually not an easy task (cf. \cite{Albergo2025Stochastic,AMVE2023}).
In \cite{follmer1988random,follmer2005entropy}, F\"ollmer proposed a Schr\"odinger-F\"ollmer diffusion process  in the context of the Schr\"odinger bridge problem \cite{schrodinger1932theorie}, connecting two distributions.
%
%
More precisely, for a target distribution 
$
\mu(\mathrm{\,d} x) 
    \propto 
     e^
     {-V(x)} 
    \mathrm{\,d} x,
$ 
the Schr\"odinger-F\"ollmer diffusion evolving on the unit interval $[0, 1]$ is given by: 
\begin{align}
\label{SFS-W2-eq:stochastic-differential-equation}
    \mathrm{\,d} X_t
    =
    ~\nabla \log 
    \big( 
    Q_{1-t} g(X_t) \big) 
    \mathrm{\,d} t
    +
    \mathrm{\,d} W_t, 
    \quad X_0=0, 
    \quad t \in ( 0, 1 ],
\end{align}
where $g$ denotes the Radon-Nikodym derivative of $\mu$ with respect to the $d$-dimensional standard Gaussian distribution
$\gamma^d$,
given by \eqref{SFS-W2-eq:Radon-Nikodym-derivative}
and $\{Q_t\}_{t \in [0, 1]}$ is the heat semigroup defined by \eqref{SFS-W2-eq:semi-group-expression}.
%
%
As shown below, the diffusion process \eqref{SFS-W2-eq:stochastic-differential-equation} transports the degenerate distribution $\delta_0$ at $t=0$ to the target distribution 
$
\mu(\mathrm{\,d} x) 
    \propto 
     e^
     {-V(x)} 
    \mathrm{\,d} x
$
at $t=1$. 
For the Gaussian mixture distributions, the drift function can be explicitly expressed (see \eqref{SFS-W2-eq:drift-GMD}-\eqref{eq:gaussian-mixture-b-form}). In general, the drift coefficients for many other distributions lack a closed form expression and a practical way to handle the expectation in \eqref{SFS-W2-eq:semi-group-expression} is to use the well-known Monte Carlo approximation. 

Based on the Euler discretizations of the Schr\"odinger-F\"ollmer diffusion \eqref{SFS-W2-eq:stochastic-differential-equation}, the authors of \cite{MRjiao} recently proposed two Schr\"odinger-F\"ollmer samplers (SFS).
%
%
%
%
%
%
When the exact drift coefficient of \eqref{SFS-W2-eq:stochastic-differential-equation} is used, a convergence rate of order $\mathcal{O}(\sqrt{h})$ in the $L^2$-Wasserstein distance was obtained there for the Euler discretization with uniform step-size $h>0$. 
When the drift term $f$ lacks a closed form expression, the authors used inexact drift based on Monte Carlo approximations and derived an error bound of order $\mathcal{O}(\sqrt{h}) + \mathcal{O}(\tfrac{1}{\sqrt{M}})$ in the $L^2$-Wasserstein distance, where $M$ is the number of samples used in the Monte Carlo approximation.
In this case, the sampling error of SFS comes not only from the time discretization, but also the Monte Carlo approximation of the drift term. 
The convergence rate of order $\frac{1}{2}$ due to time discretization is naturally expected, as the time-dependent drift is only supposed to be $\frac{1}{2}$-H\"older continuous with respect to time (see \cite{MRjiao}).
%
Unlike usual Langevin samplers as discussed above, SFS is gradient-free, works in a unit interval $[0, 1]$ and does not require any (weak) convexity condition to ensure ergodicity. Numerical experiments indicate that, the SFS outperforms vanilla Langevin samplers, particularly in sampling from multimodal distributions.

By introducing a temperature parameter, 
in this work we propose a variant of Schr\"odinger-F\"ollmer diffusion with temperatures as follows:
\begin{align}
\label{SFS-W2-eq:intro-sf-beta}
    \mathrm{\,d} X_t
    =
    \,
    \beta\,\nabla \log 
    \big(
    Q^\beta_{1-t} g_{\beta}(X_t)
    \big)
    \mathrm{\,d} t
    +
    \sqrt{\beta}
    \mathrm{\,d} W_t, 
    \quad X_0=0, 
    \quad t \in ( 0, 1 ],
\end{align}
where $g_{\beta}$ and $Q^\beta_{t}, \beta \in (0, \infty)$ are determined by \eqref{SFS-W2-eq:Radon-Nikodym-derivative-g} and \eqref{SFS-W2-eq:semi-group}, respectively.
Such a diffusion process \eqref{SFS-W2-eq:intro-sf-beta} with temperatures also transports the degenerate distribution $\delta_0$ at $t=0$ to the target distribution 
$
\mu(\mathrm{\,d} x) 
    \propto 
     e^
     {-V(x)} 
    \mathrm{\,d} x
$
at $t=1$. To sample from the target distribution, we introduce two schemes \eqref{SFS-W2-eq:Euler-scheme-no-mc} and  \eqref{SFS-W2-eq:Euler-scheme-mc} for the time discretization of SDE \eqref{SFS-W2-eq:intro-sf-beta}, producing approximations $Y_{1}$ and $\widetilde{Y}^M_{1}$ to $X_1 \sim \mu$, respectively.
%
Further, we analyze the resulting approximation error in the $L^2$-Wasserstein distance.
By carefully handling the singularity of the time derivative of the drift, we propose a novel approach of error analysis to obtain an enhanced convergence rate of order $\mathcal{O}(h)$ for the time discretization. More precisely, these findings can be summarized as follows:
\begin{itemize}
    \item 
    In the case that the drift can be exactly calculated, we introduce a new sampler with temperatures:
    \begin{align}
    Y_{t_{n+1}}
    =
    Y_{t_n}
    +
    h f_{\beta}
    (Y_{t_n}, t_n)
    +
    \sqrt{\beta}
    \Delta W_{n},
    \quad Y_0=0,
    \end{align} 
where $f_{\beta}(x,t):=
    \beta~
    \nabla \log \big( Q^{\beta}_{1-t} g_{\beta}(x) \big),
    \,x \in \mathbb{R}^d$. For this sampler, the following error bound in the $L^2$-Wasserstein distance is established:
    \begin{align*}
        \mathcal{W}_2
        \left(
        \mathcal{L}aw
        \left(Y_{1}\right), \mu\right)
        \leq 
        { 
        C d h}, 
    \end{align*}
    {where the constant $C$ is independent of $d$ and $h$, depending only on $L_g,\rho$ and $\beta$.} 
    \item  
    {  When the drift can not be calculated exactly, we propose another new sampler with an inexact drift and temperatures:
    \begin{align}
    \label{eq:intro-SFS-inexact}
    \widetilde{Y}^M_{t_{n+1}}
   =
   \widetilde{Y}^M_{t_n}
   +
   h \widetilde{f}^M_{\beta}
   \big(
   \widetilde{Y}^M_{t_n}, t_n
   \big)
   +
   \sqrt{\beta}
   \Delta W_{n}, 
   \quad \widetilde{Y}^M_0=0,
    \end{align} 
    where the inexact drift $\widetilde{f}^M_{\beta}$ due to {the} Monte Carlo approximation is given by \eqref{SFS-W2-eq:f-M-experssion}.}
    For this sampler, the following error bound in the $L^2$-Wasserstein distance is obtained:
    \begin{align*}
        \mathcal{W}_2
        \big(
        \mathcal{L}aw
        \big(
        \widetilde{Y}^M_{1}
        \big), \mu
        \big)
        \leq 
        { 
        C d h}
        +
        C
        \sqrt{\tfrac{d}{M}},
    \end{align*}
 where $M$ is the number of samples used in the Monte Carlo
 estimator of the drift, and the constant $C$ is independent of $d,h$ and $M$, depending only on $L_g,\rho$ and $\beta$. Moreover, the error analysis {does} not rely on a strong convexity condition, as required by \cite{MRjiao} (see Condition (C4) in \cite{MRjiao}).
%
%
\item Numerical results show that the SFS substantially outperforms vanilla Langevin samplers in sampling from multimodal distributions. Also, an additional interesting finding is revealed: raising the temperature always makes a significant difference in sampling from multimodal distributions. {\color{black}The intuition behind this is that the Markov chains at higher temperatures can cross energy barriers more easily. 
}
%
\end{itemize}

To conclude, the main contributions of this work are two-fold:
\begin{itemize}

\item[(1)] 
We introduce a new Schr\"odinger-F\"ollmer process with temperatures and accordingly construct two new Schr\"odinger-F\"ollmer samplers with flexible temperatures $\beta$, different from those in \cite{MRjiao}, where $\beta = 1$. In addition, $M$ independent Gaussian random variables {need} to be generated  at each iteration for the sampler \eqref{eq:intro-SFS-inexact} with inexact drift in \cite{MRjiao} ($\beta = 1$). Instead, the newly proposed samplers incorporate flexible temperatures and $M$ independent Gaussian random variables {are} only generated once and repeatedly used  at each iteration for the sampler  \eqref{eq:intro-SFS-inexact} with an inexact drift due to Monte Carlo approximation;

\item[(2)] 
By introducing a new error analysis without relying on any convexity condition, we provide enhanced error bounds with a convergence rate of order $\mathcal{O}(h)$, significantly improving relevant error bounds of order $\mathcal{O}(\sqrt{h})$ in \cite{MRjiao}.
\end{itemize}

The rest of this article is organized as follows. The next section introduces the Schr\"odinger-F\"ollmer diffusion with temperatures.
In Section \ref{SFS-W2-section:main_results}, we present assumptions and main results of this paper, whose proofs are given in Section \ref{SFS-W2-section:Proof of Theorem}.
%
%
%
Some numerical tests are presented to illustrate our theoretical findings in Section \ref{SFS-W2-section:Numerical-experiment}. 
Finally, some concluding remarks are given in the last section.

\section{Schr\"odinger-F\"ollmer diffusion with temperatures}
\label{SFS-W2-section:background_SFD}
\subsection{Notation}
Throughout this paper, we use $\mathbb{N}$ to denote the set of all positive integers and let $d\in \mathbb{N}$. 
For any $m \in \mathbb{N}$, we denote $[m]:= \{1,...,m\}$ and $[m]_0:= \{0, 1,...,m\}$.
Let $\|\cdot\|$ and $\langle\cdot, \cdot\rangle$ denote the Euclidean norm and the inner product of vectors in $\mathbb{R}^d$, respectively. 
We use $\mathbf{1}_d \in \mathbb{R}^d$ and $\mathbf{I}_d \in \mathbb{R}^{d\times d}$ to denote the { all-ones} vector (where all entries are $1$) and the identity matrix,  respectively.
{For any matrix {$A=(a_{i,j})\in \mathbb{R}^{d\times d}$,} the $\mathrm{F}$-norm is defined as
{$\|A\|_{\mathrm{F}}
:=
\sqrt{\sum_{i,j=1}^d |a_{i,j}|^2}$} and the operator norm as $\|A\| =\|A\|_{\mathrm{op}} :=\sup_{\|v\|=1}\|Av\|$. It is not difficult to see
$
\|A\|
\leq
\|A\|_{\mathrm{F}}
\leq 
\sqrt{d}
\,
\|A\|.
$}
%

%
Let $\left\{W_t\right\}_{t \in [0, 1] }$ be a standard $d$-dimensional Brownian motion process, defined on a filtered probability space $
\big(
\Omega_W, 
\mathcal{F}^W,
\mathbb{P}_W, 
\{\mathcal{F}^W_t\}_{t \in[0,1]}
\big)
$ 
satisfying the usual conditions.
Let $\{\xi_{i}\}_{ i \in \mathbb{N}}$ be an i.i.d. family of standard Gaussian distributed random variables, independent of $\left\{W_t\right\}_{t \in [0, 1]}$,  defined on an additional probability space 
$
\left(
\Omega_\xi, 
\mathcal{F}^\xi,
\mathbb{P}_\xi\right)$.
%
By $\mathbb{E}_W$ and $\mathbb{E}_\xi$, we denote expectations in these two probability spaces 
$
\big(
\Omega_W, 
\mathcal{F}^W,
\mathbb{P}_W
\big)
$
and
$
\left(
\Omega_\xi, 
\mathcal{F}^\xi,
\mathbb{P}_\xi\right),
$
respectively.
We introduce the following product probability space
{ \begin{align*}
    (\Omega,\mathcal{F},\mathbb{P},\mathcal{F}_t)
    :=
    (\Omega_W\otimes\Omega_{\xi}, \mathcal{F}^W
    \otimes
    \mathcal{F}^{\xi},
    \mathbb{P}_W
    \otimes\mathbb{P}_{\xi},
    \mathcal{F}^W_t 
    \otimes
    \mathcal{F}^{\xi}_t),
\end{align*}}
where \(\mathcal{F}^\xi_t\) is the \(\sigma\)-algebra generated by \(\{\xi_i\}_{i\in  \mathbb{N}}\), which is independent of $t$ (i.e., \(\mathcal{F}^\xi_t = \mathcal{F}^\xi\) for all \(t\)).
%
%
We use $\mathbb{E} $ {for} the expectation in the product probability space and $L^r\left(\Omega, \mathbb{R}^d\right), r \geq 1$, to denote the family of $\mathbb{R}^d$-valued random variables $\eta$ satisfying $\mathbb{E} \left[\|\eta\|^r\right] := \mathbb{E}_W \big( \mathbb{E}_\xi \left[\|\eta\|^r\right] \big) <\infty$. 
Let $\mathcal{B}\left(\mathbb{R}^d\right)$ be the Borel $\sigma$-field of $\mathbb{R}^d$ and $\mathcal{P}\left(\mathbb{R}^d\right)$ be the space of all probability distributions on $\left(\mathbb{R}^d, \mathcal{B}(\mathbb{R}^d)\right)$.
{ By $\mathcal{L}aw(X)$ we denote the probability distribution of the $\mathbb{R}^d$-valued random variable $X$.}
Let $\nu_1$ and $\nu_2$ be two probability measures defined on $\left(\mathbb{R}^d, \mathcal{B}(\mathbb{R}^d)\right)$, 
and let $\mathcal{D}\left(\nu_1, \nu_2\right)$ represent the collection of couplings $\nu$ on $\left(\mathbb{R}^{2 d}, \mathcal{B}(\mathbb{R}^{2 d})\right)$ whose first and second marginal distributions are $\nu_1$ and $\nu_2$, respectively.
The $L^2$-Wasserstein distance between $\nu_1$ and $\nu_2$ is defined by
\begin{equation}
\begin{aligned}
\label{SFS-W2-eq:def-w2}
\mathcal{W}_2
(\nu_1, \nu_2)
& :=
\inf_{\nu \in \mathcal{D}
(\nu_1, \nu_2)}
\left(
\int_{\mathbb{R}^d} \int_{\mathbb{R}^d}
\|x-y\|^2 
\mathrm{\, d} 
\,
\nu
( x, y )
\right)^{1/2}\\
& =
\inf \Big\{
\big(
\mathbb{\,E}
\big[
\|X-Y\|^2
\big]
\big)^{1/2},
\quad \mathcal{L}aw(X)=\nu_1,
\quad \mathcal{L}aw(Y)=\nu_2
\Big\}.
\end{aligned}
\end{equation}
{As a direct consequence, the $L^2$-Wasserstein distance of two random variables can be bounded by their $L^2$-distance.}
Let $\mathcal{C}\left(\mathbb{R}^d, \mathbb{R}\right)$ denote the set of all continuous functions from $\mathbb{R}^d$ to $\mathbb{R}$ and let $\mathcal{C}_b\left(\mathbb{R}^d, \mathbb{R}\right)$ denote the set of all bounded continuous functions from $\mathbb{R}^d$ to $\mathbb{R}$. 
For $k \geq 0$, denote by $\mathcal{C}^k\left(\mathbb{R}^d, \mathbb{R}\right)$ the set of functions from $\mathbb{R}^d$ to $\mathbb{R}$ which {have} continuous $0$-th, $\ldots, k$-th order derivatives, and further denote by $\mathcal{C}_b^k\left(\mathbb{R}^d, \mathbb{R}\right)$ the set of functions from $\mathbb{R}^d$ to $\mathbb{R}$ which {have} bounded continuous $0$-th, $\ldots, k$-th order derivatives.
For $y \in \mathcal{C}^3(\mathbb{R}^d, \mathbb{R})$ and $v_1, v_2, v_3, x \in \mathbb{R}^d$, we denote
\begin{equation}
\begin{aligned}
\label{SFS-W2-eq:def-dir-der}
        \nabla_{v_1} y(x) 
        & =
        \lim_{\varepsilon \rightarrow 0} \frac{y(x+\varepsilon v_1)
        -y(x)}
        {\varepsilon}, 
        \\
        \nabla_{v_2} 
        \nabla_{v_1} y(x) 
        & =
        \lim_{\varepsilon \rightarrow 0} \frac{\nabla_{v_1} y\left(x+\varepsilon v_2\right)
        -
        \nabla_{v_1} y(x)}{\varepsilon},\\
        \nabla_{v_3} 
        \nabla_{v_2} 
        \nabla_{v_1} y(x)
        & =
        \lim_{\varepsilon \rightarrow 0} \frac{\nabla_{v_2} \nabla_{v_1} y\left(x+\varepsilon v_3\right)
        -\nabla_{v_2} 
        \nabla_{v_1} y(x)}{\varepsilon},
\end{aligned}
\end{equation}
as the directional derivatives of $y$. 
If the function $y$ is differentiable at the point 
$x \in \mathbb{R}^d$, then the directional derivative exists along any nonzero vector $v\in \mathbb{R}^d$. 
In this case, we have
\begin{align*}
    \nabla_{v}y(x)=
    \left \langle
    \nabla y(x), v
    \right\rangle.
\end{align*}
One knows 
$\nabla y(x) \in \mathbb{R}^d, \nabla^2 y(x) \in \mathbb{R}^{d \times d}, 
\nabla^3 y(x) \in$ $\mathbb{R}^{d \times d \times d}$. 
Moreover, we define the operator norm of $\nabla^k y(x), k=1,2,3$ by
\begin{align}
\label{SFS-W2-eq:operator-norm-def}
    \big\|
    \nabla^k y(x)
    \big\|
    =
    \big\|
    \nabla^k y(x)
    \big\|_{\mathrm{op}}
    :=
    \sup_{\|v_i\|=1, i=1, \ldots, k}
    \left\|
    \nabla_{v_k} 
    \ldots 
    \nabla_{v_1} y(x)
    \right\|.
\end{align} 
For the vector-valued function $\mathbf{u}: \mathbb{R}^d \rightarrow \mathbb{R}^{\ell}, \mathbf{u}=\left(u_{(1)}, \ldots, u_{(\ell)}\right)^{\mathrm{T}}$, we regard its first order partial derivative as the Jacobian matrix:
\begin{align*}
    D \mathbf{u}=\left(\begin{array}{ccc}
\frac{\partial u_{(1)}}{\partial x_1} & \cdots & \frac{\partial u_{(1)}}{\partial x_d} \\
\vdots & \ddots & \vdots \\
\frac{\partial u_{(\ell)}}{\partial x_1} & \cdots & \frac{\partial u_{(\ell)}}{\partial x_d}
\end{array}\right)_{\ell \times d} .
\end{align*}
For any $v_1 \in \mathbb{R}^d$, one knows $D(\mathbf{u}) v_1 \in \mathbb{R}^{\ell}$ and one can define $D^2 \mathbf{u}\left(v_1, v_2\right)$ as
\begin{align*}
    D^2 \mathbf{u}\left(v_1, v_2\right):=D\left(D(\mathbf{u}) v_1\right) v_2, \quad \forall v_1, v_2 \in \mathbb{R}^d .
\end{align*}

Given the Banach spaces $\mathcal{X}$ and $\mathcal{Y}$, we denote by $L(\mathcal{X}, \mathcal{Y})$ the Banach space of bounded linear operators from $\mathcal{X}$ into $\mathcal{Y}$. 
Then the partial derivatives of the function $\mathbf{u}$ can be also regarded as the following operators:
\begin{align*}
    \begin{gathered}
         D \mathbf{u}(\cdot)(\cdot): \mathbb{R}^d \rightarrow L\left(\mathbb{R}^d, \mathbb{R}^{\ell}\right), \\
         D^2 \mathbf{u}(\cdot)(\cdot, \cdot): \mathbb{R}^d \rightarrow L\left(\mathbb{R}^d, L\left(\mathbb{R}^d, \mathbb{R}^{\ell}\right)\right) \cong L\left(\mathbb{R}^d \otimes \mathbb{R}^d, \mathbb{R}^{\ell}\right).
    \end{gathered}
\end{align*}

\subsection{The Schr\"odinger-F\"ollmer process (SFP) with temperatures}
In this subsection, we revisit the Schr\"odinger-F\"ollmer diffusion process in the literature.
We first make the following assumption.
\begin{assumption}
\label{SFS-W2-ass:absolutely-continuous-distribution}
    Let the target distribution 
    $
    \mu( \mathrm{\,d}x)
     \propto
    \exp \big(-V(x)\big)
    \mathrm{\,d}x, 
    \, x \in \mathbb{R}^d,
    $ 
    be absolutely continuous with respect to the $d$-dimensional 
    Gaussian distribution denoted by $\mathcal{N}(0,\beta\,\mathbf{I}_d)$.
\end{assumption}
Let $g$ denote the Radon-Nikodym derivative of $\mu$ with respect to the $d$-dimensional standard Gaussian distribution $\gamma^d$,
{i.e., }the ratio of the density of $\mu$ over the density of $\gamma^d$, i.e.,
\begin{align}
\label{SFS-W2-eq:Radon-Nikodym-derivative}
    g(x)
    :=
    \frac{\mathrm{d} \mu}
    {\mathrm{d} \gamma^d}(x)
    =
    \tfrac{(2\pi)^{d/2}}{C}
    \exp 
    \Big(
    -V(x)+\tfrac{\|x\|^2}{2}
    \Big), 
    \quad x \in \mathbb{R}^d,
\end{align}
and let $\{Q_t\}_{t \in [0, 1]}$ be the heat semigroup defined by
\begin{align}
\label{SFS-W2-eq:semi-group-expression}
    Q_t g(x)
    :=
    \mathbb{E}_{\xi}
    \big[
    g(x+\sqrt{t} \,\xi)
    \big],
    \quad  t \in[0,1],
    \quad  \xi \sim \gamma^d.
\end{align}
Also, recall the Schr\"odinger-F\"ollmer process $\{X_t\}_{t\in[0,1]}$ is defined as \cite{follmer1988random,follmer2005entropy}:
\begin{align}
\label{SFS-W2-eq:sde}
    \mathrm{\,d} X_t
    =
    ~\nabla \log \big( Q_{1-t} g(X_t) \big) 
    \mathrm{\,d} t
    +
    \mathrm{\,d} W_t, 
    \quad X_0=0, 
    \quad t \in ( 0, 1 ].
\end{align}
As shown by \cite{MRjiao}, such a diffusion process \eqref{SFS-W2-eq:sde} transports the degenerate distribution $\delta_0$ at $t=0$ to the target distribution 
$
\mu(\mathrm{\,d} x) 
    \propto 
     e^
     {-V(x)} 
    \mathrm{\,d} x
$
at $t=1$.
%
%
In this work, we introduce a variant of Schr\"odinger-F\"ollmer process with temperatures as follows:
\begin{align}
    \label{SFS-W2-eq:f-sde}
     \mathrm{\,d} X_t
     =
     f_{\beta}(X_t,t)
     \mathrm{\,d} t
     +
     \sqrt{\beta}
     \mathrm{\,d} W_t,
     \quad  X_0=0, 
     \quad t \in(0,1],
\end{align}
where $\beta \in (0, \infty)$ is a temperature parameter
and the drift $f_{\beta}$ is given by
\begin{align}
\label{SFS-W2-eq:f-expression}
    f_{\beta}(x, t) 
    :=
    \beta~
    \nabla \log \big( Q^{\beta}_{1-t} g_{\beta}(x) \big)
    ,
    \quad  x \in \mathbb{R}^d, \: t \in[0,1],
\end{align}
with $g_{\beta}$ being the Radon-Nikodym derivative of $\mu$ with respect to $\mathcal{N}(0,\beta\,\mathbf{I}_d)$:
\begin{align}
\label{SFS-W2-eq:Radon-Nikodym-derivative-g}
    g_{\beta}(x)
    :=
    \frac{\mathrm{d}\mu}
    {\mathrm{d} \mathcal{N} (0,\beta\,\mathbf{I}_d)}(x)
    =
    \tfrac{(2\pi\beta)^{d/2}}{C}
    \exp 
    \big(
    -V(x)
    +
    \tfrac{\|x\|^2}{2\beta}
    \big), 
    \quad x \in \mathbb{R}^d,
\end{align}
and $\{Q_t^{\beta}\}_{t \in [0, 1]}$ being the heat semigroup defined by
\begin{align}
\label{SFS-W2-eq:semi-group}
    Q_t^{\beta} g_{\beta}(x)
    :=
    \mathbb{E}_{\xi}
    \big[
     g_{\beta}(x+\sqrt{t\beta} \,\xi)
    \big],
    \quad  t \in[0,1],
    \quad  \xi \sim \gamma^d.
\end{align}
%
In the special case $\beta
=1$, the Schr\"odinger-F\"ollmer process \eqref{SFS-W2-eq:f-sde} with temperatures reduces into the usual Schr\"odinger-F\"ollmer process \eqref{SFS-W2-eq:sde} in the literature.
%
%
To ensure that the $\operatorname{SDE}$ \eqref{SFS-W2-eq:f-sde} admits a unique strong solution, we make the following assumption.
\begin{assumption}
    \label{SFS-W2-ass:g-nabla-g-lip}
    Suppose that the Radon-Nikodym derivative $g_{\beta}$ and $\nabla g_{\beta}$ are $L_g$-Lipschitz continuous,
    and $g_{\beta}$ is uniformly bounded below by a positive constant $\rho$:
    \begin{align}
    \label{SFS-W2-eq:g-lower-bounded}
        g_{\beta} \geq \rho >0.
    \end{align}
\end{assumption}

{ 
Noting that $g_{\beta}$ is a differentiable function satisfying  
$\mathbb{E}_{\xi}
\big[
\nabla g_{\beta}(x+\sqrt{(1-t)\beta}\,\xi)
\big]< +\infty$ (see Appendix \ref{SFS-W2-appendix:g-satisfy-condition} for details),}
one can apply Stein's lemma { \cite[Lemma 3.6.5]{garthwaite2002statistical}} to obtain,
\begin{align*}
    \mathbb{E}_{\xi}
    \big[
    \nabla g_{\beta}(x+\sqrt{(1-t)\beta} \,\xi)
    \big]
    =
    \tfrac{1}{\sqrt{(1-t)\beta}} 
    \mathbb{E}_{\xi}
    \big[
    \xi g_{\beta}\big(x+\sqrt{(1-t)\beta} \,\xi \big)
    \big],
    \quad
    t \in [0, 1).
\end{align*}
This enables us to avoid the calculation of $\nabla g_{\beta}$ and 
{ thus we derive}
\begin{align}
\label{SFS-W2-eq:f-no-mc-no-nabla}
    f_{\beta}(x, t)
    & =
    \frac{\beta
    \mathbb{\,E}_{\xi}
    [\nabla g_{\beta}(x+\sqrt{(1-t)\beta} \,\xi)]}{
    \mathbb{E}_{\xi}
    [g_{\beta}(x +\sqrt{(1-t)\beta} \,\xi)]}
    \nonumber\\
    & =
    \frac{
    \beta\mathbb{\,E}_{\xi}
    \big[
    \xi \,g_{\beta}(x+\sqrt{(1-t)\beta} \,\xi)
    \big]}
    {\mathbb{E}_{\xi}
    \big[
    g_{\beta}(x+\sqrt{(1-t)\beta} \,\xi)
    \big]
    \cdot \sqrt{(1-t)\beta}},
    \quad \xi \sim \gamma^d,
    \:
    t \in [0, 1).
\end{align}
{Under the above} assumptions,
we obtain the following well-posedness of the  Schr\"odinger-F\"ollmer diffusion \eqref{SFS-W2-eq:f-sde} (see \cite{LJ2013,MRjiao}).
\begin{prop}
    \label{SFS-W2-prop-unique-strong-solution}
     Let Assumptions \ref{SFS-W2-ass:absolutely-continuous-distribution}, \ref{SFS-W2-ass:g-nabla-g-lip} hold. Then for any $t \in[0,1]$, the drift coefficient $f_{\beta}(\cdot,t):\mathbb{R}^d \rightarrow \mathbb{R}^d$ of SDE \eqref{SFS-W2-eq:f-sde} is Lipschitz continuous and { satisfies a linear growth condition.} That is, there exist constants $L_f,\hat{L}_f > 0$, independent of $d$, such that, for any $x,y\in \mathbb{R}^d$ and $ t \in [0,1]$,
    \begin{align}
    \label{SFS-W2-eq:f-lipschitz-condition}
        \|f_{\beta}(x,t)-f_{\beta}(y,t)\|
        \leq
        L_f\|x-y\|,
    \end{align}
    and
    {
    \begin{align}
    \label{SFS-W2-eq:f-linear-growth}
        \|f_{\beta}(x,t)\|
        & \leq
        \|f_{\beta}(0,t)\|
        +
        L_f\|x\|
        \nonumber\\
        & \leq 
        \hat{L}_f
        +L_f\|x\|,
    \end{align}}
    where $L_f := \big(1+\tfrac{L_g}{\rho}
    \big)
    \tfrac{\beta L_g}{\rho},
    \,
    {\hat{L}_f := \tfrac{\beta L_g}{\rho}}$,
    and $L_g, \rho$ come from Assumption \ref{SFS-W2-ass:g-nabla-g-lip}.
    Then the Schr\"odinger-F\"ollmer diffusion \eqref{SFS-W2-eq:f-sde} has a unique strong solution $\left\{X_t\right\}_{t \in[0,1]}$ satisfying $X_1 \sim \mu$.
\end{prop}
%
Proposition \ref{SFS-W2-prop-unique-strong-solution} shows that,  
%
the Schr\"olinger-F\"ollmer process \eqref{SFS-W2-eq:f-sde} with  temperatures also transports the degenerate distribution $\delta_0$ at $t=0$ to the target distribution 
$
\mu(\mathrm{\,d} x) 
    \propto 
     e^
     {-V(x)} 
    \mathrm{\,d} x
$
at $t=1$.
One can prove it by following the same lines of \cite[Appendix B]{MRjiao}.
%
{In \cite{MRjiao}, the authors assume that the time-dependent drift satisfies
\begin{align}
\label{SFS-W2-eq:b-lipschitz-in-x-houder-in-t}
     \|f_{\beta}(x,t)-f_{\beta}(y,s)\|
     \leq 
     \bar{L}_f(\|x-y\|
     +
     d^\frac{1}{2}
     |t-s|^\frac{1}{2}).
\end{align}
This can be proved under Assumptions \ref{SFS-W2-ass:absolutely-continuous-distribution}, \ref{SFS-W2-ass:g-nabla-g-lip} (see Remark III.1 and Appendix D of \cite{MRjiao}).
The assumption \eqref{SFS-W2-eq:b-lipschitz-in-x-houder-in-t} implies the time-dependent drift is $\frac{1}{2}$-H\"older continuous with respect to the time variable.}
From a theoretical point of view, the introduction of the temperature parameter $\beta$ does not make any obvious difference. Nevertheless, as revealed by numerical results, raising the temperature is crucial and makes a significant difference in sampling from high-dimensional multimodal distributions (see subsection \ref{subsection:Gaussian-mixture} for details). 
%

Similar to \cite[Lemma A.3, Lemma A.4]{MRjiao}, the process $\left\{X_t\right\}_{ t \in [0, 1]}$ defined by \eqref{SFS-W2-eq:f-sde} has uniform moments and $\frac{1}{2}$-H\"older continuity.
\begin{lem}
\label{SFS-W2-lem:x_t-moments-bounded}
   Let Assumptions \ref{SFS-W2-ass:absolutely-continuous-distribution} and \ref{SFS-W2-ass:g-nabla-g-lip} hold. Then the solution of SDE  \eqref{SFS-W2-eq:f-sde}, denoted by $\{X_t\}_{t\in [0,1]}$ {satisfies}, for any $ t \in [0,1]$ and $0 \leq t_1 \leq t_2 \leq 1$,
   \begin{align}
        \label{SFS-W2-eq:x_t-moments-bounded}
        \mathbb{E}
        { _W}
        \left[
        \left\|
        X_t
        \right\|^2
        \right] 
        \leq 
        { M_1} d,
    \end{align}
    and 
    \begin{align}
    \label{SFS-W2-eq:x-1/2-houlder-continuous}
         \mathbb{E}{ _W}
         \left[
         \left\|
         X_{t_2}-X_{t_1}
         \right\|^2
         \right] 
         \leq 
         { M_2}
         \,
         d(t_2-t_1),
     \end{align}
where 
{ 
$
M_1 := 2(2\hat{L}^2_f + \beta) \exp(4 L^2_f)
$}
and
{ 
$
M_2 := 8 L^2_f \exp(4 L^2_f) (2\hat{L}^2_f + \beta) + 4 \hat{L}^2_f + 2\beta,
$}
with constants $L_f$ and $\hat{L}_f$ given by \eqref{SFS-W2-eq:f-lipschitz-condition} and \eqref{SFS-W2-eq:f-linear-growth}, respectively.
\end{lem}

\section{Schr\"odinger-F\"ollmer samplers with temperatures}
\label{SFS-W2-section:main_results}

\subsection{Schr\"odinger-F\"ollmer sampler with exact drift}
%

%
%
To sample from the target distribution, one just needs to solve SDE \eqref{SFS-W2-eq:f-sde} to attain $X_1 \sim \mu$. A practical way is to  discretize the continuous-time process and obtain the approximate solutions.
To this end, we use the Euler-Maruyama discretization of \eqref{SFS-W2-eq:f-sde} with a fixed step-size. 
Let $ \pi_h $ be a temporal grid of the form 
$$
\pi_h:=\{t_n=n \cdot h, 
\quad n\in [N]_0, \text { with } h=1 / N\}.
$$
Then the Euler-Maruyama discretization of \eqref{SFS-W2-eq:f-sde} is given by
\begin{align}
    \label{SFS-W2-eq:Euler-scheme-no-mc}
    Y_{t_{n+1}}
    =
    Y_{t_n}
    +
    h f_{\beta}(Y_{t_n}, t_n)
    +
    \sqrt{\beta}
    \Delta W_{n},
    \quad n\in [N-1]_0,
    \quad Y_0=0,
\end{align}
where 
$\Delta W_{n} 
:= 
W_{t_{n+1}}
-
W_{t_n}, 
\,n \in [N-1]_0, \,N \in \mathbb{N}$.

In \cite{MRjiao}, the authors obtained a convergence rate of order $\mathcal{O}(\sqrt{h})$ for the Euler discretization of the Schr\"odinger-F\"ollmer diffusion \eqref{SFS-W2-eq:stochastic-differential-equation}.
%
The order-half convergence is naturally expected, due to the fact that the time-dependent drift is only $\frac{1}{2}$-H\"older continuity with respect to the time variable. 
In the present paper, we properly handle the singularity of the time derivative of the drift and improve the convergence rate of order $\mathcal{O}(\sqrt{h})$ to order { $\mathcal{O}(h)$.} The price to pay is to put slightly stronger smoothness conditions on the drift coefficient with respect to the state variable. 
\begin{assumption}
    \label{SFS-W2-ass:f-in-C^2}
    (Smoothness assumptions)
    For any $t \in [0,1]$, the drift coefficient $f_{\beta}(\cdot,t):\mathbb{R}^d
    \rightarrow \mathbb{R}^d$ of SDE \eqref{SFS-W2-eq:f-sde} is twice continuously differentiable with bounded partial derivatives. Namely, there exist a constant
    $L'_f$, independent of $d,t$, {and depending only on $L_g, \rho,\beta$,} such that, for any $ x,v_1,v_2 \in \mathbb{R}^d$ and $t \in [0,1]$,
    \begin{align}
    \label{SFS-W2-eq:f-c-2}
             \|
             D^2f_{\beta}(x,t)(v_1,v_2)\| 
             \leq 
        { L'_f}
             \|v_1\|
             \cdot
             \|v_2\|.
    \end{align}
    Moreover, for any $x \in \mathbb{R}^d$, the drift coefficient $f_{\beta}(x,\cdot): [0,1) \rightarrow \mathbb{R}^d$ of SDE \eqref{SFS-W2-eq:f-sde} is continuously differentiable and there exists a constant $\widetilde{L}_f$, independent of $d,t$, {and depending only on $L_g, \rho,\beta$,} such that, for any $t \in [0,1)$,
    \begin{align}
    \label{SFS-W2-eq:f-t-1-2}
        \big\|
            \partial_t f_{\beta}(x, t)
            \big\|
            \leq 
            \widetilde{L}_f
            \,
            d^\frac{1}{2}
            \frac{1}
            {\sqrt{1-t}}.
    \end{align}
\end{assumption}
%

%
%
From \eqref{SFS-W2-eq:f-t-1-2}, one can evidently see the singularity of $\partial_tf_{\beta}(x,t)$ at the terminal time $t=1$, which would {pose} a difficulty in enhancing the convergence rate in the error analysis. 
{We particularly mention \cite{Kusuoka2001Approximation,lyons2004cubature}, where the authors provided a non-equidistant discretization to handle the singularity. 
In this paper, we rely on a new approach to overcome the singularity of $\partial_t f_{\beta}$ successfully, even with an equidistant discretization (see subsection \ref{sfs-w2-subsec:proof-thm} for details).}

The next proposition shows when Assumption \ref{SFS-W2-ass:f-in-C^2} is satisfied.

%
\begin{prop}
\label{SFS-W2-prop:g-satisfy-condition}
    For $\beta \in (0, \infty)$, we assume $g_{\beta}$ 
    $\in \mathcal{C}^3(\mathbb{R}^d,\mathbb{R})$ and there exist constants $L_g,\rho > 0$ such that
    $g_{\beta},\nabla g_{\beta},\nabla^2 g_{\beta}$ are $L_g$-Lipschitz continuous and \eqref{SFS-W2-eq:g-lower-bounded} holds.
    { Then Assumption \ref{SFS-W2-ass:f-in-C^2} is satisfied with
$
L'_f := \big( 1 + \tfrac{3L_g}{\rho} + \tfrac{2 L_g^2}{\rho^2} \big) \tfrac{L_g \beta}{\rho}$ and $
\widetilde{L}_f := \big( 1 + \tfrac{L_g}{\rho} \big) \tfrac{L_g \beta^{3/2}}{2\rho}.
$}
{ Moreover, we have
\begin{equation} \label{SFS-W2-prop:Df-bound}
    \|Df_{\beta}(x,t)v_1\|
    \leq 
    L_f 
    \|v_1\|,
    \quad
    \forall\, x, v_1 \in \mathbb{R}^d, \, ~ t \in [0,1],
\end{equation}
where $L_f := \big(1+\tfrac{L_g}{\rho}
    \big)
    \tfrac{\beta L_g}{\rho}$}.
\end{prop}
The proof of Proposition \ref{SFS-W2-prop:g-satisfy-condition} is straightforward and postponed to Appendix \ref{SFS-W2-appendix:g-satisfy-condition}. {In what follows, we provide sufficient conditions on the potential $V$ of the target distribution, to ensure the smoothness assumptions required in Proposition \ref{SFS-W2-prop:g-satisfy-condition}. 

\begin{prop}
\label{sfs-w2-prop:ass-v}
Let $V \in \mathcal{C}^3(\mathbb{R}^d,\mathbb{R})$ and let the following inequalities hold:
\begin{equation}
\label{sff-w2-eq:v-condition}
    \begin{aligned}
        \sup_{x
        \in\mathbb{R}^d}
        \big\|
        \beta^{-1}
        x
        -
        \nabla V(x)
        \big\|^i
        \cdot
        \exp
        \big(
        -V(x)
        +
        \tfrac{\|x\|^2}
        {2\beta}
        \big)
        & < 
        \infty,
        \:
        i \in \{ 1,2,3 \},
        \\
        \sup_{x
        \in\mathbb{R}^d}
        \|
        \beta^{-1}
        \mathbf{I}_d
        -
        \nabla^2 V(x)
        \|
        \cdot
        \exp
        \big(
        -V(x)
        +
        \tfrac{\|x\|^2}
        {2\beta}
        \big)
        & < 
        \infty,\\
        \sup_{x
        \in\mathbb{R}^d}
        \|
        \nabla^3 V(x)
        \|
        \cdot
        \exp
        \big(
        -V(x)
        +
        \tfrac{\|x\|^2}
        {2\beta}
        \big)
        & < 
        \infty,\\
\sup_{x\in\mathbb{R}^d}
\|
\beta^{-1}
\mathbf{I}_d
-
\nabla^2 V(x)
\|
\cdot
\|
\beta^{-1}
x
-
\nabla V(x)
\|
\cdot
\exp
\big(
-V(x)
+
\tfrac{\|x\|^2}
{2\beta}
\big)
& < 
\infty.
\end{aligned}
\end{equation}
Then the functions $g_{\beta},\nabla g_{\beta}$ and $\nabla^2 g_{\beta}$ for any 
$\beta \in (0,\infty )$ are $L_g$-Lipschitz continuous.
\end{prop}}

{ The proof of Proposition \ref{sfs-w2-prop:ass-v} is put in Appendix \ref{sfs-w2-app:ass-v}. It is not difficult to see, for the Gaussian mixture distribution \eqref{SFS-W2-eq:def-Gaussian-mixture-distribution} being the target distribution, $g_{\beta},\nabla g_{\beta} \text{ and }\nabla^2 g_{\beta}$ are Lipschitz continuous on the condition that $
\Sigma_i^{-1}
-
\beta^{-1}
\mathbf{I}_d, \, i=1,{\ldots},\kappa,$ 
are positive {matrices} (see also \cite{MRjiao} for the case $\beta = 1$).}
%

{ 
However, to check the Lipschitzness of the drift of SDEs for sampling is not an easy task and a similar problem exists in Langevin type samplers \cite{Chen2023Improved}. To guarantee the Lipschitzness of the drift $f_\beta$, we additionally put the lower bound assumption \eqref{SFS-W2-eq:g-lower-bounded} on $g_{\beta}$, which seems restrictive but was commonly used in the literature  \cite{LJ2013,MRjiao,Ruzayqat2023Unbiased,tzen2019theoretical}.
To remedy it, one can regularize $\mu$ by mixing it with $\mathcal{N}(0,\beta\,\mathbf{I}_d)$, similarly to \cite{MRjiao}.
Specifically, for $\rho \in (0, 1)$, we introduce a modified target distribution $\mu_{\rho} := (1-\rho)\mu + \rho \,\mathcal{N}(0,\beta \,\mathbf{I}_d)$,
whose corresponding density ratio is thus given by $g_{\beta,\rho} 
= 
\mathrm{d}
\mu_{\rho} / 
\mathrm{d} 
\mathcal{N}(0,\beta\,\mathbf{I}_d) 
= 
(1-\rho)g_{\beta} 
+ \rho$.
It follows directly that $g_{\beta,\rho} \geq \rho > 0$, i.e., \eqref{SFS-W2-eq:g-lower-bounded} holds true. 
Furthermore,  the Lipschitz continuity of $g_{\beta}$ and $\nabla g_{\beta}$ ensures that $g_{\beta,\rho}$ and $\nabla g_{\beta,\rho}$ are also Lipschitz continuous.}
Thanks to the above assumptions, 
we are able to derive an enhanced non-asymptotic error bound in the $L^2$-Wasserstein distance for the sampling algorithm \eqref{SFS-W2-eq:Euler-scheme-no-mc}.
\begin{thm}
\label{SFS-W2-thm:main-rerult-no-mc}
     (Main result: enhanced error bounds with exact drift) 
     Let $N \in\mathbb{N}$ and let Assumptions \ref{SFS-W2-ass:absolutely-continuous-distribution}, \ref{SFS-W2-ass:g-nabla-g-lip}, \ref{SFS-W2-ass:f-in-C^2} hold. 
     Let $\left\{Y_{t_n}\right\}_{n \in [N]_0}$ be produced by the sampler \eqref{SFS-W2-eq:Euler-scheme-no-mc} using the uniform step-size $h=\tfrac{1}{N}$. 
     Then there exists a constant $C$ independent of $d$ and $h$, such that,
    \begin{align*}
        \mathcal{W}_2
        \left(
        \mathcal{L}aw
        \left(Y_{1}\right), \mu\right)
        \leq 
        { 
        C d h}.
    \end{align*}
\end{thm}
%
Its proof is deferred to Section \ref{SFS-W2-section:Proof of Theorem}. As a direct consequence, we obtain the following result, 
whose proof is easy and thus omitted.
\begin{prop}
\label{SFS-W2-prop:mixing-time-no-mcmc}
    Let Assumptions \ref{SFS-W2-ass:absolutely-continuous-distribution}, \ref{SFS-W2-ass:g-nabla-g-lip} and \ref{SFS-W2-ass:f-in-C^2} hold. 
    To achieve a given precision level $\epsilon>0$ in the $L^2$-Wasserstein distance, a required number of { evaluations of the drift $f_{\beta}$ and the number of Brownian increments} of the SFS algorithm \eqref{SFS-W2-eq:Euler-scheme-no-mc} is of order { $\mathcal{O}\left(\frac{d}{\epsilon}\right)$.}
\end{prop}
Compared to \cite[Theorem III.1]{MRjiao}, our refined error analysis guarantees an enhanced convergence rate of order one with respect to the time step-size $h>0$, but at the cost of an increasing dimension dependence that scales linearly with $d$. This also happened in the non-asymptotic error analysis of LMC (see, e.g., \cite{DA2017,DAKA2019,DAME2019,mou2022improved}). 
%
%
{In Table \ref{tab:comparison-1}, we compare our results with those in \cite{MRjiao}, 
in terms of the error bounds, the total number of evaluations
of the drift $f_{\beta}$ and the number of Brownian increments
required to achieve a given precision level $\epsilon>0$ in the $L^2$-Wasserstein distance.
Evidently, the computational cost of the new sampler is significantly reduced.}
%
%
\begin{table}[h]
\caption{A comparison of non-asymptotic error bounds in the $L^2$-Wasserstein distance for SFS \eqref{SFS-W2-eq:Euler-scheme-no-mc}.}
\vspace{-2pt} 
\centering
\renewcommand{\arraystretch}{1.2} 
\small 
\begin{tabular}{ccccc}
\Xhline{1pt}  
 & \makecell[c]{Additional\\ condition\textsuperscript{1}}
 & \makecell[c]{Error\\ bound} 
 & \makecell[c]{ {Number of}\\ 
 { evaluations of drift} }
 & \makecell[c]{ {Number of}
 \\
 { Brownian increments} }
 \\ \hline
\cite{MRjiao} 
& No 
& $\mathcal{O}(\sqrt{dh})$ 
& $\mathcal{O}(d/\epsilon^{2})$
& { $\mathcal{O}(d/\epsilon^{2})$}
\\ 
This work 
& Yes 
& { 
$\mathcal{O}(dh)$} 
& { 
$\mathcal{O}(d/\epsilon)$}
& { 
$\mathcal{O}(d/\epsilon)$}
\\ 
\Xhline{1pt} 
\end{tabular}
\label{tab:comparison-1}
\caption*{\textsuperscript{1}
Smoothness assumptions other
than the Lipschitz condition for the drift.}
\end{table}

Next, we give examples when the drift has a known closed-form and can be exactly calculated.
\newline

\noindent
\textbf{Examples with exact drift: Gaussian mixture distributions.}
\newline

As validated in \cite{MRjiao}, in the case of Gaussian mixture distributions \cite{LLMW2021},
\begin{align}
\label{SFS-W2-eq:def-Gaussian-mixture-distribution}
        \mu
        =
        \sum_{i=1}^\kappa \theta_i 
        \,\mathcal{N}(\alpha_i, \Sigma_i),
        \quad \sum_{i=1}^\kappa \theta_i=1 
        \text { and } 
        0 \leq \theta_i \leq 1, 
        \quad i=1, \ldots, \kappa,
\end{align}
the drift term of the Schr\"odinger-F\"ollmer diffusion \eqref{SFS-W2-eq:f-sde} is given by
\begin{align}
\label{SFS-W2-eq:drift-GMD}
    f_{\beta}(x, t)
    =
    \frac{\beta \sum_{i=1}^\kappa \theta_i 
    \mathbb{E}_{\xi} 
    \left[
    \nabla g_{\beta,i}(x+\sqrt{(1-t)\beta} \,\xi)
    \right]}
    {\sum_{i=1}^\kappa \theta_i \mathbb{E}_{\xi} 
    \left[ g_{\beta,i}(x+\sqrt{(1-t)\beta} \,\xi)
    \right]}, 
    \quad \xi \sim \gamma^d.
\end{align}
Here $\kappa \in \mathbb{N}$ is the number of mixture components, $\mathcal{N}(\alpha_i, \Sigma_i)$ is the $i$-th Gaussian component with mean $\alpha_i \in \mathbb{R}^d$ and covariance matrix $\Sigma_i \in \mathbb{R}^{d \times d}$.
{ 
Following an argument similar as that in \cite[Appendix C]{MRjiao},} we obtain
\begin{align*}
         & 
         \mathbb{\,E}_{\xi}
         \Big[
         \nabla g_{\beta,i}(x+\sqrt{(1-t)\beta} \,\xi)
         \Big]\\
         & \quad =
         \tfrac{\Sigma_i^{-1} \alpha_i+\left(\tfrac{1}{\beta}\mathbf{I}_d-\Sigma_i^{-1}\right)
         \left[
         t \mathbf{I}_d+(1-t)\beta \,\Sigma_i^{-1}\right]^{-1}\left[(1-t)\beta \Sigma_i^{-1} \alpha_i+x\right]}
         {\left\|t \Sigma_i+(1-t)\beta \,\mathbf{I}_d\right\|^{1 / 2}} b_{\beta,i}(x, t), \\
         & 
         \mathbb{E}_{\xi}
         \left[
         g_{\beta,i}(x+\sqrt{(1-t)\beta} \,\xi)
         \right]
         =
         \tfrac{b_{\beta,i}(x, t)}
         {\left\|t \Sigma_i+(1-t) \beta\mathbf{I}_d\right\|^{1 / 2}},
\end{align*}
where
\begin{align} \label{eq:gaussian-mixture-b-form}
        b_{\beta,i}(x, t)
        = & 
        \exp \left(\tfrac{1}{2(1-t)\beta}
        \Big\|
        \big(t \mathbf{I}_d+(1-t)\beta \,\Sigma_i^{-1}\big)^{-1 / 2}
        \big((1-t)\beta\, \Sigma_i^{-1} \alpha_i+x\big)
        \Big\|^2
        \right)\nonumber \\
        & \cdot \exp \left(-\tfrac{1}{2} \alpha_i^T \Sigma_i^{-1} \alpha_i-\tfrac{1}{2(1-t)\beta}\|x\|^2\right).
\end{align}
{ Inserting these expressions into \eqref{SFS-W2-eq:drift-GMD} gives an explicit representation of the drift $f_{\beta}$.}

\subsection{Schr\"odinger-F\"ollmer sampler with inexact drift}

Unfortunately, { analytically calculating} the drift term $f_{\beta}$ is, in general, intractable when the target distribution $\mu$ is complex.
In this case, one can just get {an} estimator $\widetilde{f}^M_{\beta}$ of the drift $f_{\beta}$, by replacing $\mathbb{E}_{\xi}$ in $f_{\beta}$ with sample mean.
For $M\in \mathbb{N}$, let $\{\xi_{j}\}_{j \in [M]}$ be a family of independent standard Gaussian distributed random variables on the probability space 
$
\left(
\Omega_\xi, 
\mathcal{F}^\xi,
\mathbb{P}_\xi\right)$.
%
The random variables $\{\xi_{j}\}_{j \in [M]}$, independent of the randomness already presented in SDE \eqref{SFS-W2-eq:f-sde}, can be regarded as added random inputs for the approximation of the drift term.
{More precisely}, we approximate $f_{\beta}(x,t_n)$ by $\widetilde{f}^M_{\beta}: \Omega_{\xi} \times \mathbb{R}^d \times \pi_h \rightarrow \mathbb{R}^d$:
{ 
\begin{align}
\label{SFS-W2-eq:f-M-experssion}
\widetilde{f}^M_{\beta}
    (x,t_n)
    & =
    \frac{\frac{\beta }{M} 
    \sum_{j=1}^M
    \left[
    \nabla 
    g_{\beta}
    (x+\sqrt{(1-t_n)\beta} 
    \,
    {\xi}_{j})
    \right]}
    {\frac{1}{M} 
    \sum_{j=1}^M
    \left[
    g_{\beta}
    (x+\sqrt{(1-t_n)\beta} \,{\xi}_{j})
    \right]}
    \nonumber\\
    & =
    \frac{\frac{\beta }{M} 
    \sum_{j=1}^M
    \left[
    \xi_j
    \,
    g_{\beta}(x+\sqrt{(1-t_n)\beta} \,{\xi}_{j})
    \right]}
    {\frac{1}{M} 
    \sum_{j=1}^M
    \left[
    g_{\beta}
    (x+\sqrt{(1-t_n)\beta} \,{\xi}_{j})
    \right]
    \sqrt{(1-t)\beta}}, 
    \quad \xi_{j} \sim \gamma^d,
    \:
    n \in [N-1]_0,
\end{align}
where the second equality stands due to Stein's lemma.}
In this case, we propose another new sampler with temperatures as follows:
\begin{align}
\label{SFS-W2-eq:Euler-scheme-mc}
   \widetilde{Y}^M_{t_{n+1}}
   =
   \widetilde{Y}^M_{t_n}
   +
   h \widetilde{f}^M_{\beta}
   \big(
   \widetilde{Y}^M_{t_n}, t_n
   \big)
   +
   \sqrt{\beta}
   \Delta W_{n}, 
   \quad n\in [N-1]_0,
   \quad \widetilde{Y}^M_0=0,
\end{align}
where the inexact drift $\widetilde{f}^M_{\beta}$ is given by \eqref{SFS-W2-eq:f-M-experssion}. 
%

For the SFS with inexact drift in \cite{MRjiao}, the temperature was taken to be unit (i.e., $\beta = 1$) and $M$ independent Gaussian random variables need to be generated  at each iteration. Instead, the newly proposed sampler \eqref{SFS-W2-eq:Euler-scheme-mc} incorporates a flexible temperature $\beta$ and $M$ independent Gaussian random variables are only generated once and repeatedly used  at each iteration for the sampler  \eqref{SFS-W2-eq:Euler-scheme-mc}.
{ 
We would like to mention that, instead of a single-level Monte Carlo method used here, 
%
the multi-level technique introduced by
Heinrich \cite{heinrich2001multilevel} and Giles \cite{Giles2008Multilevel} might be helpful for the approximation of the drift $f_{\beta}$. This idea is promising but non-trivial,
which we leave as a possible future work.}
In the following proposition, we show that the inexact drift $\widetilde{f}^M_{\beta}$ is also Lipschitz continuous.
\begin{prop}
    Let Assumptions \ref{SFS-W2-ass:absolutely-continuous-distribution} and  \ref{SFS-W2-ass:g-nabla-g-lip} hold. Then 
    for any $n\in [N-1]_0$, the estimator $\widetilde{f}^M_{\beta} \colon
    (\cdot,t_n)
    \rightarrow \mathbb{R}^d$ is uniformly Lipschitz continuous.
    Namely, we have
    \begin{align}
    \label{sfs-w2-eq:til-f-lip}
        \big\|
        \widetilde{f}^M
        _{\beta}(x,t_n)
        -
        \widetilde{f}^M
        _{\beta}(y,t_n)
        \big\|
        \leq
        L_f\|x-y\|,
        \quad
        \forall \, x,y \in \mathbb{R}^d,
    \end{align} 
    { where $L_f := \big(1+\tfrac{L_g}{\rho}
    \big)
    \tfrac{\beta L_g}{\rho}$ is independent of $d$ and $h$.}
\end{prop}
Now
we can present an enhanced non-asymptotic error bound for the SFS algorithm \eqref{SFS-W2-eq:Euler-scheme-mc}.

\begin{thm}
\label{SFS-W2-thm:main-rerult-mc}
    (Main result: enhanced error bounds with inexact drift)
    Let Assumptions \ref{SFS-W2-ass:absolutely-continuous-distribution}, \ref{SFS-W2-ass:g-nabla-g-lip} and \ref{SFS-W2-ass:f-in-C^2} hold.
    Let $M,N \in\mathbb{N}$ and let $\big\{\widetilde{Y}^M_{t_n}\big\}_{n \in [N]_0}$ be produced by the sampler \eqref{SFS-W2-eq:Euler-scheme-mc} with the uniform step-size $h=\tfrac{1}{N}
    $. Then there exists a constant $C$ independent of $d,h$ and $M$, such that,
    \begin{align*}
        \mathcal{W}_2
        \big(
        \mathcal{L}aw
        (
        \widetilde{Y}^M
        _{1}
        ), 
        \mu
        \big)
        \leq 
        { 
        C d h
        +
        C
        \sqrt{\tfrac{d}{M}},}
    \end{align*}
    where $M$ is the number of samples used in the Monte Carlo
    estimator \eqref{SFS-W2-eq:f-M-experssion}.
\end{thm}
Its proof is deferred to Section \ref{SFS-W2-section:Proof of Theorem}.
For the sampler \eqref{SFS-W2-eq:Euler-scheme-mc}, we obtain an enhanced convergence rate of order { $\mathcal{O}(h)$,} significantly improving the convergence rate of order $\mathcal{O}(\sqrt{h})$ in \cite[Theorem III.2]{MRjiao}.
Moreover, the error analysis here does not rely on a strong convexity condition, as required by \cite{MRjiao} (see Condition (C4) in \cite{MRjiao}).
As a direct consequence, we obtain the following result complexity concerning the required number of evaluations of $g_{\beta}$, the number of Brownian increments and the number of samples used in the Monte Carlo estimator, whose proof is postponed to Appendix \ref{SFS-W2-appendix:mixing-time}.

\begin{prop}
\label{SFS-W2-prop:mixing-time}
    Let Assumptions \ref{SFS-W2-ass:absolutely-continuous-distribution}, \ref{SFS-W2-ass:g-nabla-g-lip} and \ref{SFS-W2-ass:f-in-C^2} hold. 
    To achieve a given precision level $\epsilon>0$ in the $L^2$-Wasserstein distance, a required number of {  Brownian increments} of the SFS algorithm \eqref{SFS-W2-eq:Euler-scheme-mc} is of order { $\mathcal{O}\left(\frac{d}{\epsilon}\right)$,
    the number of evaluations of $g_{\beta}$ is of order $\mathcal{O}\left(\frac{d^2}{\epsilon^3}\right)$} and the number of {samples of $\xi$} used in the Monte Carlo estimator \eqref{SFS-W2-eq:f-M-experssion} is of order
    { $\mathcal{O}
    \left(\frac{d}{\epsilon^2}\right)$}.
\end{prop}
According to \cite[Theorem III.2]{MRjiao}, a required number of iterations of the SFS algorithm \eqref{SFS-W2-eq:Euler-scheme-mc} is of order $\mathcal{O}\left(\frac{d}{\epsilon^2}\right)$, and the number of samples used in the Monte Carlo estimator \eqref{SFS-W2-eq:f-M-experssion} is of order
    $\mathcal{O}
    \left(\frac{d^2}{\epsilon^4}\right)$, in order to attain a given precision level $\epsilon>0$ in the  $L^2$-Wasserstein distance.
{We mention that in the SFS with an inexact drift proposed by \cite{MRjiao}, $M$ independent Gaussian random variables are updated at each time step. Consequently, the cost in the Monte Carlo estimator is of order {$\mathcal{O}(h^{-1}Md)$, where $d$ appears here as the Gaussians are $d$-dimensional.}
By contrast, $M$ independent Gaussian random variables are only generated once and repeatedly used at each iteration for the newly proposed samplers \eqref{SFS-W2-eq:Euler-scheme-mc} with an inexact drift in our paper.
Thus, the cost of Monte Carlo in our work is reduced to an order of {$\mathcal{O}(Md)$}.}

In Table \ref{tab:comparison}, we compare our results with those in \cite{MRjiao}, in terms of the error bounds, the number of Brownian increments, the number of evaluations of $g_{\beta}$ and the number of samples used in Monte Carlo estimator of
the SFS \eqref{SFS-W2-eq:Euler-scheme-mc} required to achieve a given precision level $\epsilon>0$ in the $L^2$-Wasserstein distance.
Evidently, our results significantly improve those in \cite[Theorem III.2]{MRjiao}.

\begin{table}[h]
\caption{A comparison of non-asymptotic error bounds in the $L^2$-Wasserstein distance for SFS \eqref{SFS-W2-eq:Euler-scheme-mc}.}
\vspace{-2pt} 
\centering
\renewcommand{\arraystretch}{1.5} 
\footnotesize 
\begin{tabular}{ccccccc}
\Xhline{1pt} 
 & \makecell[c]{Strong\\ convexity} 
 & \makecell[c]{Additional\\ condition\textsuperscript{1}}
 & \makecell[c]{Error\\ bound} 
 & \makecell[c]{{ Number of} \\{  Brownian }\\{ increments}}
 & \makecell[c]{{ Number of}
 \\{ MC samples}}
 & \makecell[c]{{ Number of} \\ { evaluations of  $g_{\beta}$} }\\ \hline
\cite{MRjiao} 
& Yes
& No 
& $\mathcal{O}(\sqrt{dh}+\sqrt{d/M})$
& { 
$\mathcal{O}(d/\epsilon^{2})$}
& { $\mathcal{O}(d^2/\epsilon^{4})$ }
& { $\mathcal{O}(d^2/\epsilon^{4})$ }
\\ 
This work 
& No 
& Yes 
& { $\mathcal{O}(dh+\sqrt{d/M})$ }
& { $\mathcal{O}(d/\epsilon)$}
& { $\mathcal{O}(d/\epsilon^{2})$}
& { $\mathcal{O}(d^2/\epsilon^{3})$}\\ 
\Xhline{1pt} 
\end{tabular}
\label{tab:comparison}
\caption*{\textsuperscript{1}
Smoothness assumptions other
than the Lipschitz condition for the drift.}
\end{table}

In the following, we also present a comparison between our samplers and LMC in Table \ref{tab:comparison-lmc}, in terms of the smoothness assumptions, the number of Brownian increments,
the number of evaluations of functions and the number of samples used in Monte Carlo estimator required to attain a tolerance level $\epsilon>0$ in the $L^2$-Wasserstein distance.
%
\begin{table}[h]
\caption{A comparison of computational cost between SFS and LMC in the $L^2$-Wasserstein distance.}
\vspace{-2pt} 
\centering
\renewcommand{\arraystretch}{1.5} 
\small
\begin{tabular}{ccccc}
\Xhline{1pt} 
& \makecell[c]{Convexity\\ condition}
& \makecell[c]{Number of \\ Brownian increments}
& \makecell[c]{Number of \\ MC samples}
& \makecell[c]{Number of \\ evaluations of functions}
\\ \hline
\cite{DA2017}
& SC
& $\mathcal{O}
\big(
\tfrac{d}
{\epsilon^2}
\log(\tfrac{\sqrt{d}}{\epsilon})
\big)$
& No
& $\mathcal{O}
\big(
\tfrac{d}
{\epsilon^2}
\log(\tfrac{\sqrt{d}}{\epsilon})
\big)$
\\ 
\cite{mou2022improved}
& LSI
& $\mathcal{O}
\big(
\tfrac{d}
{\epsilon}
\log(\tfrac{\sqrt{d}}{\epsilon})
\big)$
& No
& $\mathcal{O}
\big(
\tfrac{d}
{\epsilon}
\log(\tfrac{\sqrt{d}}{\epsilon})
\big)$
\\ 
SFS \eqref{SFS-W2-eq:Euler-scheme-no-mc} in this work
& No 
& $\mathcal{O}(\tfrac{d}{\epsilon})$
& No
& $\mathcal{O}(\tfrac{d}{\epsilon})$
\\ 
SFS \eqref{SFS-W2-eq:Euler-scheme-mc} in this work
& No 
& $\mathcal{O}(\tfrac{d}{\epsilon})$
& $\mathcal{O}(\tfrac{d}{\epsilon^2})$
& $\mathcal{O}(\tfrac{d^2}{\epsilon^3})$
\\
\Xhline{1pt} 
\end{tabular}
\caption*{
\footnotesize
{ 
SC: strong convexity \quad
LSI: log-Sobolev inequality}}
\label{tab:comparison-lmc}
\end{table}
{  As shown in Table \ref{tab:comparison-lmc}, the LMC seems cheaper than SFS \eqref{SFS-W2-eq:Euler-scheme-mc} with inexact drift in terms of computational costs.  But when the target distribution exhibits complex structures, notably high dimensions and multiple modes, 
the sampled points from LMC are prone to becoming trapped in a single mode and struggle to escape from it (see subsection \ref{subsection:Gaussian-mixture} in details). 
Numerical results show that the SFS substantially outperforms vanilla Langevin samplers in multimodal sampling.}

\section{Proof of Theorem \ref{SFS-W2-thm:main-rerult-no-mc} and Theorem \ref{SFS-W2-thm:main-rerult-mc}}
\label{SFS-W2-section:Proof of Theorem}
In this section, we aim to prove Theorem \ref{SFS-W2-thm:main-rerult-no-mc} and Theorem \ref{SFS-W2-thm:main-rerult-mc}.
%
%
For convenience, we unify the two proposed samplers \eqref{SFS-W2-eq:Euler-scheme-no-mc} and \eqref{SFS-W2-eq:Euler-scheme-mc} in the following form
\begin{align}
\label{SFS-W2-eq:unify-e-m}
   \hat{Y}_{t_{n+1}}
   =
   \hat{Y}_{t_n}
   +
   h 
   \hat{f}_{\beta}
   \big(
   \hat{Y}_{t_n}, t_n
   \big)
   +
   \sqrt{\beta}
   \Delta W_{n}, 
   \quad n\in [N-1]_0,
   \quad \hat{Y}_0=0,
\end{align}
where 
$\Delta W_{n} 
:= 
W_{t_{n+1}}
-
W_{t_n}.$
{  
In the case when the drift can be computed exactly, we take $\hat{Y}=Y$ and $\hat{f}_{\beta}=f_{\beta}$, reducing \eqref{SFS-W2-eq:unify-e-m} to the SFS algorithm \eqref{SFS-W2-eq:Euler-scheme-no-mc}.
In the other case that $f_{\beta}$ lacks a closed form expression, we set $\hat{Y}=\widetilde{Y}^M$,\,$\hat{f}_{\beta}=\widetilde{f}^M_{\beta}$, 
and \eqref{SFS-W2-eq:unify-e-m} becomes the SFS algorithm \eqref{SFS-W2-eq:Euler-scheme-mc} with an inexact drift.}
Next we prove the following result, describing {   the error propagation} of the scheme \eqref{SFS-W2-eq:unify-e-m} for the Schr\"odinger-F\"ollmer diffusion \eqref{SFS-W2-eq:f-sde}.

\begin{prop}
\label{SFS-W2-prop:Upper-mean-square error-bounds}
     Let Assumptions \ref{SFS-W2-ass:absolutely-continuous-distribution} and \ref{SFS-W2-ass:g-nabla-g-lip} hold.  Let $\left\{X_t\right\}_{t \in [0,1]}$ and $\big\{\hat{Y}_{t_n}\big\}_{n \in [N]_0}$ be the solutions of SDE \eqref{SFS-W2-eq:f-sde} and the SFS algorithm \eqref{SFS-W2-eq:unify-e-m}, respectively.
     { 
     Then
\begin{align*}
\mathbb{E}
\Big[
\big\|
X_{t_n}
-
\hat{Y}_{t_n}
\big\|^2
\Big] 
\leq
2 L_f^2 h
\sum_{i=1}^{n-1}
\mathbb{E}
\Big[
\big\|
X_{t_i}
-
\hat{Y}_{t_i}
\big\|^2
\Big]
+
2 \,
\mathbb{E}
\bigg[
\Big\|
\sum_{i=1}^{n}
\mathcal{R}_i
\Big\|^2
\bigg],
\end{align*}}
where, for $i \in [n-1], n \in [N]$, we denote
\begin{align}
\label{SFS-W2-eq:defination-R_n}
         \mathcal{R}_i:
         = 
         \int_{t_{i-1}}^{t_i}
         \big[
         f_{\beta}(X_s,s)
         -
         \hat{f}_{\beta}(X_{t_{i-1}},t_{i-1})
         \big] 
         \mathrm{\,d} s.
\end{align}
\end{prop}
\begin{proof}
\textbf{Proof.}
By \eqref{SFS-W2-eq:f-sde}, one can write
\begin{align*}
    X_{t_n}
    =
    X_{t_{n-1}}
    +
    \hat{f}_{\beta}(X_{t_{n-1}},t_{n-1}) h
    + 
    \sqrt{\beta}
    \Delta W_{n-1}
    +
    \mathcal{R}_n,
\end{align*}
where $\mathcal{R}_n$ is defined by \eqref{SFS-W2-eq:defination-R_n}.
{ 
Subtracting \eqref{SFS-W2-eq:unify-e-m} from this yields
\begin{align}
\label{SFS-W2-eq:x_n-y_n-1}
        X_{t_n}
        -
        \hat{Y}_{t_n}
        & =  
        X_{t_{n-1}}
        -
        \hat{Y}_{t_{n-1}}
        + 
        h
        \big(
        \hat{f}_{\beta}
        (X_{t_{n-1}},t_{n-1})
        -
        \hat{f}_{\beta}
        (\hat{Y}_{t_{n-1}},t_{n-1})
        \big)
        +
        \mathcal{R}_n
        \nonumber\\
        & =
        h
        \sum_{i=0}^{n-1}
        \big(
        \hat{f}_{\beta}
        (X_{t_{i}},t_{i})
        -
        \hat{f}_{\beta}
        (\hat{Y}_{t_{i}},
        t_{i})
        \big)
        +
        \sum_{i=1}^{n}
        \mathcal{R}_i,
\end{align}
where the fact was also used that 
$X_{0}
=
\hat{Y}_{0}=0$.
Squaring both sides of \eqref{SFS-W2-eq:x_n-y_n-1}, taking expectations and using the Lipschitz conditions \eqref{SFS-W2-eq:f-lipschitz-condition} and \eqref{sfs-w2-eq:til-f-lip},
we obtain
\begin{align}
\mathbb{E}
\Big[
\big\|
X_{t_n}
-
\hat{Y}_{t_n}
\big\|^2
\Big]
& =
\mathbb{E}
\Big[
\Big\|
h
\sum_{i=0}^{n-1}
\big(
\hat{f}_{\beta}
(X_{t_{i}},t_{i})
-
\hat{f}_{\beta}
(\hat{Y}_{t_{i}},t_{i})
\big)
+
\sum_{i=1}^{n}
\mathcal{R}_i
\Big\|^2
\Big]
\nonumber
\\
& 
\leq
2 h^2
\,
\mathbb{E}
\Big[
\Big\|
\sum_{i=0}^{n-1}
\big(
\hat{f}_{\beta}
(X_{t_{i}},t_{i})
-
\hat{f}_{\beta}
(\hat{Y}_{t_{i}},t_{i})
\big)
\Big\|^2
\Big]
+
2 \,
\mathbb{E}
\bigg[
\Big\|
\sum_{i=1}^{n}
\mathcal{R}_i
\Big\|^2
\bigg]
\nonumber\\
& \leq
2 L_f^2 h
\sum_{i=1}^{n-1}
\mathbb{E}
\Big[
\big\|
X_{t_i}
-
\hat{Y}_{t_i}
\big\|^2
\Big]
+
2 \,
\mathbb{E}
\bigg[
\Big\|
\sum_{i=1}^{n}
\mathcal{R}_i
\Big\|^2
\bigg].
\end{align}
The proof is completed.}
\end{proof}

\subsection{Proof of Theorem \ref{SFS-W2-thm:main-rerult-no-mc}}
\label{sfs-w2-subsec:proof-thm}
%
\begin{proof}
Bearing Proposition \ref{SFS-W2-prop:Upper-mean-square error-bounds} in mind, one just needs to properly handle the error term
$
\mathbb{E}
\Big[
\big\|
\sum_{i=1}^n
\mathcal{R}_i
\big\|^2
\Big],
$ 
where  
we have
$
\hat{Y}=Y \text{ and } \hat{f}_{\beta}=f_{\beta}
$
for the sampler \eqref{SFS-W2-eq:Euler-scheme-no-mc}.
By means of the It\^o formula \cite[Theorem 4.2.1]{oksendal2013stochastic} we infer, {for 
$ s \in [t_{i-1}, t_i)$,}
\begin{equation}
\begin{aligned}
\label{SFS-W2-eq:f-Taylor-expansion-1}
        f_{\beta}(X_s,s)
        -
        f_{\beta}(X_{t_{i-1}},t_{i-1})
        & =
        \int^s_{t_{i-1}}
        \partial_r 
        f_{\beta}(X_r,r)
        \mathrm{\,d} r 
        +
        \int^s_{t_{i-1}} 
        D f_{\beta}(X_r,r) 
        f_{\beta}(X_r,r) \mathrm{\,d} r
        \\
        & \quad +
        \sqrt{\beta}
        \int^s_{t_{i-1}}   
        D f_{\beta}(X_r,r) 
        \mathrm{\,d} W_r
        +
        \frac{\beta}{2} 
        \sum_{j=1}^d
        \int^s_{t_{i-1}}
        D^2 f_{\beta}(X_r,r)
        (e_j,e_j)
        \mathrm{\,d} r,
\end{aligned}
\end{equation}
where $\{e_j\}_{j \in\{1, \cdots, d\}}$ is denoted as the orthonormal basis of $\mathbb{R}^d$.
Therefore,
\begin{align}
\mathbb{E}
\bigg[
\Big\|
\sum_{i=1}^{n}
\mathcal{R}_i
\Big\|^2
\bigg]
& =
\mathbb{E}
\bigg[
\Big\|
\sum_{i=1}^n
\int_{t_{i-1}}^{t_i}
\big[
f_{\beta}(X_s,s)
-
f_{\beta}(X_{t_{i-1}},t_{i-1})
\big] 
\mathrm{\,d} s
\Big\|^2
\bigg]
\nonumber\\
& \leq
\underbrace{
4
\,
\mathbb{E}
\bigg[
\Big\|
\sum_{i=1}^n
\int_{t_{i-1}}^{{t_{i}}}
\int_{t_{i-1}}^{s}
\partial_r 
f_{\beta}(X_r,r)
\mathrm{\,d} r 
\mathrm{\,d} s
\Big\|^2
\bigg]}_{=:T_1}
\nonumber\\
& \quad +
\underbrace{
4 \,
\mathbb{E}
\bigg[
\Big\|
\sum_{i=1}^n
\int_{t_{i-1}}^{{t_{i}}}
\int_{t_{i-1}}^{s}
D f_{\beta}(X_r,r) f_{\beta}(X_r,r) 
\mathrm{\,d} r 
\mathrm{\,d} s
\Big\|^2
\bigg]}_{=:T_2}
\nonumber\\
& \quad +
\underbrace{
4 \beta \,
\mathbb{E}
\bigg[
\Big\|
\sum_{i=1}^n
\int_{t_{i-1}}^{{t_{i}}}
\int_{t_{i-1}}^{s}
D f_{\beta}(X_r,r) 
\mathrm{\,d} W_r
\mathrm{\,d} s
\Big\|^2
\bigg]}_{=:T_3}
\nonumber\\
& \quad +
\underbrace{
 \beta^2 \,
\mathbb{E}
\bigg[
\Big\|
\sum_{i=1}^n
\sum_{j=1}^d
\int_{t_{i-1}}^{{t_{i}}}
\int_{t_{i-1}}^{s}
D^2 f_{\beta}(X_r,r)
(e_j,e_j)
\mathrm{\,d} r
\mathrm{\,d} s
\Big\|^2
\bigg]}_{=:T_4}.
\end{align}
For 
$T_{1}$, we use the Minkowski inequality, the H\"older inequality and Assumption \ref{SFS-W2-ass:f-in-C^2} to deduce
\begin{equation}
\label{sfs-w2-eq:est-j11}
\begin{aligned}
T_{1}
& =
4 \,
\Big\|
\sum_{i=1}^n
\int_{t_{i-1}}^{{t_{i}}}
\int_{t_{i-1}}^{s}
\partial_r 
f_{\beta}
(X_r,r)
\mathrm{\,d} r 
\mathrm{\,d} s
\Big\|^2_{L^2(\Omega;\mathbb{R}^d)} \\
& \leq
4
\bigg(
\sum_{i=1}^n
\int_{t_{i-1}}^{{t_{i}}}
\int_{t_{i-1}}^{s}
\big\|
\partial_r 
f_{\beta}
(X_r,r)
\big\|_{L^2(\Omega;\mathbb{R}^d)}
\mathrm{\,d} r 
\mathrm{\,d} s
\bigg)^2\\
& \leq
4
\bigg(
\widetilde{L}_f
\,
d^\frac{1}{2}
\sum_{i=1}^n
\int_{t_{i-1}}^{{t_{i}}}
\int_{t_{i-1}}^{s}
\frac{1}
{\sqrt{1-r}}
\mathrm{\,d} r 
\mathrm{\,d} s
\bigg)^2\\
& \leq
4
\bigg(
\widetilde{L}_f
\,
d^\frac{1}{2}
h
\sum_{i=1}^n
\int_{t_{i-1}}^{{t_{i}}}
\frac{1}
{\sqrt{1-s}} 
\mathrm{\,d} s
\bigg)^2\\
& =
16
\widetilde{L}^2_f\,
dh^2,
\end{aligned}
\end{equation}
where we noted
\begin{align*}
    \sum_{i=1}^{n}
    \int_{t_{i-1}}^{t_i}
    \frac{1}
    {\sqrt{1-s}} 
    \mathrm{\,d} s
    =
    \int_{0}^{1}
    \frac{1}
    {\sqrt{1-s}} 
    \mathrm{\,d} s
    =
    \int_0^1
    \frac{1}{\sqrt{t}}
    \mathrm{\,d} t
    =
    2.
\end{align*}
With regard to $T_2$, by using the inequality $(\sum_{i=1}^n u_i )^2 \leq n \sum_{i=1}^n u_i^2$, the H\"older inequality, Assumption \ref{SFS-W2-ass:f-in-C^2} and Lemma \ref{SFS-W2-lem:x_t-moments-bounded}, we obtain
\begin{equation}
\label{sfs-w2-eq:est-j12}
\begin{aligned}
T_2
& \leq
4
nh^2 
\sum_{i=1}^n
\mathbb{E}
\bigg[
\int_{t_{i-1}}^{{t_{i}}}
\int_{t_{i-1}}^{s}
\big\|
D f_{\beta}(X_r,r) f_{\beta}(X_r,r)
\big\|^2
\mathrm{\,d} r 
\mathrm{\,d} s
\bigg]\\
& \leq
4 h^2
\sup _{r\in[0, 1]} 
\mathbb{E}
\big[
\|
Df_{\beta}(X_r,r) 
f_{\beta}(X_r,r)\|^2
\big]
\\
& \leq 
4 L_f^2 h^2 
\left\|
f_{\beta}(X_r,r)
\right\|^2_{L^2(\Omega ; \mathbb{R}^d)}
\\
& \leq
8 L_f^2 h^2
\Big(
\hat{L}_f^2  
+
L_f^2
\sup_{r \in[0, 1]}
\left\|
X_r
\right\|^2_{L^2(\Omega ; \mathbb{R}^d)}\Big)
\\
& \leq
8 L_f^2
\big(
\hat{L}_f^2  
+
L_f^2 M_1
\big) 
d h^2.
\end{aligned}
\end{equation}
Next we cope with $T_3$.
First, we have
\begin{equation}
\begin{aligned}
T_3
& =
\underbrace{
4 \beta
\sum_{i=1}^n
\mathbb{E}
\bigg[
\Big\|
\int_{t_{i-1}}^{{t_{i}}}
\int_{t_{i-1}}^{s}
D f_{\beta}(X_r,r) 
\mathrm{\,d} W_r
\mathrm{\,d} s
\Big\|^2
\bigg]}_{=:T_{3,1}}
\nonumber\\
& \quad +
8\beta
\sum_{1 \leq i < j \leq n}
\underbrace{
\E
\bigg[
\Big \langle
\int_{t_{i-1}}^{{t_{i}}}
\int_{t_{i-1}}^{s}
D f_{\beta}(X_r,r) 
\mathrm{\,d} W_r 
\mathrm{\,d} s,
\int_{t_{j-1}}^{{t_{j}}}
\int_{t_{j-1}}^{s}
D f_{\beta}(X_r,r) 
\mathrm{\,d} W_r 
\mathrm{\,d} s
\Big \rangle
\bigg]}_{=:T_{3,2}}.
\end{aligned}
\end{equation}
For any $1\leq i < j \leq n$, we show that $T_{3,2}$ vanishes: 
\begin{equation}
\begin{aligned}
T_{3,2}
& =
\E
\bigg[
\E
\bigg[
\Big \langle
\int_{t_{i-1}}^{{t_{i}}}
\int_{t_{i-1}}^{s}
D f_{\beta}(X_r,r) 
\mathrm{\,d} W_r 
\mathrm{\,d} s,
\int_{t_{j-1}}^{{t_{j}}}
\int_{t_{j-1}}^{s}
D f_{\beta}(X_r,r) 
\mathrm{\,d} W_r 
\mathrm{\,d} s
\Big \rangle
\bigg|
\mathcal{F}^W_{t_{j-1}}
\bigg]
\bigg]\\
& =
\E
\bigg[
\Big \langle
\int_{t_{i-1}}^{{t_{i}}}
\int_{t_{i-1}}^{s}
D f_{\beta}(X_r,r) 
\mathrm{\,d} W_r 
\mathrm{\,d} s,
\int_{t_{j-1}}^{{t_{j}}}
\E
\bigg[
\int_{t_{j-1}}^{s}
D f_{\beta}(X_r,r) 
\mathrm{\,d} W_r
\Big|
\mathcal{F}^W_{t_{j-1}}
\bigg]
\mathrm{\,d} s
\Big \rangle
\bigg]\\
& =0,
\end{aligned}
\end{equation}
where we used the fact that the first integral is $\mathcal{F}^W_{t_{j-1}}$-measurable and the basic property of the It\^o integral.
As a consequence, we further employ the H\"older inequality, the It\^o Isometry \cite[Lemma 5.4]{mao2007stochastic}, Assumption \ref{SFS-W2-ass:f-in-C^2} and Lemma \ref{SFS-W2-lem:x_t-moments-bounded} to show
\begin{equation}
\begin{aligned}
\label{sfs-w2-eq:est-j13}
T_3 
= 
T_{3,1}
&\leq
4 \beta h
\sum_{i=1}^n
\mathbb{E}
\bigg[
\int_{t_{i-1}}^{{t_{i}}}
\Big\|
\int_{t_{i-1}}^{s}
D f_{\beta}(X_r,r) 
\mathrm{\,d} W_r
\Big\|^2
\mathrm{\,d} s
\bigg]
 \\
& =
4 \beta h
\sum_{i=1}^n
\mathbb{E}
\bigg[
\int_{t_{i-1}}^{{t_{i}}}
\int_{t_{i-1}}^{s}
\|
D f_{\beta}(X_r,r) 
\|^2_{\mathrm{F}}
\mathrm{\,d} r
\mathrm{\,d} s
\bigg]
\\
& \leq
4 \beta L_f^2 
d h^2,
\end{aligned}
\end{equation}
{where the last inequality follows from the direct result of \eqref{SFS-W2-prop:Df-bound} that}
\begin{align}
    \|Df_{\beta}(x,t)\|_{\mathrm{F}}
    \leq 
    \sqrt{d}
    \,
    \|Df_{\beta}(x,t)\|
    =
    \sqrt{d}
    \,
    \big(
    \sup_{\|v_1\|=1}
    \|Df_{\beta}(x,t)v_1\|
    \big)
    \leq 
    \sqrt{d}
    \,L_f,
    \:
    \forall\, x \in \mathbb{R}^d, \, ~ t \in [0,1].
\end{align}
For the estimate of $T_4$, by the inequality $(\sum_{i=1}^n u_i )^2 \leq n \sum_{i=1}^n u_i^2$, the H\"older inequality, 
Assumption \ref{SFS-W2-ass:f-in-C^2} and Lemma \ref{SFS-W2-lem:x_t-moments-bounded}, one obtains
\begin{equation}
\label{sfs-w2-eq:est-j14}
    \begin{aligned}
        T_4
        & \leq
        \beta^2 nh^2 
        \sum_{i=1}^n
        \mathbb{E}
        \bigg[
        \int_{t_{i-1}}^{{t_{i}}}
        \int_{t_{i-1}}^{s}
        \Big\|
        \sum_{j=1}^d
        D^2 f_{\beta}(X_r,r) (e_j,e_j)
        \Big\|^2
        \mathrm{\,d} r 
        \mathrm{\,d} s
        \bigg]\\
        & \leq
        \beta^2
        h^2
        \sup_{r\in[0, 1]} 
        \mathbb{E}
         \bigg[
         \Big\|
         \sum_{j=1}^d
         D^2 f_{\beta}(X_r,r) (e_j,e_j)
         \Big\|^2
         \bigg]\\
         & \leq
        \beta^2
        d
        h^2
        \sup_{r\in[0, 1]}
        \Big(
        \sum_{j=1}^d
        \mathbb{E}
         \Big[
         \big\|
         D^2 f_{\beta}(X_r,r) (e_j,e_j)
         \big\|^2
         \Big]
         \Big)\\
         & \leq
         \beta^2
         L_f^{' 2}
         d^2 h^2.
    \end{aligned}
\end{equation}
Combining \eqref{sfs-w2-eq:est-j11}, \eqref{sfs-w2-eq:est-j12}, \eqref{sfs-w2-eq:est-j13} with \eqref{sfs-w2-eq:est-j14} yields
\begin{align*}
\mathbb{E}
\bigg[
\Big\|
\sum_{i=1}^{n}
\mathcal{R}_i
\Big\|^2
\bigg]
\leq
C_1 d^2 h^2,
\end{align*}
where $C_1
: =
\max
\big\{
16\widetilde{L}^2_f,
8 L_f^2
\big(
\hat{L}_f^2  
+
L_f^2 M_1
\big),
4 \beta L^2_f,
\beta^2 L_f^{' 2}
\big\}$.
By virtue of Proposition \ref{SFS-W2-prop:Upper-mean-square error-bounds} and the discrete Gronwall inequality \cite[Lemma 2.3]{kruse2019randomized}, we get
\begin{align}
\mathbb{E}
\Big[
\big\|
X_{1}
-
{Y}_{1}
\big\|^2
\Big]
& \leq
2
C_1 d^2 h^2
\exp 
\Big(
\sum_{i=1}^N
2 L_f^2 h
\Big)
= 
2
C_1
\exp 
\big(
2 L_f^2
\big)
d^2 h^2.
\end{align}
The proof of Theorem \ref{SFS-W2-thm:main-rerult-no-mc} is completed.
\end{proof}

\subsection{Proof of Theorem \ref{SFS-W2-thm:main-rerult-mc}}
%
%
Before proving Theorem \ref{SFS-W2-thm:main-rerult-mc}, we present the following lemma.
\begin{lem}
\label{SFS-W2-lem:jiao-M-error}
    Let Assumption \ref{SFS-W2-ass:g-nabla-g-lip} hold. Then for any $x\in \mathbb{R}^d$ and $n \in [N-1]_0$,
    \begin{align*}
        \mathbb{E}_{\xi}
        \big[
        \|f_{\beta}(x, t_n)
        -\widetilde{f}^M_{\beta}
        (x, t_n)
        \|^2
        \big]
        \leq
        { 
        \frac{C_2 d}
        {M},}
    \end{align*}
    { where 
    $C_2
    :=
    \tfrac{4 L_g^2\beta^2}
    {\rho^2}    \Big(1+\tfrac{L_g^2}{\rho^2}\Big)$,}
    $L_g, \rho$ come from Assumption \ref{SFS-W2-ass:g-nabla-g-lip},
    and $M$ is the number of samples used in the Monte Carlo estimator \eqref{SFS-W2-eq:f-M-experssion}.
\end{lem}
See \cite[Lemma A.6]{MRjiao} for a similar proof of this lemma.
Now we can prove Theorem \ref{SFS-W2-thm:main-rerult-mc}.
\begin{proof}
For the considered sampler \eqref{SFS-W2-eq:Euler-scheme-mc}, we have
$\hat{Y}=\widetilde{Y}^M $ and $ \hat{f}_{\beta}=\widetilde{f}^M_{\beta}$.
{First, one can do the following error decomposition:
\begin{equation}
\begin{aligned}
\mathbb{E}
\bigg[
\Big\|
\sum_{i=1}^{n}
\mathcal{R}_i
\Big\|^2
\bigg]
& =
\mathbb{E}
\bigg[
\Big\|
\sum_{i=1}^n
\int_{t_{i-1}}^{{t_{i}}}
\big[
f_{\beta}  (X_s,s)
-
\widetilde{f}^M_{\beta}  
(
X_{t_{i-1}}, t_{i-1}
)
\big]
\,
\mathrm{\,d} s
\Big\|^2
\bigg]\\
& \leq
\underbrace{
2
\,
\mathbb{E}
\bigg[
\Big\|
\sum_{i=1}^n
\int_{t_{i-1}}^{{t_{i}}}
\big[
f_{\beta}  
(X_s,s)
-
f_{\beta}  
(
X_{t_{i-1}}, t_{i-1}
)
\big]
\,
\mathrm{\,d} s
\Big\|^2
\bigg]}_{=:J_1}\\
& \quad +
\underbrace{
2
\,
\mathbb{E}
\bigg[
\Big\|
\sum_{i=1}^n
\int_{t_{i-1}}^{{t_{i}}}
\big[
f_{\beta} 
(X_{t_{i-1}}, t_{i-1})
-
\widetilde{f}^M_{\beta}  
(
X_{t_{i-1}}, t_{i-1}
)
\big]
\,
\mathrm{\,d} s
\Big\|^2
\bigg]}_{=:J_2},
\end{aligned}
\end{equation}
where $J_1$ has been estimated in the proof of Theorem \ref{SFS-W2-thm:main-rerult-no-mc}:
\begin{align}
\label{SFS-W2-eq:J_1_J_3}
    J_1
    \leq
    2 C_1 d^2 h^2.
\end{align}
Now it remains to estimate $J_2$.
Thanks to the inequality $(\sum_{i=1}^n u_i )^2 \leq n \sum_{i=1}^n u_i^2$, the property of the conditional expectation \cite[Theorem 2.24]{klebaner2012introduction} and Lemma \ref{SFS-W2-lem:jiao-M-error}, one can get
\begin{align}
\label{sfs-w2-eq:est-j2}
        J_2
        & \leq
        2 n h^2
        \sum_{i=0}^{n-1}
        \mathbb{E}
        \Big[
        \big\|
        f_{\beta}\big(
        X_{t_i},t_i
        \big)
        -
        \widetilde{f}^M_{\beta}
        \big(
        X_{t_i},t_i
        \big)
        \big\|^2
        \Big]
        \nonumber\\
        & =
        2 h
        \sum_{i=0}^{n-1}
        \mathbb{E}
        \left[
        \mathbb{E}
        \left[
        \big\|
        f_{\beta}\big(
        X_{t_i},t_i
        \big)
        -
        \widetilde{f}^M_{\beta}
        \big(
        X_{t_i},t_i
        \big)
        \big\|^2
        \Big|
        X_{t_i}
        \right]
        \right]
        \nonumber\\
        & =
        2 h
        \sum_{i=0}^{n-1}
        \mathbb{E}_W
        \left[
        \mathbb{E}_{\xi}
        \left[
        \big\|
        f_{\beta}
        \big(
        X_{t_i},t_i
        \big)
        -
        \widetilde{f}^M_{\beta}
        \big(
        X_{t_i},t_i
        \big)
        \big\|^2
        \right]
        \right]
        \nonumber\\
        & \leq
        \frac{2 C_2 d}
        {M}.
\end{align}
In view of \eqref{SFS-W2-eq:J_1_J_3} and \eqref{sfs-w2-eq:est-j2}, and by the discrete Gronwall inequality \cite[Lemma 2.3]{kruse2019randomized}, we derive
\begin{align}
\mathbb{E}
\Big[
\big\|
X_{1}
-
\widetilde{Y}^M_{1}
\big\|^2
\Big]
& \leq
2
\Big(
C_1 d^2 h^2
+
\tfrac{C_2 d}{M}
\Big)
\exp 
\big(
\sum_{i=1}^N
2 L_f^2 h
\big)
= 
2
\Big(
C_1 d^2 h^2
+
\tfrac{C_2 d}{M}
\Big)
\exp 
\big(
2 L_f^2
\big).
\end{align}
All together, we obtain
\begin{equation}
\label{sfs-w2-eq:l2-total-bound}
\mathbb{E}
\Big[
\big\|
X_{1}
-
\widetilde{Y}^M_{1}
\big\|^2
\Big] 
\leq
C d^2 h^2
+
\frac{C d}{M},
\end{equation}
where
$C
: =
2
e^{2 L_f^2}\cdot\max\{C_1,C_2\}$.
The proof of Theorem \ref{SFS-W2-thm:main-rerult-mc} is completed.}
\end{proof}

\section{Numerical experiments}
\label{SFS-W2-section:Numerical-experiment}
\noindent
In order to test the efficiency and quality of the proposed sampling algorithms, some numerical results are performed in this section. 
We illustrate our findings via the one-dimensional, two-dimensional, high-dimensional Gaussian mixture distributions and other complex distributions.
%
\subsection{Three shaped distributions in two dimensions}
In this subsection, we consider sampling from the following two-dimensional distributions with probability density functions (PDF):
$$
    p_1(x)
    =
    c_1
    \,\mathrm{e}^{-\left(\left(x_1 x_2\right)^2+x_1^2+x_2^2-8\left(x_1+x_2\right)\right) / 2}, 
$$
$$
    p_2(x) 
    =
    c_2
    \,\mathrm{e}^{-\tfrac{\left(r-r_0\right)^2}{2 \sigma^2}}, 
    \quad r=\sqrt{x_1^2+x_2^2},
    ~r_0=2,~\sigma=\tfrac{1}{5},
$$
$$
    p_3(x) 
    =
    c_3
    \,
    \mathrm{e}^{-x_1^2/2}
    \mathrm{e}^{-x_2^2/2\mathrm{e}^{2\alpha x_1}}, 
    \quad \alpha=\tfrac{3}{5},
$$
where $c_1,c_2,c_3 $ are {normalizing} constants.
Here, $p_1(x)$ is taken from \cite[Example 6.4]{rubinstein2016simulation}, $p_2$ is the Ring distribution, with $r_0$ being the ring center radius and $\sigma$ being the ring width, and $p_3(x)$ is the Anisotropic Funnel distribution, with the parameter $\alpha >0$ controlling the shape of the funnel.
We simulate $2000$ independent Markov chains using SFS algorithm \eqref{SFS-W2-eq:Euler-scheme-mc} with $\beta=1$ and step-size $h=10^{-3}$.
In Figures \ref{figure:p1}, \ref{figure:ringsdistribution}, \ref{figure:gmd_4_density}, 
the left pictures present the true contour plots, while the right ones depict the scatter plots of sampled points.
Evidently, the SFS gives a good performance in sampling from these complex distributions.
\begin{figure}[htbp]
    \centering

    \begin{subfigure}[b]{0.45\textwidth}
        \centering
        \includegraphics[width=\linewidth]{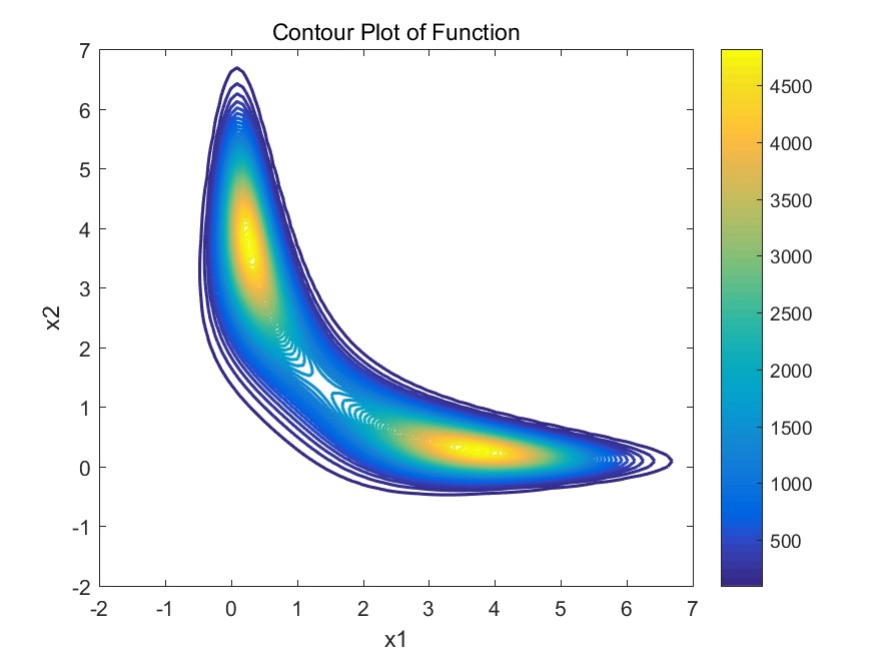}
    \end{subfigure}
    \hspace{-0.5cm} 
    \begin{subfigure}[b]{0.45\textwidth}
        \centering
        \includegraphics[width=\linewidth]{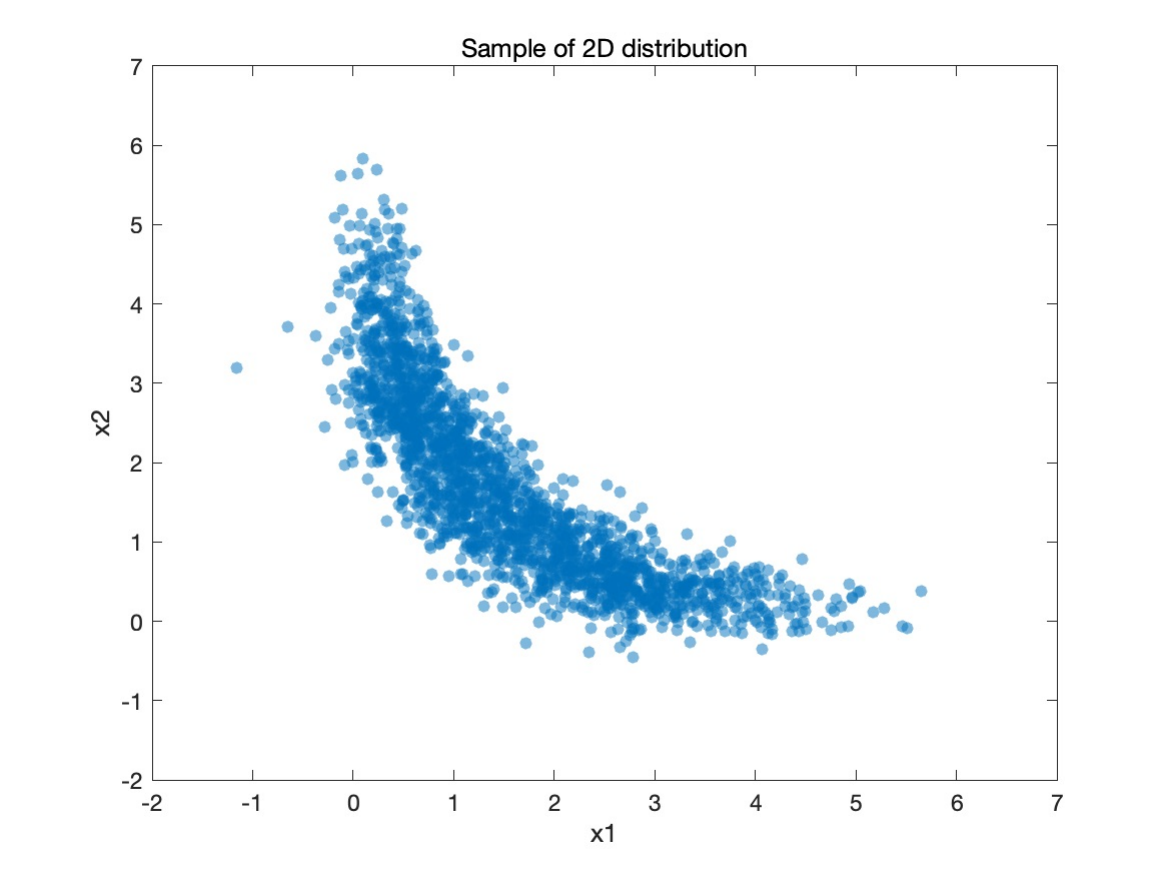}
    \end{subfigure}
    \caption{Sampling from $p_1(x)$ with SFS.}
    \label{figure:p1}
    \vspace{1em} 

    \begin{subfigure}[b]{0.45\textwidth}
        \centering
        \includegraphics[width=\linewidth]{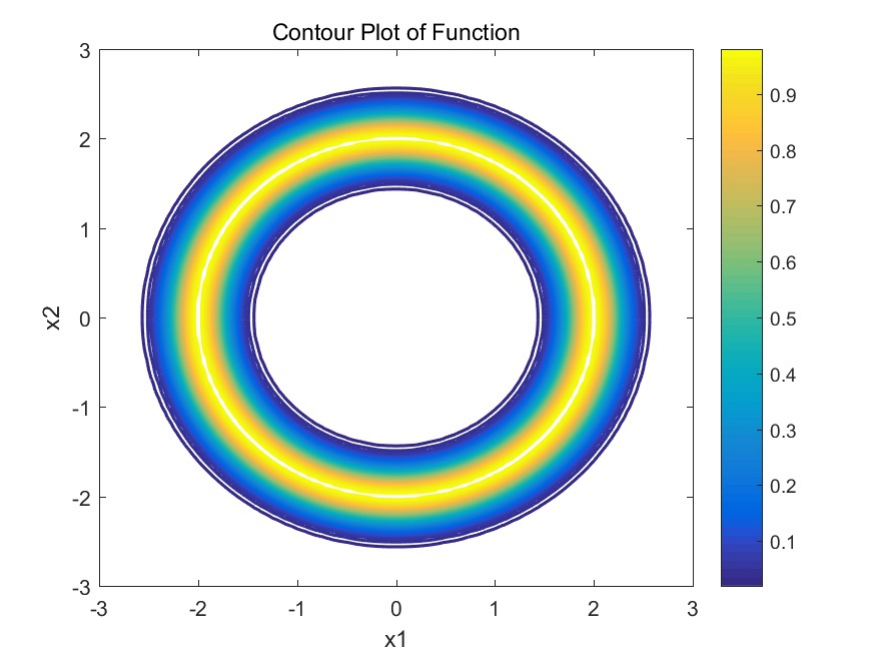}
    \end{subfigure}
    \hspace{-0.5cm} 
    \begin{subfigure}[b]{0.45\textwidth}
        \centering
        \includegraphics[width=\linewidth]{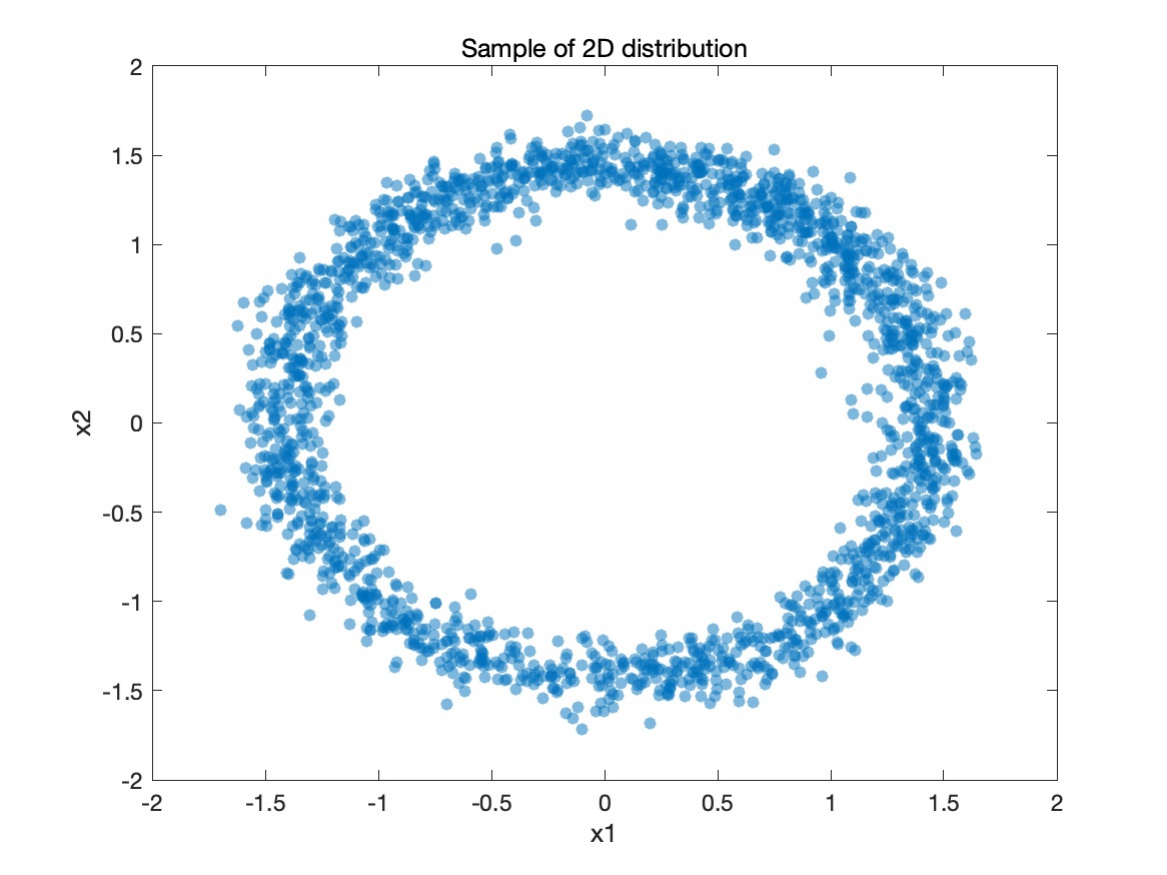}
    \end{subfigure}
    \caption{Sampling from $p_2(x)$ with SFS.}
    \label{figure:ringsdistribution}
    
    \vspace{1em}
    
    \begin{subfigure}[b]{0.45\textwidth}
        \centering
        \includegraphics[width=\linewidth]{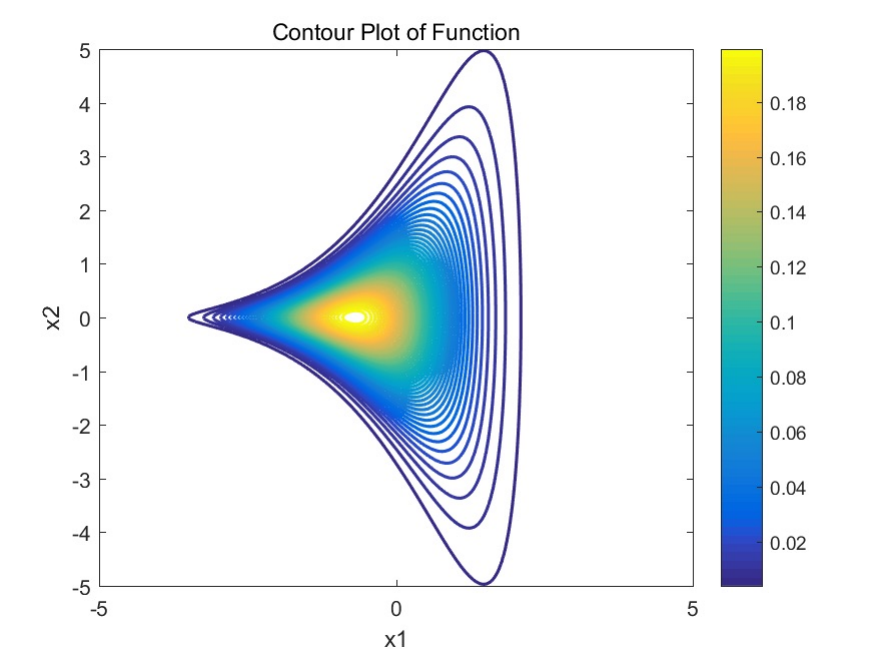}
    \end{subfigure}
    \hspace{-0.5cm} 
    \begin{subfigure}[b]{0.45\textwidth}
        \centering
        \includegraphics[width=\linewidth]{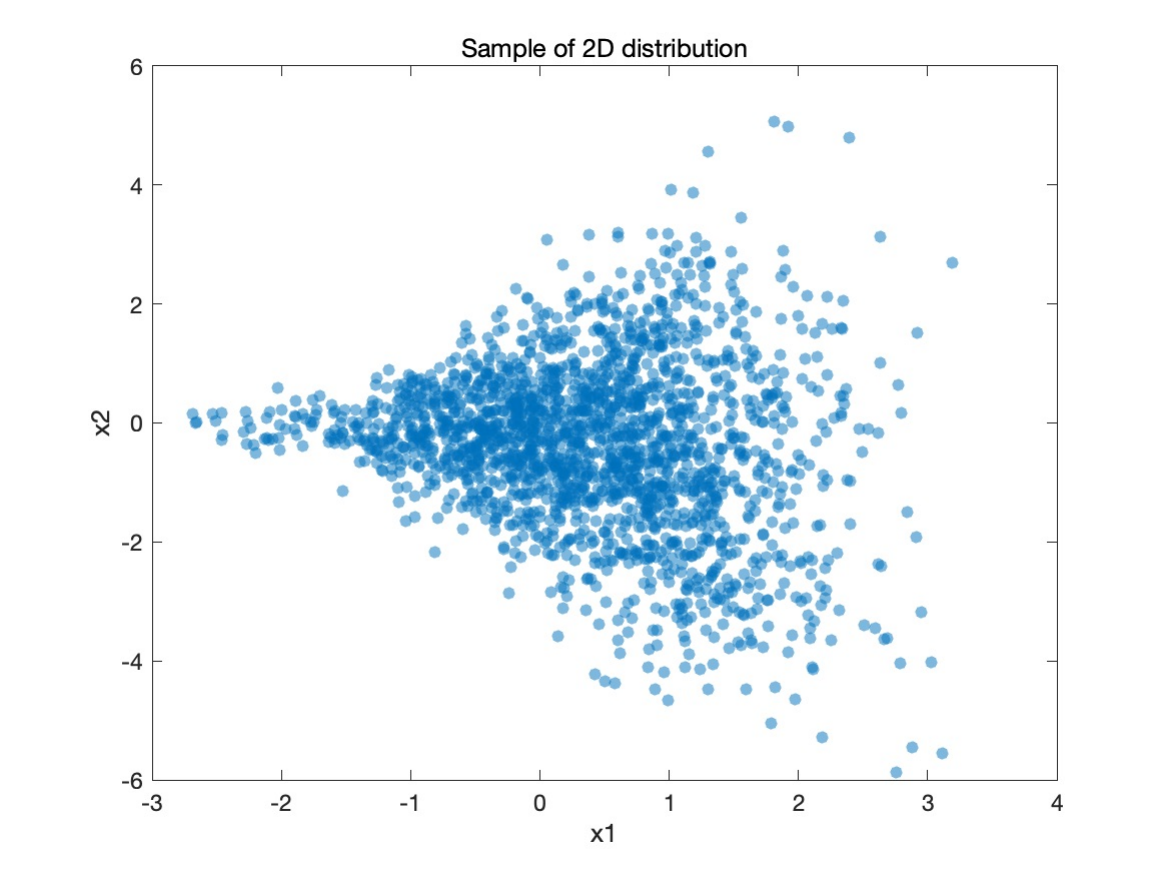}
    \end{subfigure}
    \caption{Sampling from $p_3(x)$ with SFS.}
    \label{figure:gmd_4_density}
\end{figure}

\subsection{Gaussian mixture distributions}
\label{subsection:Gaussian-mixture}
In this subsection, we focus on sampling from one-dimensional, two-dimensional and high-dimensional Gaussian mixture distributions. 
\subsubsection{One-dimensional Gaussian mixture distributions}
First, consider sampling from the following one-dimensional Gaussian mixture distribution,
\begin{align*}
    p_4(x)
    =
    \tfrac{3}{4}
    \mathcal{N}(x ;\alpha_1,\Sigma_1)
    +
    \tfrac{1}{4}
    \mathcal{N}(x ;\alpha_2,\Sigma_2),
\end{align*}
where $\alpha_1, \alpha_2 \in \mathbb{R}$ and $\Sigma_1=\tfrac{1}{5},\Sigma_2=\tfrac{4}{5}$ are the mean and variance of the Gaussian component, respectively.
{ For a fixed step-size $h=10^{-3}$, we simulate $1000$ independent samples using the SFS \eqref{SFS-W2-eq:Euler-scheme-no-mc} with $\beta=1$ up to terminal time 
$T=1$, the overdamped LMC algorithm \eqref{SFS-W2-eq:ULA} and underdamped LMC algorithm \eqref{SFS-W2-eq:under-EM} up to $T = 10.$}
In Figures \ref{figure:gmd_lmc_different_cov_density_1},
the PDF obtained through SFS is depicted by red circles, whereas those computed via the overdamped LMC and underdamped LMC  are plotted by blue stars and green dashed line, respectively.
For reference, the exact density is plotted as solid black line.

%
One can observe that, in the case $\alpha_1=-2, \alpha_2=2$,  the overdamped LMC and SFS both exhibit bimodal distributions. As shown in Figure \ref{figure:gmd_lmc_different_cov_density_1} (a), SFS demonstrates significantly better performance than the overdamped LMC.
%
%
When the distance between two means of the Gaussian components increases (e.g., $\alpha_1=-6, \,\alpha_2=6$),  only samples from SFS can accurately recover the underlying target distribution with two peaks, while the other Langevin algorithms including overdamped and underdamped LMC give a poor performance and collapse on only one mode, as shown in Figure \ref{figure:gmd_lmc_different_cov_density_1} (b).

\begin{figure}[htbp]
    \centering
    \begin{subfigure}{0.45\linewidth}  
        \includegraphics[width=\linewidth]{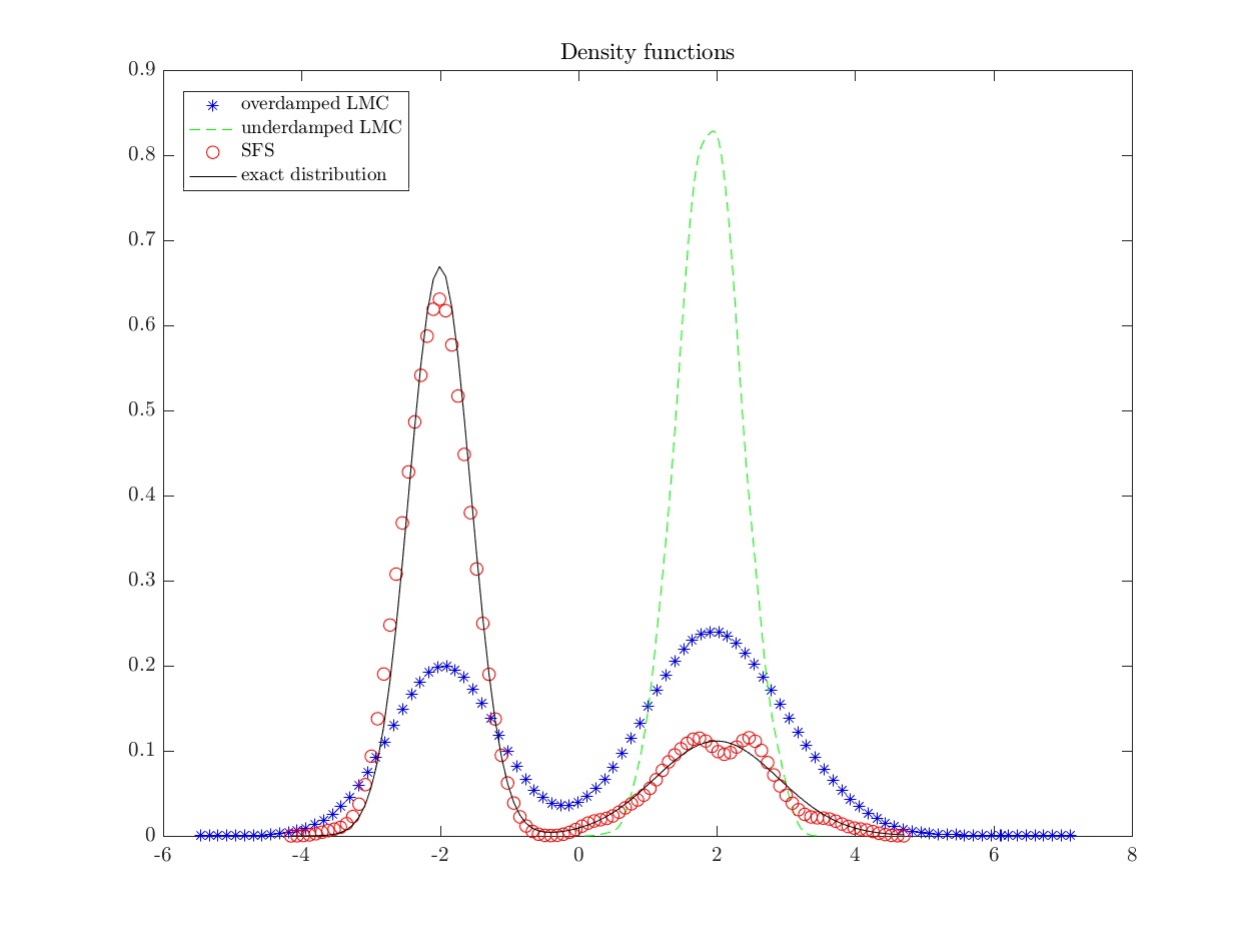} 
        \caption{$\alpha_1=-2,~\alpha_2=2$}	
    \end{subfigure}
    \hspace{-0.5cm} 
    \begin{subfigure}{0.45\linewidth}
        \includegraphics[width=\linewidth]{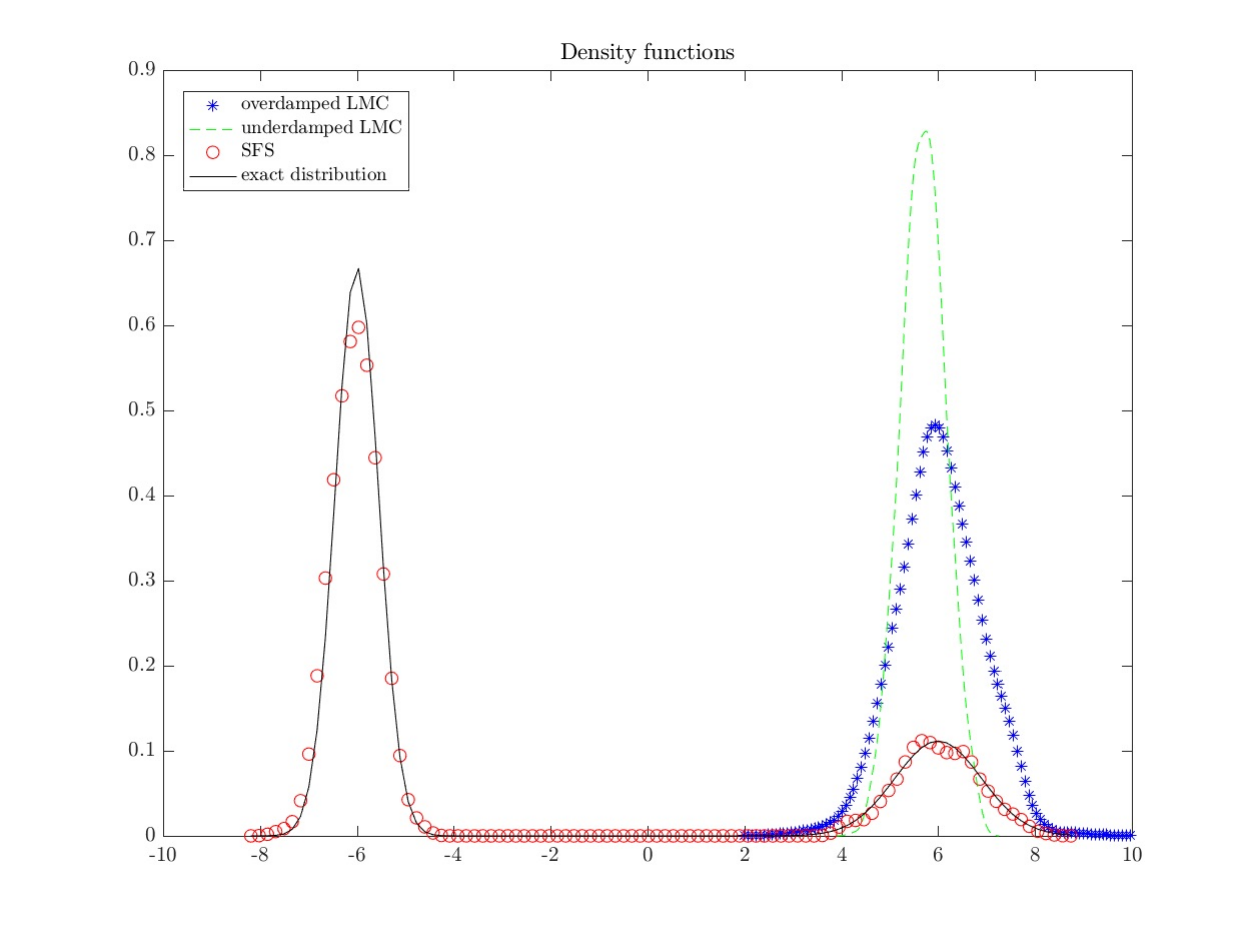}
        \caption{$\alpha_1=-6,~\alpha_2=6$}	
    \end{subfigure}
    \caption{Probability density of the one-dimensional Gaussian mixture distribution.}
    \label{figure:gmd_lmc_different_cov_density_1}
\end{figure}

\subsubsection{Two-dimensional Gaussian mixture distributions}
Next, we consider the following two-dimensional Gaussian mixture distributions,
$$
    p_5(x) 
    =
    \tfrac{1}{\kappa}
    \sum_{i=1}^\kappa
    \mathcal{N}(x;\alpha_i, \Sigma_i),
    \quad
    \kappa \in \mathbb{N},
$$
where $\alpha_i
=\lambda_i(\sin (2(i-1) \pi / \kappa), \cos (2(i-1) \pi / \kappa))$ are the mean value and $\Sigma_i \in 
\mathbb{R}^{2\times2}$ are the two-dimensional covariance matrix.

\textbf{Sampling via SFS and Langevin algorithms:}
We compare the SFS algorithm \eqref{SFS-W2-eq:Euler-scheme-no-mc} with two Langevin sampling algorithms: the overdamped LMC \eqref{SFS-W2-eq:ULA} and the BAOAB method for the underdamped Langevin dynamics \cite{LBNC2013}.
For a uniform step-size $h=10^{-3}$, we simulate $2000$ independent paths.
%
In Figure \ref{fig:gmd_8}, we set $\kappa=8, \, \lambda_1=\lambda_3=\lambda_5=\lambda_7=6,\,\lambda_2=\lambda_4=\lambda_6=\lambda_8=2\sqrt{3}$ and $\Sigma_i=\tfrac{1}{5}\mathbf{I}_2, \, i=1,\cdots,8$.
The top panel presents the true contour plots. The left panel shows both scatter and contour plots of points sampled using SFS \eqref{SFS-W2-eq:Euler-scheme-no-mc} with $\beta=1$ { up to the terminal time $T=1$,} while the middle panel displays corresponding plots generated by overdamped LMC  \eqref{SFS-W2-eq:ULA} { up to the $T=10$.} The right panel shows the sampling results obtained by the BAOAB method { up to $T=10$.}
One can observe that the SFS sampler exhibits significantly better sampling performance than the other two Langevin algorithms, which exhibit only 
four peaks in Figure \ref{fig:gmd_8}.
\begin{figure}[htbp]
    \centering
    \begin{minipage}{\textwidth}
        \centering

        \begin{subfigure}{0.8\textwidth} 
            \centering
            \includegraphics[width=\linewidth, height=5cm, keepaspectratio]{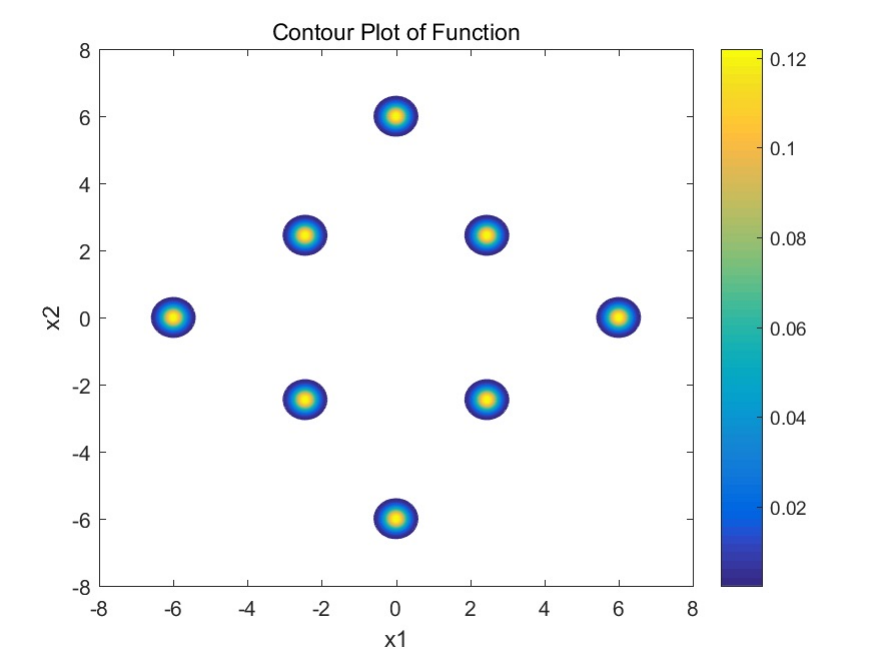} 
            \caption{True contour plot for $p_5$ with $\kappa=8$} 
        \end{subfigure}
        
        \vspace{0.3cm} 
     
        \begin{minipage}{\textwidth}
            \centering
            \begin{subfigure}{0.32\textwidth} 
                \includegraphics[width=\linewidth, height=4.5cm]{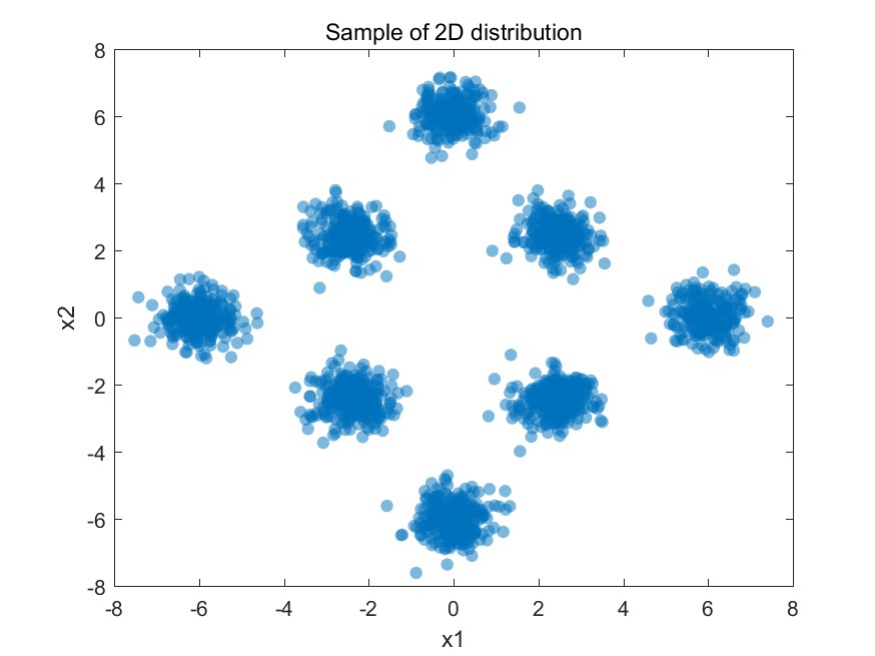} 
                \caption{SFS (scatter)}
            \end{subfigure}
            \hfill
            \begin{subfigure}{0.32\textwidth} 
                \includegraphics[width=\linewidth, height=4.5cm]{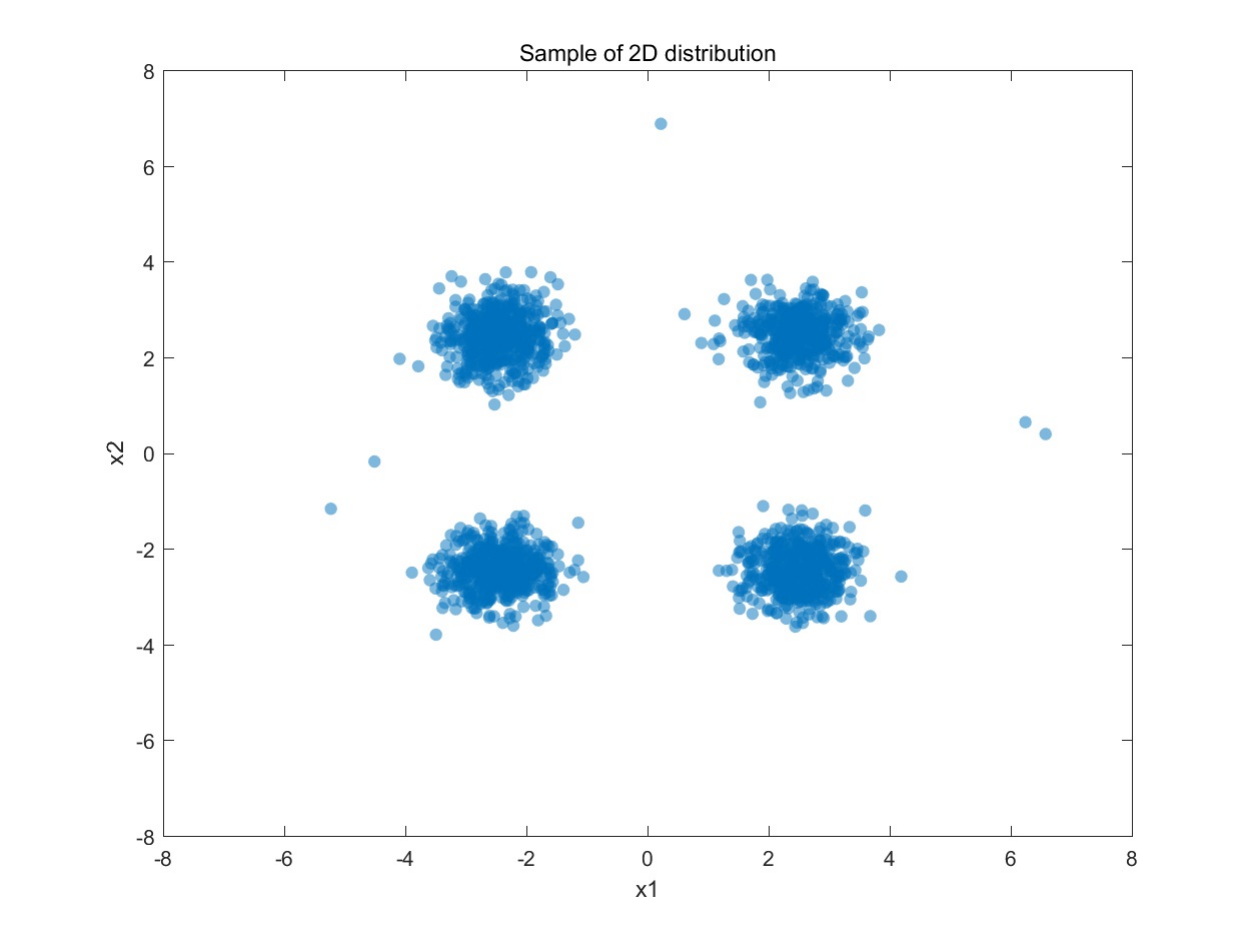}
                \caption{Overdamped LMC (scatter)}
            \end{subfigure}
            \hfill
            \begin{subfigure}{0.32\textwidth} 
                \includegraphics[width=\linewidth, height=4.5cm]{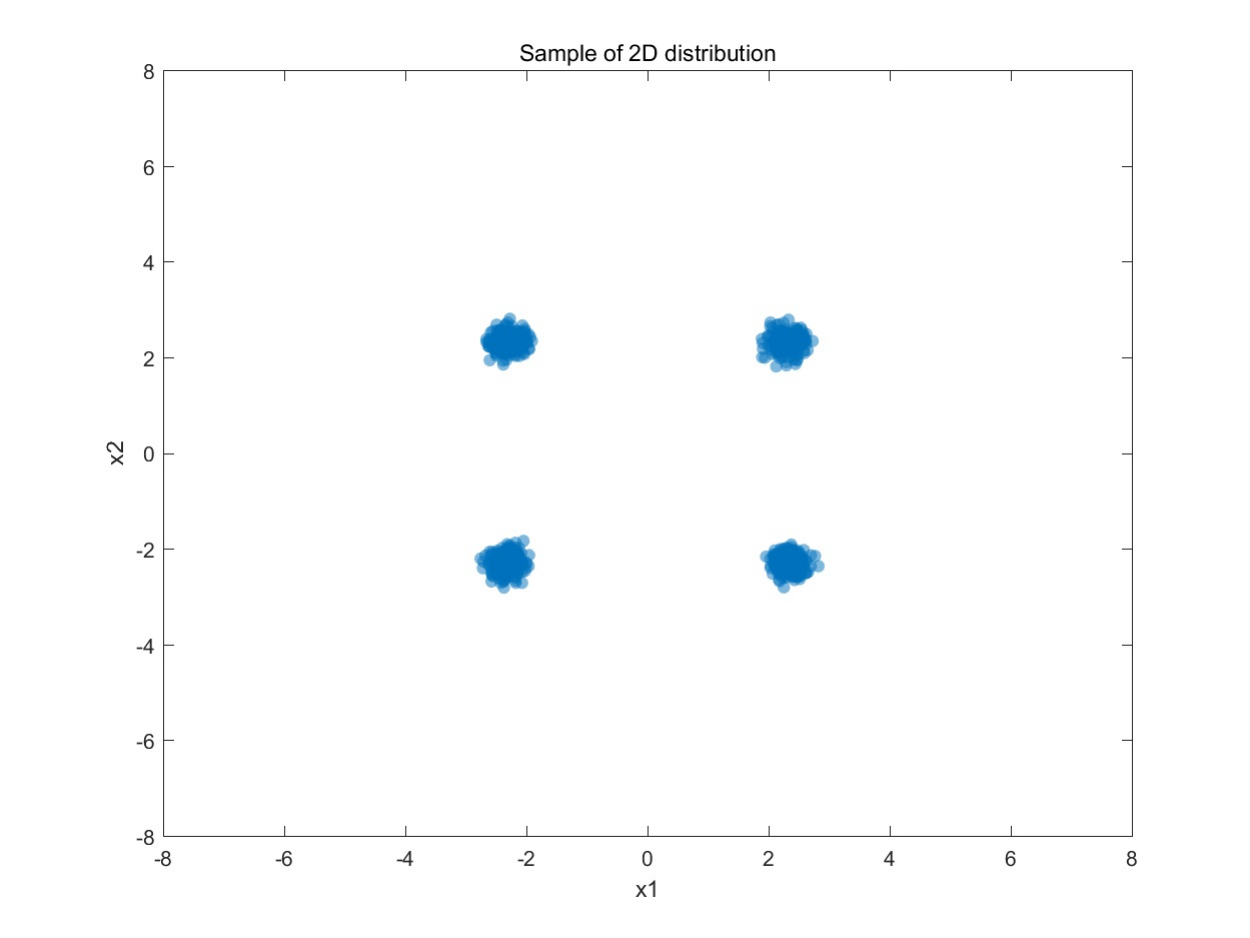} 
                \caption{BAOAB (scatter)} 
            \end{subfigure}
        \end{minipage}

        \vspace{0.3cm}

        \begin{minipage}{\textwidth}
            \centering
            \begin{subfigure}{0.32\textwidth} 
                \includegraphics[width=\linewidth, height=4.5cm]{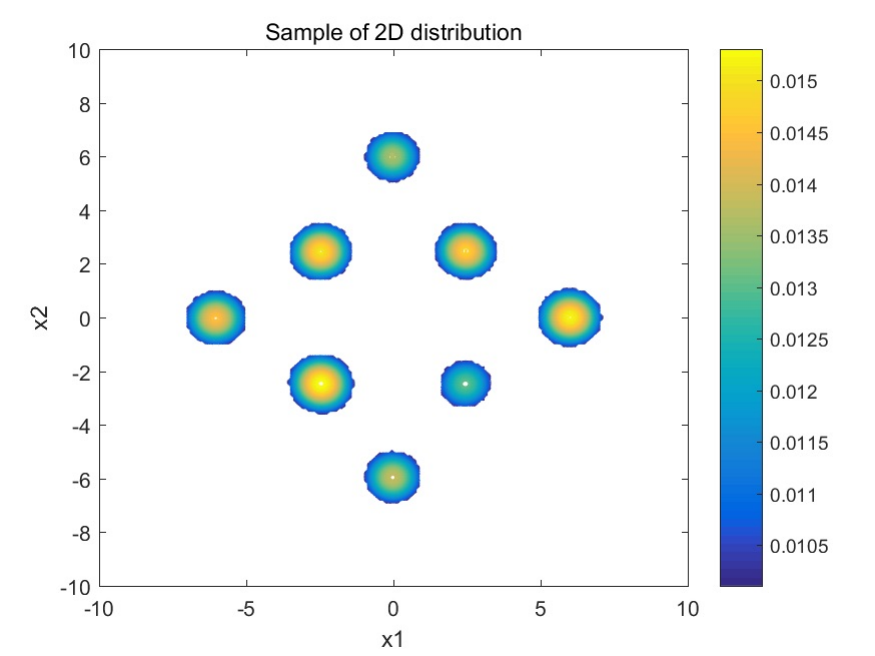} 
                \caption{SFS (contour)}
            \end{subfigure}
            \hfill
            \begin{subfigure}{0.32\textwidth} 
                \includegraphics[width=\linewidth, height=4.5cm]{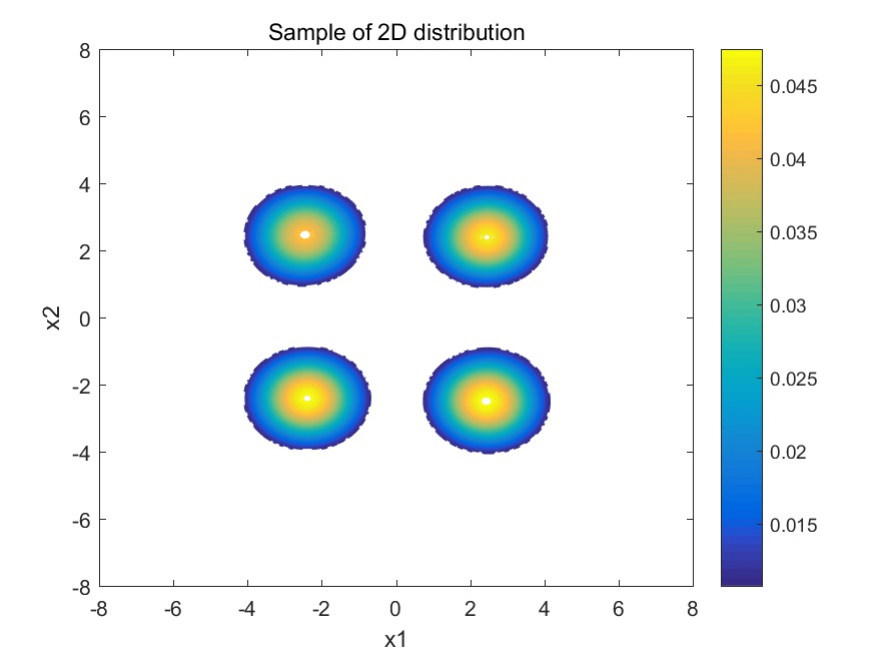}
                \caption{Overdamped LMC (contour)}
            \end{subfigure}
            \hfill
            \begin{subfigure}{0.32\textwidth} 
                \includegraphics[width=\linewidth, height=4.5cm]{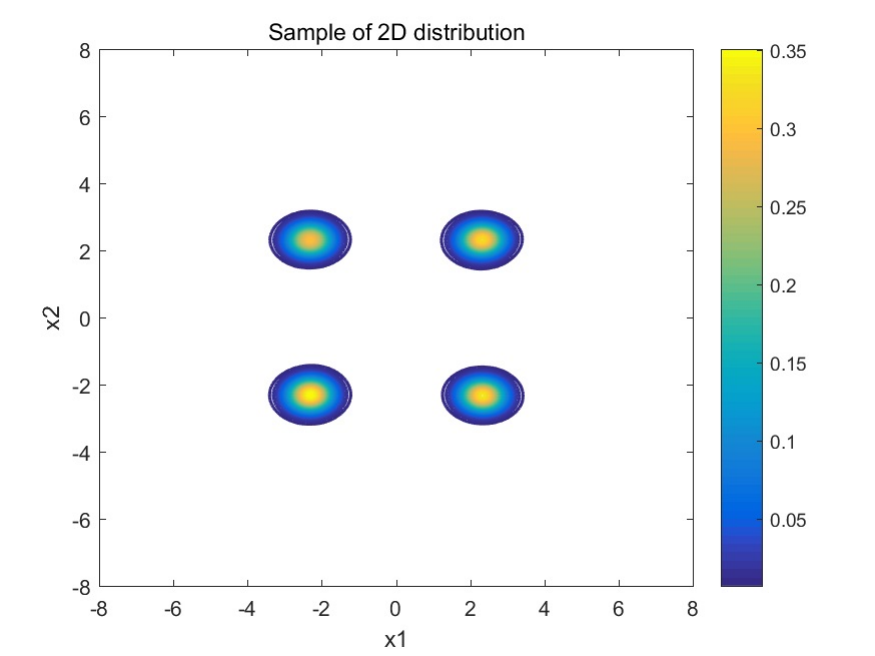} 
                \caption{BAOAB (contour)} 
            \end{subfigure}
        \end{minipage}

        \caption{
           Sampling via SFS, Overdamped LMC, and BAOAB.
        }
        \label{fig:gmd_8}
    \end{minipage}
\end{figure}

\textbf{Sampling via SFS with different temperatures $\beta$:}
To demonstrate the impact of $\beta$, we test the SFS algorithm \eqref{SFS-W2-eq:Euler-scheme-mc} with different temperatures $\beta=1,2,5$, where the drift is approximated by using Monte Carlo methods with $M=200$ samples.
For a uniform time step-size $h=10^{-3}$, we simulate $2000$ independent Markov chains,
for two target distributions: $\kappa=4, \,\lambda_i=6,\,i=1,\cdots,4$ , $\Sigma_1=\Sigma_3=\tfrac{1}{5}\mathbf{I}_2, \Sigma_2=\Sigma_4=\tfrac{2}{5}\mathbf{I}_2$
and $\kappa=8, \, \lambda_1=\lambda_3=\lambda_5=\lambda_7=6,\,\lambda_2=\lambda_4=\lambda_6=\lambda_8=2\sqrt{3}, \,\Sigma_i=\tfrac{1}{5}\mathbf{I}_2, \, i=1,\cdots,8$.
In Figure \ref{fig:2-beta}, the top panel displays the true contour plots and the other ones present scatter plots of SFS with different temperatures $\beta$ (from top to bottom: $\beta =1, 2, 5$). For the case $\beta =1$, the SFS algorithm \eqref{SFS-W2-eq:Euler-scheme-mc} reduces into the SFS algorithm with inexact drift in \cite{MRjiao} and gives reduced modes in sampling. As the temperature $\beta$ increases to  $\beta =2, 5$, the SFS algorithms show significant improvements in sampling performance and generate high-quality samples with the desired numbers of modes, as observed in Figure \ref{fig:2-beta}.
%
\begin{figure}[htbp]
    \centering
    \begin{tabular}{@{}c cc@{}}
    
    \multirow{4}{*}
    &
    \begin{subfigure}[b]{0.4\textwidth}
        \includegraphics[width=\textwidth]{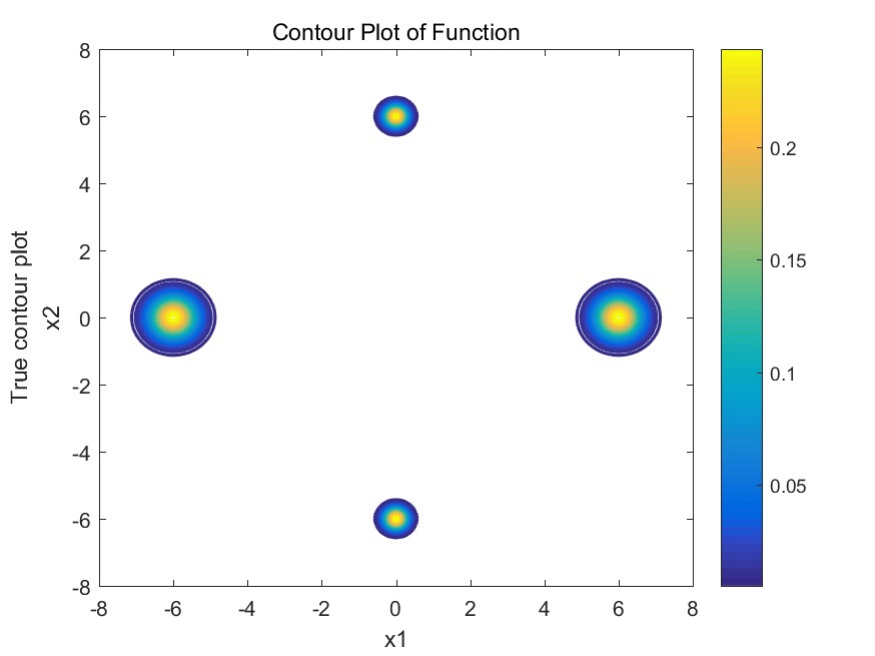}
        \label{fig:sub1}
    \end{subfigure} &
    \begin{subfigure}[b]{0.4\textwidth}
        \includegraphics[width=\textwidth]{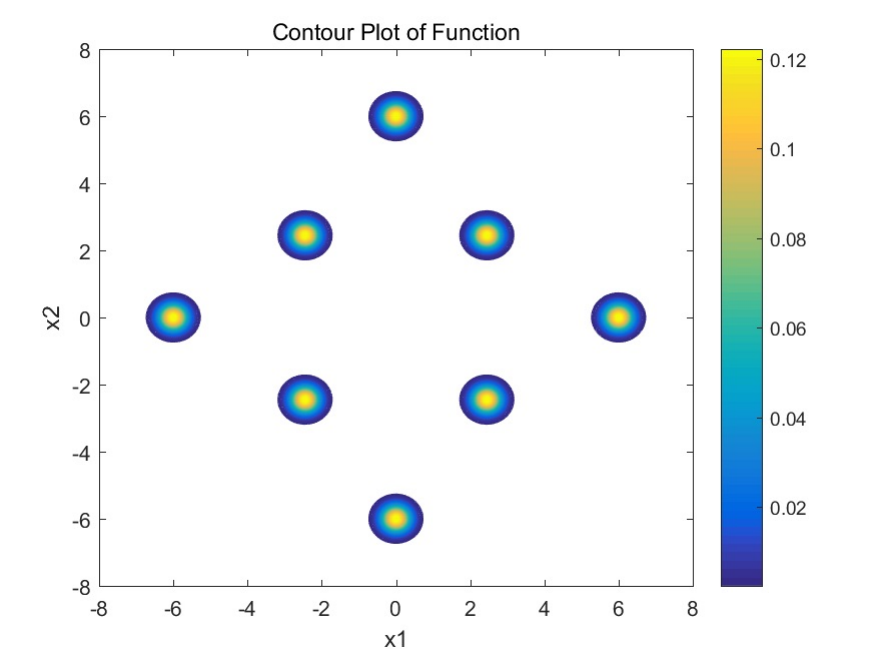}
        \label{fig:sub2}
    \end{subfigure} \\[-9pt] 
    
    &
    \begin{subfigure}[b]{0.4\textwidth}
        \includegraphics[width=\textwidth]{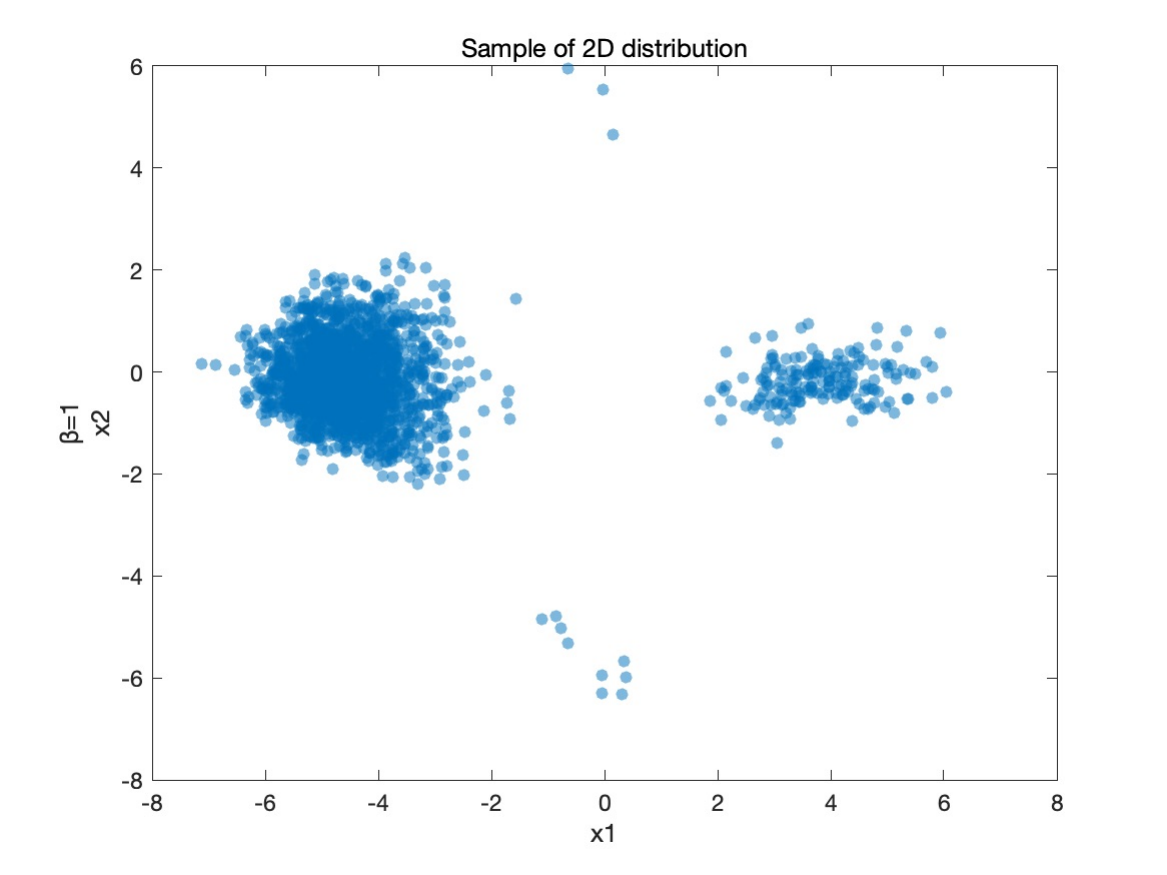}
        \label{fig:sub3}
    \end{subfigure} &
    
    \begin{subfigure}[b]{0.4\textwidth}
        \includegraphics[width=\textwidth]{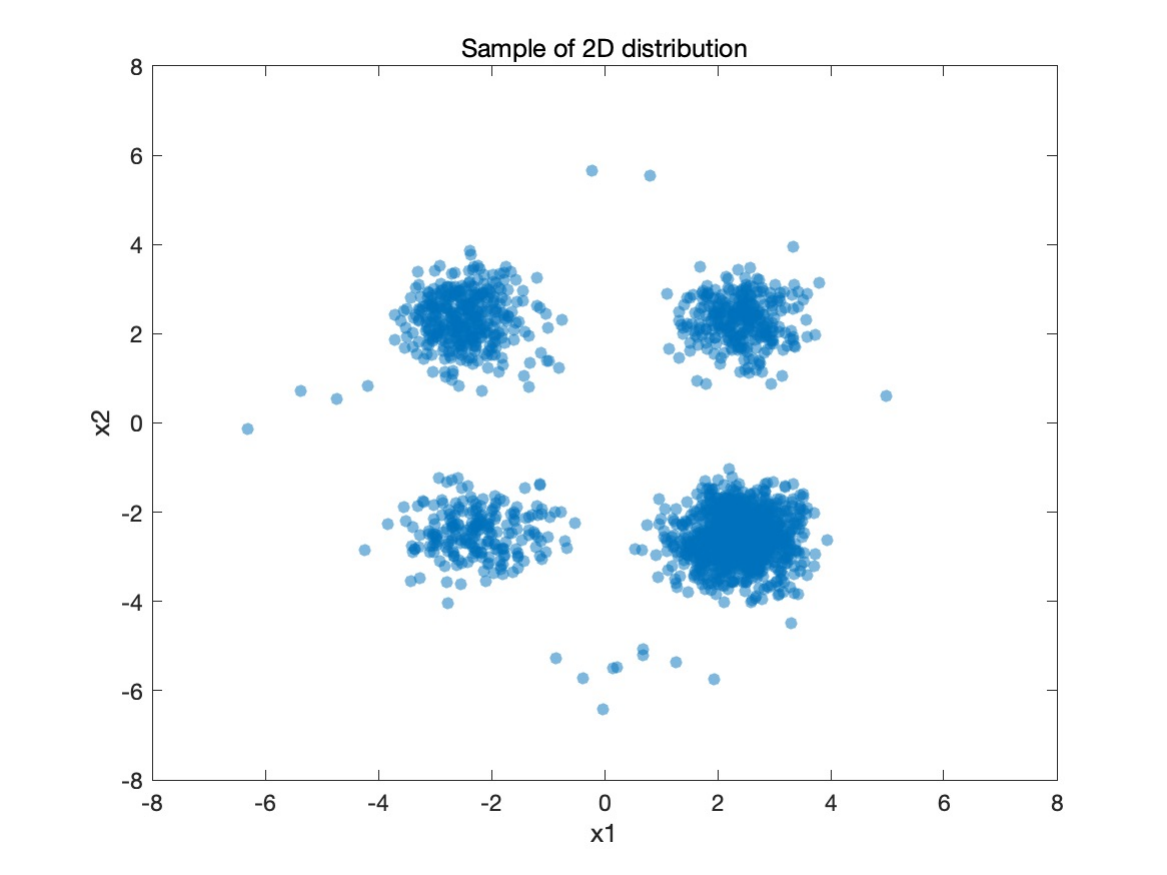}
        \label{fig:sub4}
    \end{subfigure} \\[-9pt]
    
    &
    \begin{subfigure}[b]{0.4\textwidth}
        \includegraphics[width=\textwidth]{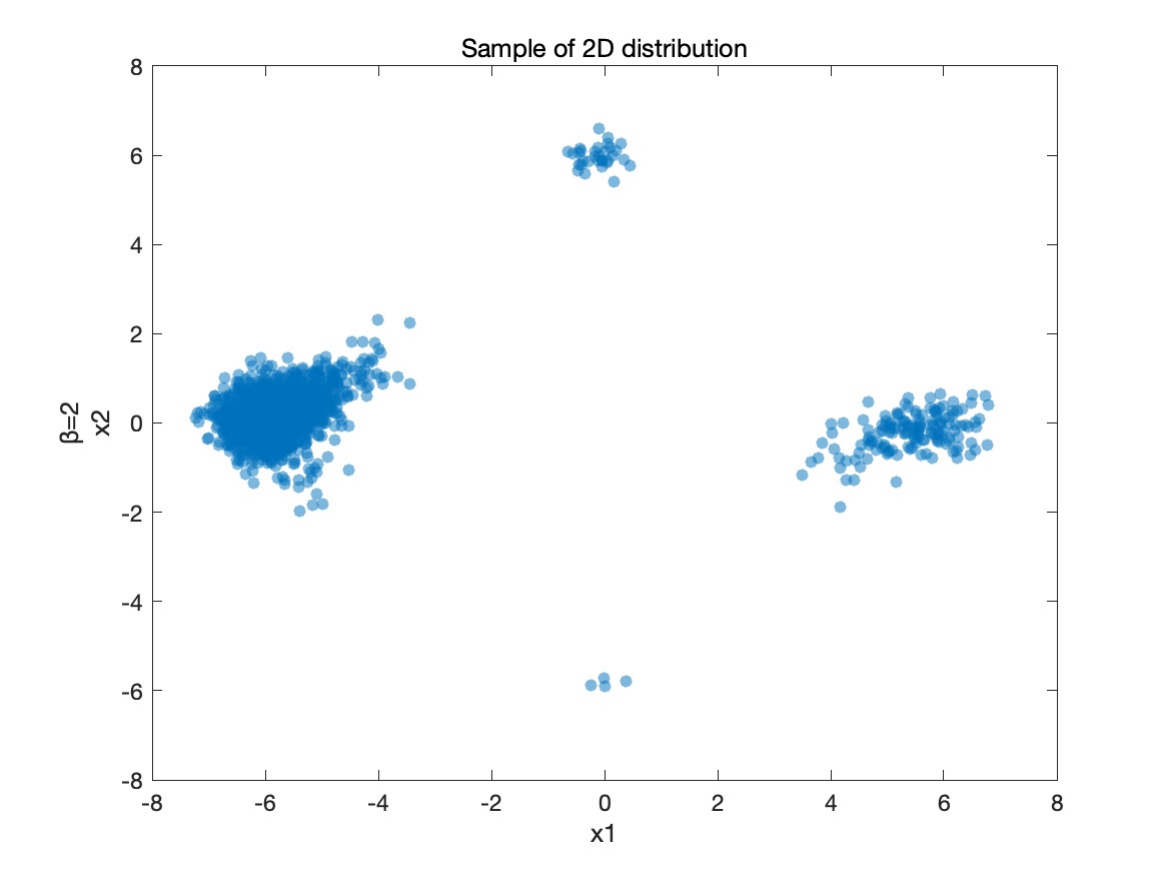}
        \label{fig:sub5}
    \end{subfigure} &
    \begin{subfigure}[b]{0.4\textwidth}
        \includegraphics[width=\textwidth]{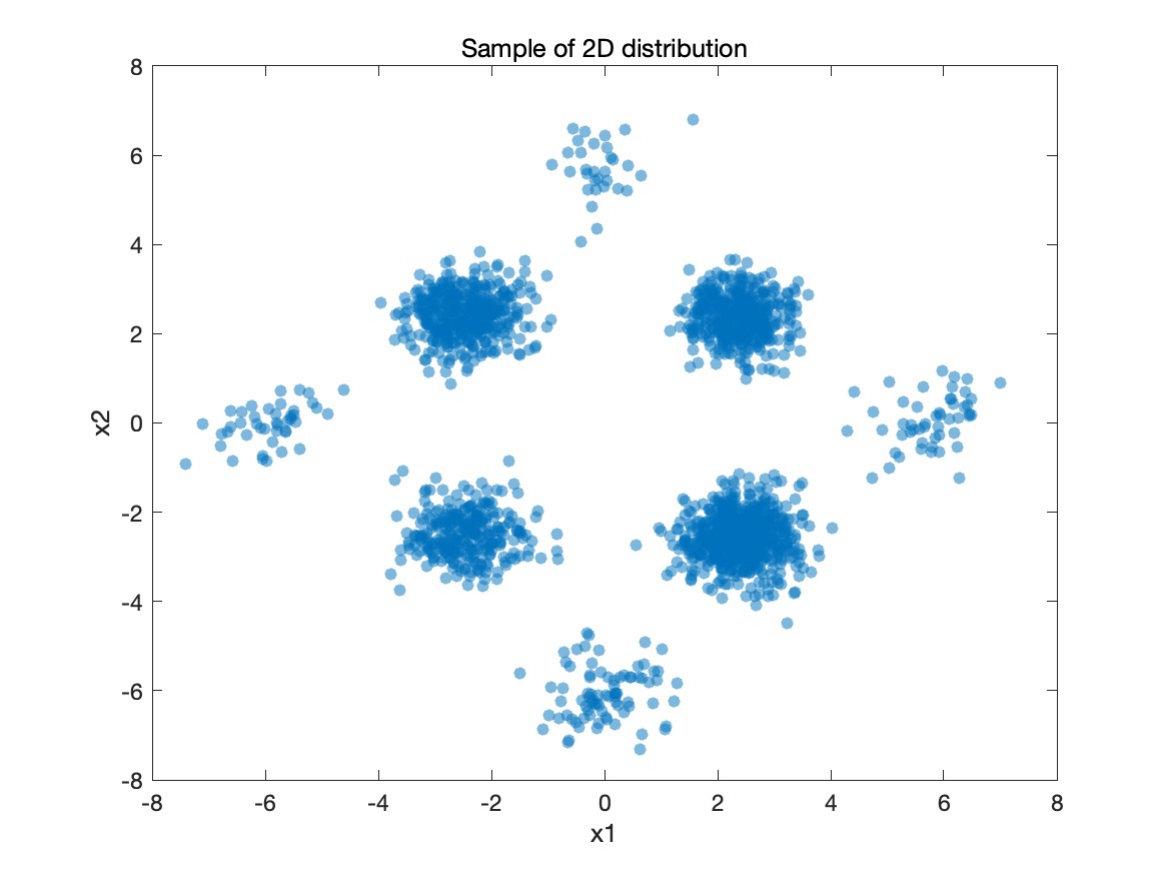}
        \label{fig:sub6}
    \end{subfigure} \\[-9pt]

    &
    \begin{subfigure}[b]{0.4\textwidth}
        \includegraphics[width=\textwidth]{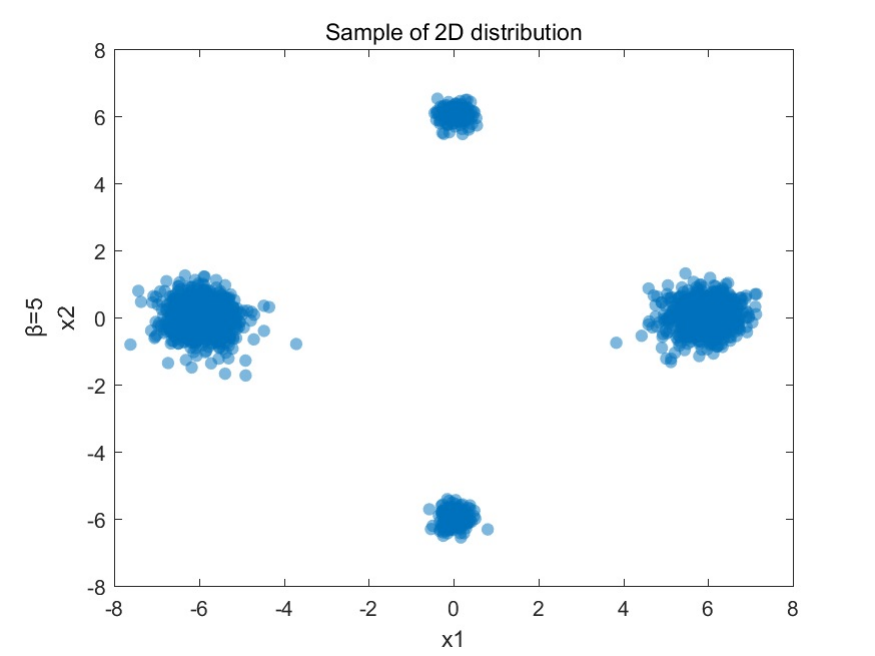}
        \label{fig:sub7}
    \end{subfigure} &
    \begin{subfigure}[b]{0.4\textwidth}
        \includegraphics[width=\textwidth]{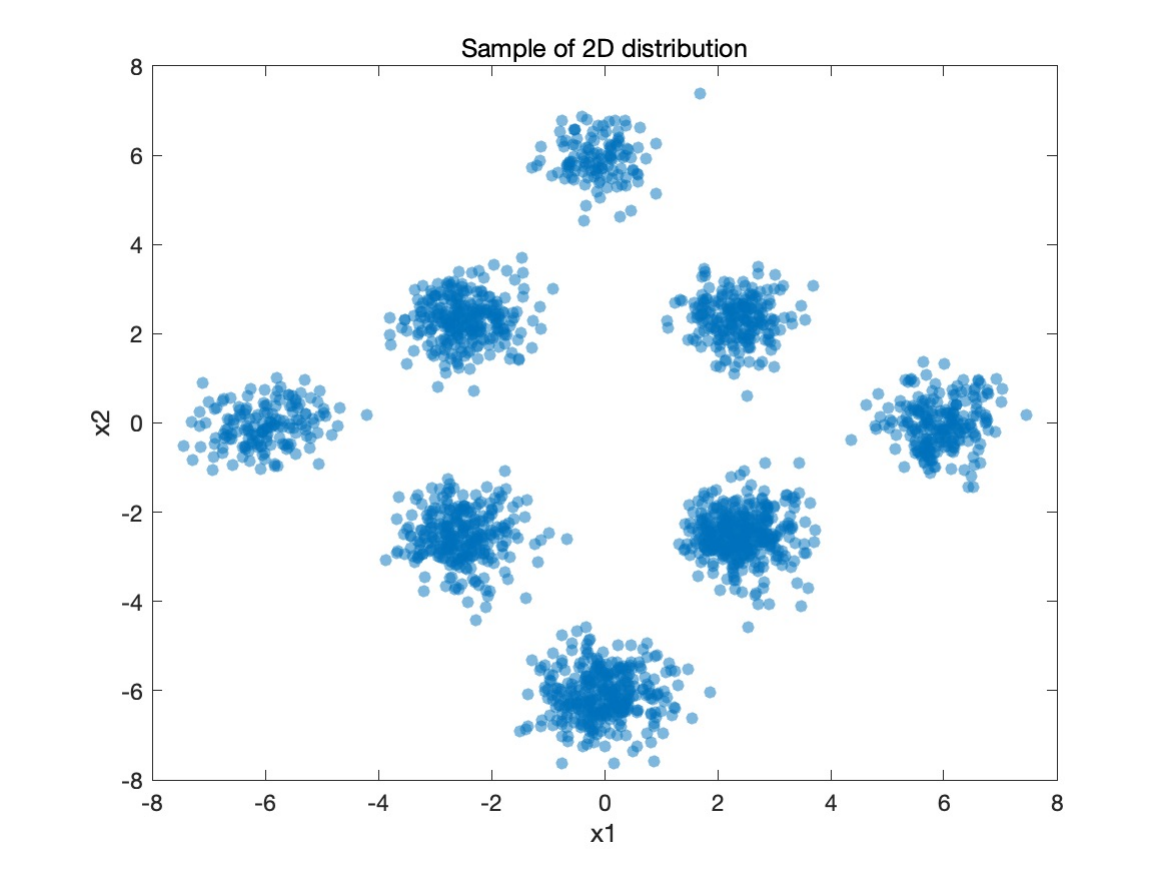}
        \label{fig:sub8}
    \end{subfigure} \\[-9pt]
    
    \end{tabular}
    \caption{Sampling via SFS with different temperatures $\beta$ (from top to bottom: $\beta =1, 2, 5$).}
    \label{fig:2-beta}
\end{figure}

\subsubsection{High-dimensional Gaussian mixture distributions}
In this subsection, we turn to sampling from high-dimensional Gaussian mixture distributions,
\begin{align*}
    p_6(x)
    =
    \theta_1
    \mathcal{N}\left(x ;\alpha_1,\Sigma_1\right)
    +
    \theta_2\,
    \mathcal{N}\left(x ; \alpha_2,\Sigma_2\right),
    \quad \theta_1+\theta_2=1 
    \text { and } 
    0 \leq \theta_1, \theta_2 \leq 1,
\end{align*}
where $\alpha_1, \alpha_2 \in \mathbb{R}^d$ and $\Sigma_1,\Sigma_2 \in \mathbb{R}^{d \times d}$ are the mean and covariance matrix of Gaussian component, respectively.

\textbf{Sampling via SFS with different temperatures $\beta$:}
Given $d=30$ and a uniform time step-size $h=10^{-3}$, we simulate $1000$ independent Markov chains using the SFS algorithm \eqref{SFS-W2-eq:Euler-scheme-no-mc} with different temperatures $\beta$.
By setting $\alpha_1=-6\mathbf{1}_d,\,\alpha_2=8\mathbf{1}_d,\,\Sigma_1=\Sigma_2=\tfrac{1}{4}\mathbf{I}_d$ and $\theta_1=\theta_2=\tfrac{1}{2}$,
Figure \ref{fig:gmd_30_beta} (a) displays the true contour plots and Figure \ref{fig:gmd_30_beta} (b)-(d) present scatter plots of SFS with different temperatures $\beta = 1, 2, 5$. 
As revealed by Figure \ref{fig:gmd_30_beta} (b), sampling via the SFS algorithm \eqref{SFS-W2-eq:Euler-scheme-no-mc} with $\beta=1$ exhibits mode collapse.
By contrast, the SFS algorithm with higher temperatures $\beta=2,5$ gives a much better performance. To be precise, high-quality samples are generated, successfully capturing the two modes of the target distribution, as evidently shown in Figure \ref{fig:gmd_30_beta} (c)-(d).
\begin{figure}[htbp]
    \centering
    \begin{tabular}{@{\hspace{-10pt}}c@{\hspace{-10pt}}c@{}}

    \begin{subfigure}{0.45\textwidth}
        \includegraphics[width=\linewidth]{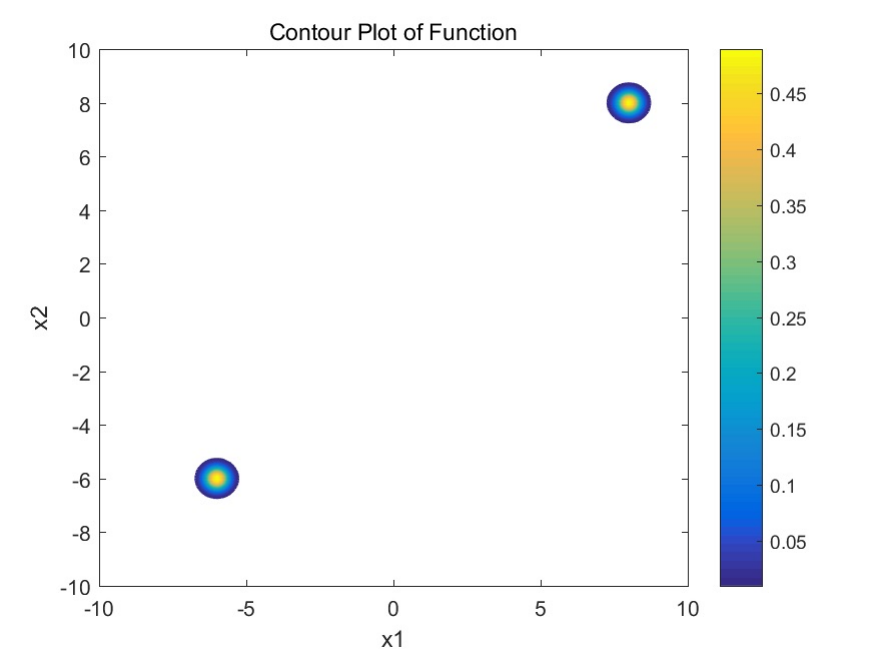} 
        \caption{True contour plot}
    \end{subfigure} &
    \begin{subfigure}{0.45\textwidth}
        \includegraphics[width=\linewidth]{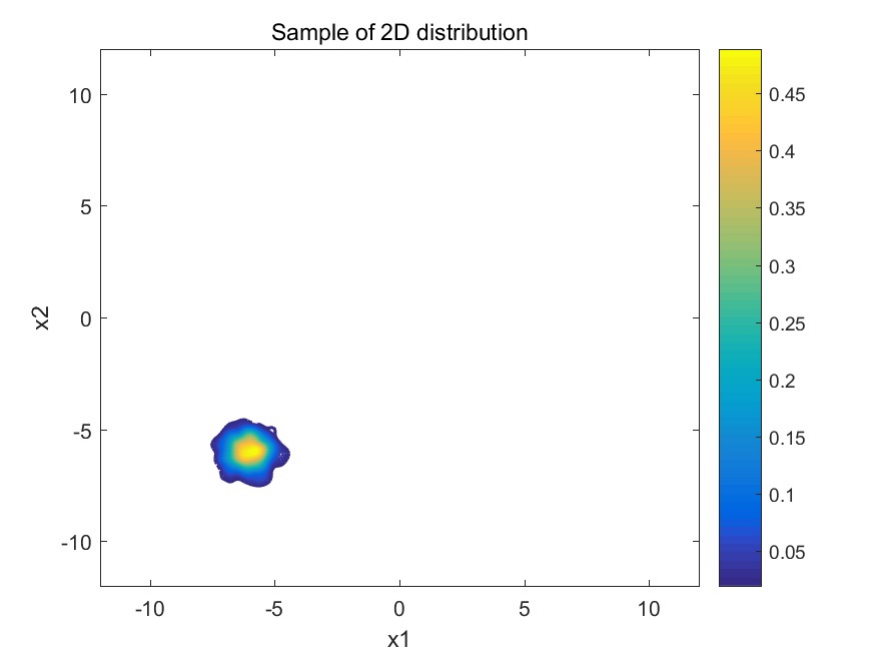}
        \caption{$\beta=1$}
    \end{subfigure} \\
    \\[-10pt]
    \begin{subfigure}{0.45\textwidth}
        \includegraphics[width=\linewidth]{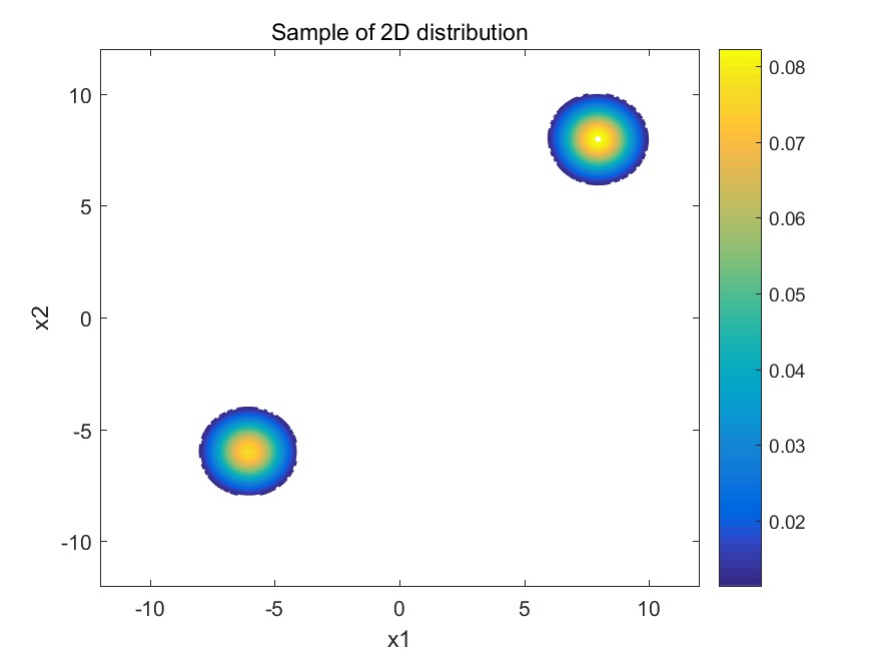}
        \caption{$\beta=2$}
    \end{subfigure} &
    \begin{subfigure}{0.45\textwidth}
        \includegraphics[width=\linewidth]{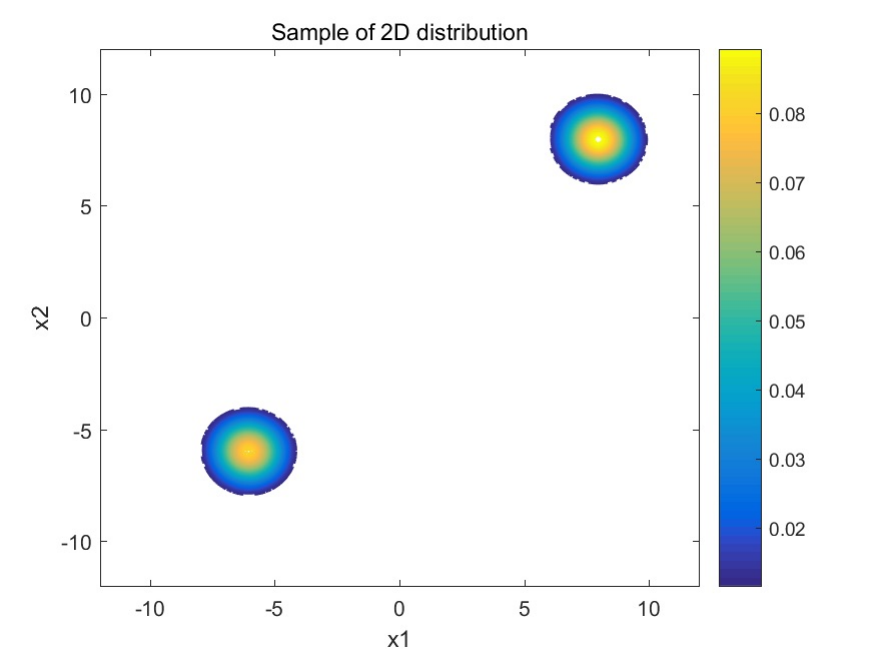}
        \caption{$\beta=5$}
    \end{subfigure}
    
    \end{tabular}
    \caption{Sampling via SFS with different temperatures $\beta = 1, 2, 5$.}
    \label{fig:gmd_30_beta}
\end{figure}

\textbf{Sampling via SFS and Langevin algorithms:}
By fixing the dimension $d=30$ and $\beta=1$, $\theta_1=\tfrac{3}{4},\theta_2=\tfrac{1}{4}$, we compare the SFS algorithm \eqref{SFS-W2-eq:Euler-scheme-no-mc} { at a terminal time $T=1$} with the traditional overdamped LMC algorithm \eqref{SFS-W2-eq:ULA} and with underdamped LMC algorithm \eqref{SFS-W2-eq:under-EM} { both at a terminal time $T=10$,} in the context of high-dimensional sampling. 
For a fixed step-size $h=10^{-3}$, we simulate $1000$ independent Markov chains.
%
In Figure \ref{figure:lmc_ulmc_diff_cov}, the PDF obtained via SFS is plotted using red circles, while those computed by the overdamped LMC and underdamped LMC are depicted using blue stars and green dashed line, respectively. For reference, the exact density curves are also plotted using black solid lines.

%
As clearly illustrated, when the distance between two means of the Gaussian components increases further (e.g., $\alpha_1=-4\cdot\mathbf{1}_d,-6\cdot\mathbf{1}_d,\,\alpha_2=4\cdot\mathbf{1}_d,6\cdot\mathbf{1}_d$), but with different covariances (e.g., $\Sigma_1=\tfrac{1}{5} \mathbf{I}_d,\ \Sigma_2=\tfrac{4}{5} \mathbf{I}_d$), samples from SFS still accurately capture the  multimodality of the target distribution.
By contrast, both the overdamped and underdamped LMC  collapse on one mode, as shown in Figure \ref{figure:lmc_ulmc_diff_cov}.

\textbf{Convergence rate:}
To test the convergence rates of the SFS algorithm \eqref{SFS-W2-eq:Euler-scheme-no-mc}, we fix $\beta=1$ and run the scheme using different step-sizes $h\in\{2^{-5},2^{-6},2^{-7},2^{-8},2^{-9}\}$ in different dimensions $d \in {\{1,2,6,10,20\}}$.
%
%
The exact solution is identified as the numerical one using a fine step-size $h_{ref}=2^{-13}$ and the expectations are approximated by computing averages over $3000$ samples. 
The reference line of order $1$ is also presented, depicted using black solid lines.
Figure \ref{figure:strong_gmd} presents the mean-square convergence rates of the SFS for Gaussian mixture distributions with means $\alpha_1 =-\mathbf{1}_d, -\tfrac{1}{2}\cdot\mathbf{1}_d,~\alpha_2=\mathbf{1}_d, \tfrac{1}{2}\cdot\mathbf{1}_d$ and  identical covariance matrices $\Sigma_1 = \Sigma_2 = \tfrac{4}{5}\mathbf{I}_d$.
From Figure \ref{figure:strong_gmd}, one can observe a mean-square convergence rate close to order $1$, {confirming the previous findings concerning the error estimates for the time discretization.}
\begin{figure}[htbp]
    \centering
    \begin{subfigure}{0.45\linewidth} 
        \includegraphics[width=\linewidth]{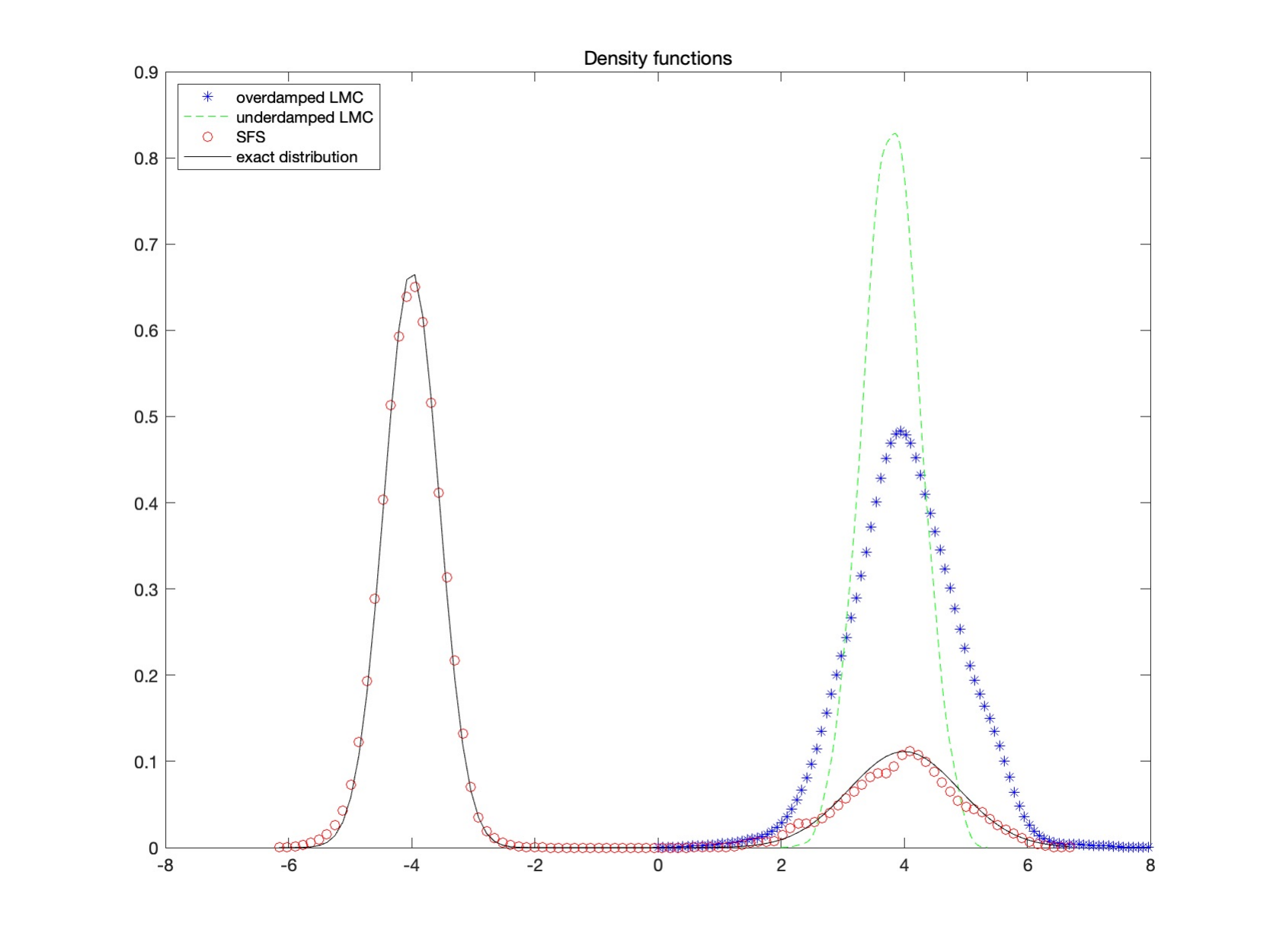} 
        \caption{\centering$\begin{aligned}
            &\alpha_1=-4\cdot\mathbf{1}_d,\ \alpha_2=4\cdot\mathbf{1}_d,\\
            &\Sigma_1=\tfrac{1}{5} \mathbf{I}_d,\ \Sigma_2=\tfrac{4}{5} \mathbf{I}_d
        \end{aligned}$}	
    \end{subfigure}
    \hspace{-0.5cm}
    \begin{subfigure}{0.45\linewidth}
        \includegraphics[width=\linewidth]{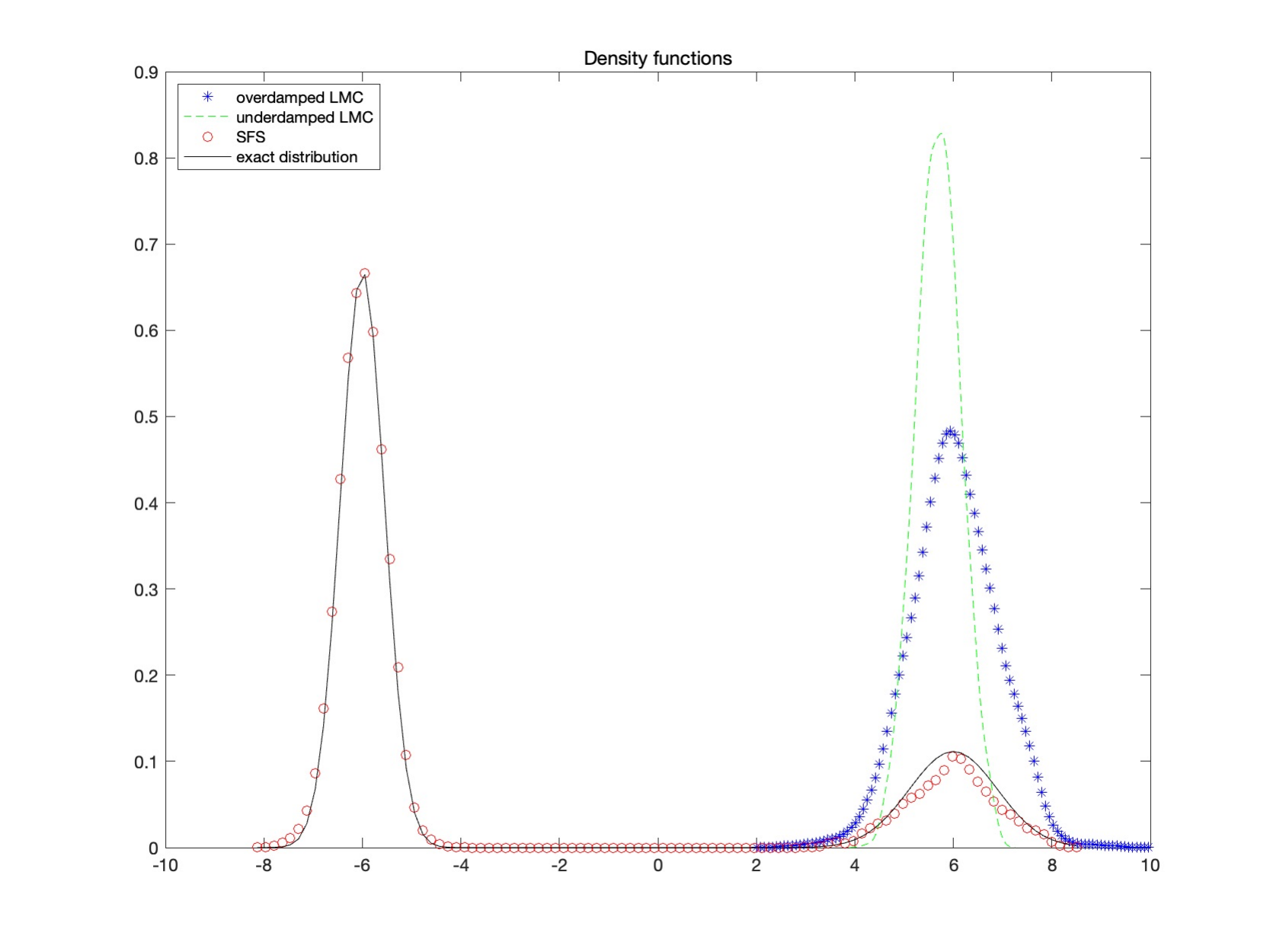}
        \caption{\centering$\begin{aligned}
            &\alpha_1=-6\cdot\mathbf{1}_d,\ \alpha_2=6\cdot\mathbf{1}_d,\\
            &\Sigma_1=\tfrac{1}{5} \mathbf{I}_d,\ \Sigma_2=\tfrac{4}{5} \mathbf{I}_d
        \end{aligned}$}	
    \end{subfigure}
    \caption{Probability density of the first component of the Gaussian mixture distribution.}
    \label{figure:lmc_ulmc_diff_cov}
\end{figure}

\begin{figure}[htbp]
	\centering
	\begin{subfigure}{0.45\linewidth}
		\includegraphics[width=\linewidth]{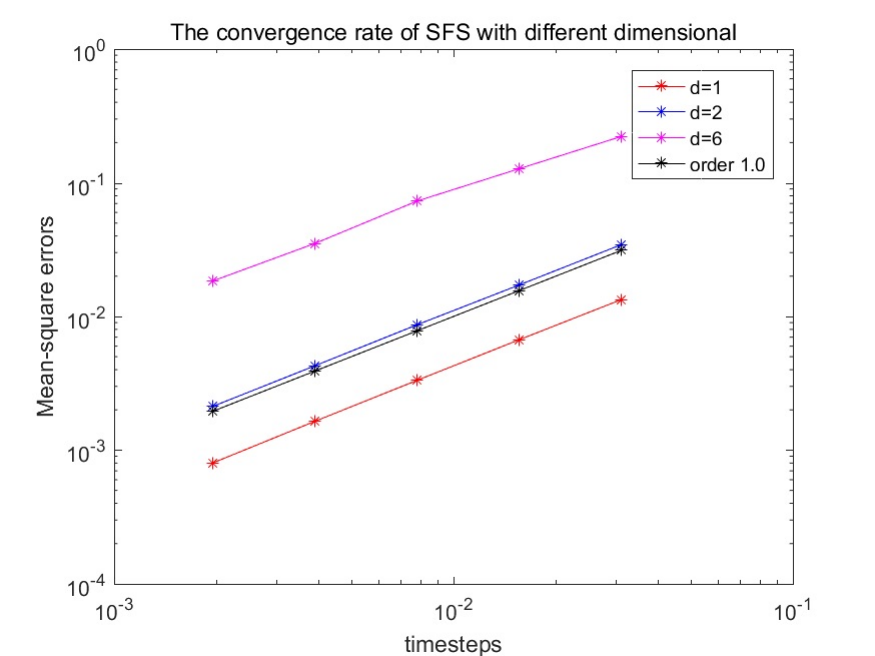}
		\caption{{ $\alpha_1 =-\mathbf{1}_d,~
        \alpha_2 =\mathbf{1}_d$
        \newline
        $\Sigma_1 = \Sigma_2 = \tfrac{4}{5}\mathbf{I}_d$
        }}
		
	\end{subfigure}
	\hspace{-0.5cm}
	\begin{subfigure}{0.45\linewidth}
	\includegraphics[width=\linewidth]{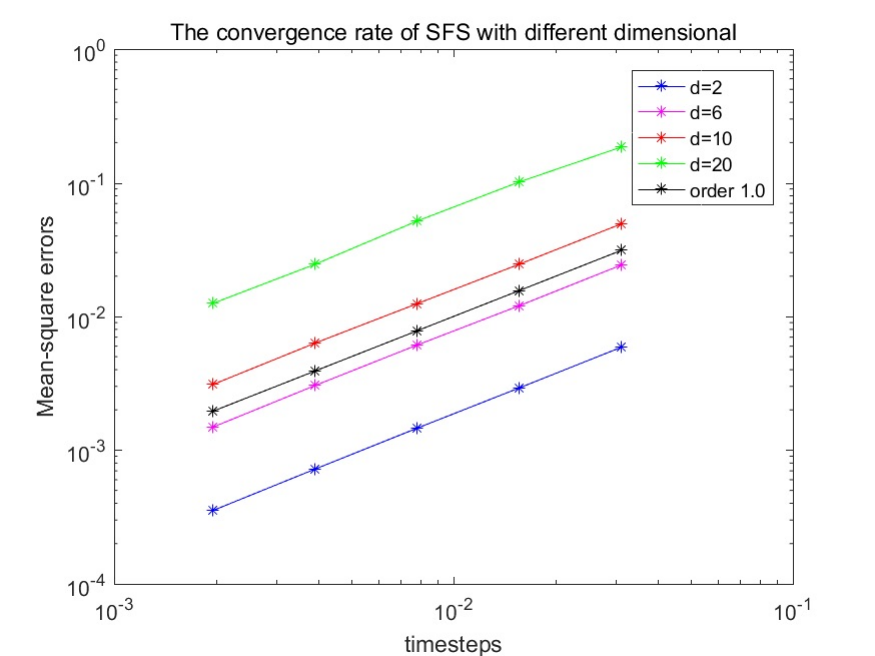}
		\caption{{ $\alpha_1 =-\tfrac{1}{2}\cdot\mathbf{1}_d,~
        \alpha_2 =\tfrac{1}{2}
        \cdot\mathbf{1}_d$
        \newline
        $\Sigma_1 = \Sigma_2 = \tfrac{4}{5}\mathbf{I}_d$
        }}
		
	\end{subfigure}
	\centering
	\caption{Mean-square convergence rates of SFS algorithms for the
 Gaussian mixture distributions.}
\label{figure:strong_gmd}
\end{figure}

\subsection{Bayesian ridge regression}
To illustrate the versatility of SFS in applications, we further examine its performance in a Bayesian ridge regression model,
$$
y=X \eta^*+\varepsilon,
\quad i=1, \ldots, n,
$$
where $X \in \mathbb{R}^{n \times d}$ denotes the covariance matrix, $y \in \mathbb{R}^n$ the response variable, $\eta^*=\left(\eta_1^*, \ldots, \eta_d^*\right)^{T} \in \mathbb{R}^d$ the vector of underlying regression coefficients, and $\varepsilon \sim N\left(0, \sigma_1 \mathbf{I}_n\right)$ the noise. We assume a prior for the coefficients as $\eta \sim N\left(0, \sigma_2 \mathbf{I}_d\right)$, leading to a posterior distribution of
$$
P(\eta \mid y, X) 
\propto 
\exp \left(
-\left(
\frac{\|y-X \eta\|^2}
{2 \sigma_1^2}
+
\frac{\|\eta\|^2}
{2 \sigma_2^2}
\right)
\right).
$$
In our experiments, we set $n = d $ and $X =\mathbf{I}_d$ for simplicity.

%
\textbf{Convergence rate:}
To test the convergence rate, we assign $\beta=1$ and run the SFS algorithm \eqref{SFS-W2-eq:Euler-scheme-mc} for the SDEs in different dimensions $d \in {\{1,2,6\}}$, using different step-sizes $h\in\{2^{-5},2^{-6},2^{-7},2^{-8},2^{-9}\}$.
We consider two cases: $\sigma_1=\sigma_2=0.03$ and $\sigma_1=\sigma_2=0.01 $.
%
%
The exact solutions are identified as the numerical ones using a fine step-size $h=2^{-13}$. 
The reference lines of slope $1$ are also presented, depicted as black solid lines. 
From Figure \ref{figure:strong_brr}, it is easy to detect the mean-square convergence rate close to order $1$, matching the theoretical results.
%

\begin{figure}[htbp]
	\centering
	\begin{subfigure}{0.45\linewidth}
		\includegraphics[width=\linewidth]{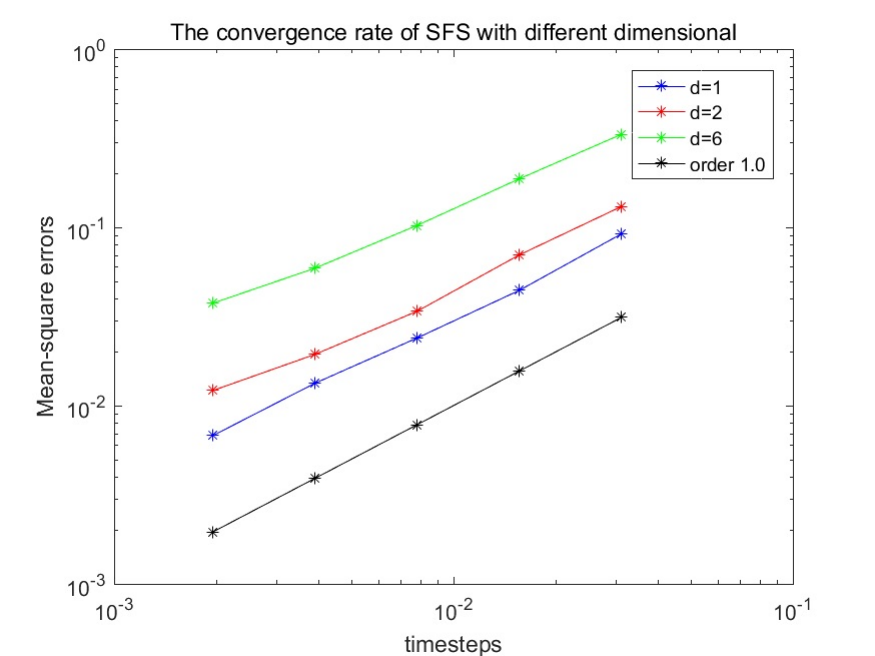}
		\caption{$\sigma_1=\sigma_2=0.03 $}
		
	\end{subfigure}
	\hspace{-0.5cm} 
	\begin{subfigure}{0.45\linewidth}	\includegraphics[width=\linewidth]{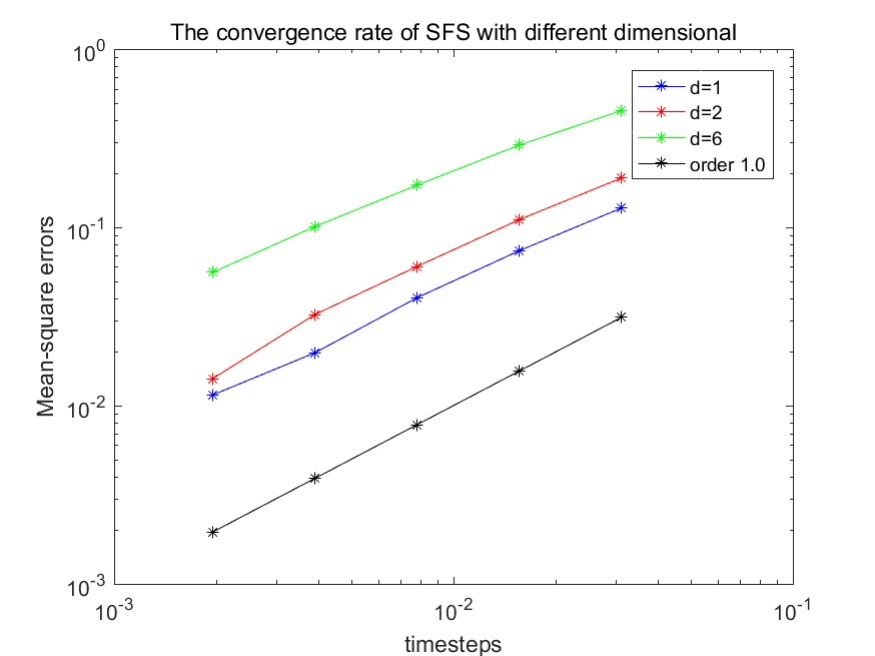}
		\caption{$\sigma_1=\sigma_2=0.01 $}
		
	\end{subfigure}
	\caption{Mean-square convergence rates of SFS algorithm for the
    Bayesian ridge regression.}
        \label{figure:strong_brr}
\end{figure}

\section{Conclusion}
In this work, we provide enhanced error bounds in the $L^2$-Wasserstein distance for two variants of Schr\"odinger-F\"ollmer samplers. 
By carrying out more delicate error estimates, {we derive an enhanced convergence rate of order $\mathcal{O}({ {h}})$ for the time discretization of the Schr\"odinger-F\"ollmer diffusion under certain smoothness conditions on the drift,} significantly improving the existing convergence rate of order $\mathcal{O}(\sqrt{h})$ obtained in \cite{MRjiao}. 
As ongoing projects \cite{Wang2025randomization,Lin2025SRK}, we propose more efficient (higher-order) sampling algorithms based on such Schr\"odinger-F\"ollmer diffusion process with temperatures.

\appendix

\section{Proof of Proposition \ref{SFS-W2-prop:g-satisfy-condition}}
\label{SFS-W2-appendix:g-satisfy-condition}
\begin{proof}
\textbf{Proof.}
Under the given assumptions, the function $g_{\beta}$ is of class $\mathcal{C}^3$, and moreover, $g_{\beta}, \nabla g_{\beta},\nabla^2 g_{\beta}$ are Lipschitz continuous. Then there exists $L_g>0$ such that, for any $ x,y \in \mathbb{R}^d,$
\begin{align}
\label{SFS-W2-eq:g-nabla-is-lipschitz}
        \|\nabla^k g_{\beta}(x)
        -
        \nabla^k g_{\beta}(y)\| 
        \leq &
        L_g
        \|x-y\|,
        \quad
        k=0,1,2.
\end{align}
{ 
For any $x,\alpha_1,\alpha_2, \alpha_3\in \mathbb{R}^d$ satisfying $\|\alpha_1\|=\|\alpha_2\|=\|\alpha_3\|=1$,
in light of definitions of the operator norm \eqref{SFS-W2-eq:operator-norm-def} and the directional derivatives \eqref{SFS-W2-eq:def-dir-der} one can deduce} 
\begin{align*}
        \|\nabla Q_{1-t}^{\beta} g_{\beta}(x)\|
        & =
        \sup_{\|\alpha_1\|=1}
        \big\|
        \nabla_{\alpha_1} Q_{1-t}^{\beta} g_{\beta}(x)
        \big\|\\
        & =
        \sup_{\|\alpha_1\|=1}
        \left\|
        \lim_{\varepsilon \rightarrow 0} 
        \frac{Q_{1-t}^{\beta}g_{\beta}\left(x+\varepsilon \alpha_1\right)
        -
        Q_{1-t}^{\beta}g_{\beta}(x)}
        {\varepsilon}
        \right\|\\
        & \leq
        \sup_{\|\alpha_1\|=1}
        \lim_{\varepsilon \rightarrow 0}
        \frac{
        \mathbb{E}
        \big[
        \|g_{\beta}(x+\varepsilon \alpha_1 +\sqrt{(1-t)\beta}\,\xi)
        -
        g_{\beta}
        (x+\sqrt{(1-t)\beta}\,\xi)\|
        \big]}
        {\varepsilon}\\
        & \leq
        L_g,
\end{align*}
\begin{align*}
        \|\nabla^2 Q_{1-t}^{\beta} g_{\beta}(x)\|
        & =
        \sup_{
        \|\alpha_1\|=\|\alpha_2\|=1}
        \big\|
        \nabla_{\alpha_2}
        \nabla_{\alpha_1} 
        Q_{1-t}^{\beta} g_{\beta}(x)
        \big\|\\
        & =
        \sup_{
        \|\alpha_1\|=\|\alpha_2\|=1}
        \big\|
        \nabla_{\alpha_2}
        \big\langle
        \nabla 
        Q_{1-t}^{\beta} g_{\beta}(x),
        \alpha_1
        \big \rangle
        \big\|\\
        & =
        \sup_{
        \|\alpha_1\|=\|\alpha_2\|=1}
        \left\|
        \lim_{\varepsilon \rightarrow 0} 
        \frac{
        \big\langle
        \nabla 
        Q_{1-t}^{\beta}
        g_{\beta}(x+\varepsilon \alpha_2),
        \alpha_1
        \big \rangle
        - 
        \big\langle
        \nabla
        Q_{1-t}^{\beta}g_{\beta}(x),
        \alpha_1
        \big \rangle}
        {\varepsilon}
        \right\|\\
        & \leq
        \sup_{\|\alpha_2\|=1}
        \lim_{\varepsilon \rightarrow 0}
        \frac{
        \big\|\nabla
        Q_{1-t}^{\beta} g_{\beta}(x+\varepsilon \alpha_2)
        -
        \nabla
        Q_{1-t}^{\beta} g_{\beta}(x)
        \big\|} 
        {\varepsilon}\\
        & \leq
        \sup_{\|\alpha_2\|=1}
        \lim_{\varepsilon \rightarrow 0}
        \frac{
        \mathbb{E}
        \big[
        \|\nabla g_{\beta}(x+\varepsilon \alpha_2+\sqrt{(1-t)\beta}\,\xi)
        -
        \nabla g_{\beta}(x+\sqrt{(1-t)\beta}\,\xi)\|
        \big]}
        {\varepsilon}\\
        & \leq
        L_g,
\end{align*}
and
\begin{align*}
        &
        \|\nabla^3 Q_{1-t}^{\beta} g_{\beta}(x)\|\\
        & =
        \sup_{
        \|\alpha_1\|
        =\|\alpha_2\|=\|\alpha_3\|
        =1}
        \big\|
        \nabla_{\alpha_3}
        \nabla_{\alpha_2}
        \nabla_{\alpha_1}
        Q_{1-t}^{\beta} g_{\beta}(x)
        \big\|\\
        & =
        \sup_{
        \|\alpha_1\|
        =\|\alpha_2\|=\|\alpha_3\|
        =1}
        \bigg\|
        \nabla_{\alpha_3}
        \Big \langle
        \nabla
        \Big(
        \big \langle
        \nabla
        Q_{1-t}^{\beta} g_{\beta}(x),
        \alpha_1
        \big\rangle
        \Big),
        \alpha_2
        \Big \rangle
        \bigg\|\\
        & =
        \sup_{
        \|\alpha_1\|
        =\|\alpha_2\|=\|\alpha_3\|
        =1}
        \left\|
        \lim_{\varepsilon \rightarrow 0} 
        \frac{
        \Big \langle
        \nabla
        \Big(
        \big \langle
        \nabla
        Q_{1-t}^{\beta} g_{\beta}(x+\varepsilon \alpha_3),
        \alpha_1
        \big\rangle
        \Big),
        \alpha_2
        \Big \rangle
        - 
        \Big \langle
        \nabla
        \Big(
        \big \langle
        \nabla
        Q_{1-t}^{\beta} g_{\beta}(x),
        \alpha_1
        \big\rangle
        \Big),
        \alpha_2
        \Big \rangle}
        {\varepsilon}
        \right\|\\
        & \leq
        \sup_{
        \|\alpha_1\|=\|\alpha_3\|=1}
        \lim_{\varepsilon \rightarrow 0} 
        \frac{
        \left\|
        \nabla 
        \Big(
        \big \langle
        \nabla
        Q_{1-t}^{\beta} g_{\beta}(x+\varepsilon \alpha_3),
        \alpha_1
        \big\rangle
        \Big)
        -
        \nabla
        \Big(
        \big \langle
        \nabla
        Q_{1-t}^{\beta} g_{\beta}(x),
        \alpha_1
        \big\rangle
        \Big)
        \right\|}
        {\varepsilon}\\
        & \leq
        \sup_{\|\alpha_3\|=1}
        \lim_{\varepsilon \rightarrow 0}
        \frac{
        \big\|\nabla^2
        Q_{1-t}^{\beta} g_{\beta}(x+\varepsilon \alpha_3)
        -
        \nabla^2
        Q_{1-t}^{\beta} g_{\beta}(x)
        \big\|} 
        {\varepsilon}\\
        & \leq
        \sup_{\|\alpha_3\|=1}
        \lim_{\varepsilon \rightarrow 0}
        \frac{
        \mathbb{E}
        \big[
        \|\nabla^2 g_{\beta}(x+\varepsilon \alpha_3+\sqrt{(1-t)\beta}\,\xi)
        -
        \nabla^2 g_{\beta}(x+\sqrt{(1-t)\beta}\,\xi)\|
        \big]}
        {\varepsilon}\\
        & \leq
        L_g.
\end{align*}
Since $g_{\beta}$ has a lower bound greater than $0$, i.e.,
$
    g_{\beta} \geq \rho >0,
$
one can infer that for any $x, v_1 \in \mathbb{R}^d$ and $t\in [0,1]$,
\begin{align*}
        \|D f_{\beta}(x,t)v_1\|
        = &
        \Bigg\|
        \frac{
        \beta
        \nabla^2
        Q_{1-t}^{\beta} g_{\beta}(x)
        Q_{1-t}^{\beta} g_{\beta}(x)
        -
        \beta
        \nabla 
        Q_{1-t}^{\beta} g_{\beta}(x)
        \nabla Q_{1-t}^{\beta} g_{\beta}(x)^T
        }
        { {(Q^{\beta}_{1-t} g_{\beta}(x))^2}}
        \cdot v_1
        \Bigg\|\\
        \leq &
        \Bigg\|
        \frac{
        \beta
        \nabla^2
        Q_{1-t}^{\beta} g_{\beta}(x)}
        {Q_{1-t}^{\beta} g_{\beta}(x)}
        \cdot
        v_1
        \Bigg\|
        +
        \Bigg\|
        \frac{
        \beta
        \nabla 
        Q_{1-t}^{\beta} g_{\beta}(x)
        \nabla Q_{1-t}^{\beta} g_{\beta}(x)^T}
        { {(Q^{\beta}_{1-t} g_{\beta}(x))^2}}
        \cdot
        v_1
        \Bigg\|\\
        \leq &
        \left(
        1+
        \frac{L_g}{\rho}
        \right)
        \frac{L_g \beta}{\rho}        
        \|v_1\|. 
\end{align*}
In the same manner, we show, for any $ x, v_1, v_2\in \mathbb{R}^d$ and $t\in [0,1]$,
\begin{align*}
        \|D^2 f_{\beta}(x,t)(v_1,v_2)\|
        & \leq 
        \Bigg\|
        \frac{
        \beta
        \nabla
        \big(
        \nabla^2
        Q_{1-t}^{\beta} g_{\beta}(x)v_1
        \big)
        v_2
        }
        {Q_{1-t}^{\beta} g_{\beta}(x)}
        \Bigg\|
        +
        \Bigg\|
        \frac{
        \beta
        \big(
        \nabla^2 
        Q_{1-t}^{\beta} g_{\beta}(x)v_1
        \nabla Q_{1-t}^{\beta} g_{\beta}(x)^T
        \big) 
        v_2}
        { {(Q^{\beta}_{1-t} g_{\beta}(x))^2}}
        \Bigg\|\\
        & \quad +
        \Bigg\|
        \frac{
        \beta
        \nabla 
        \big(
        \nabla Q_{1-t}^{\beta} g_{\beta}(x)
        \nabla Q_{1-t}^{\beta} g_{\beta}(x)^T v_1
        \big)
        v_2}
        { {(Q^{\beta}_{1-t} g_{\beta}(x))^2}}
        \Bigg\|\\
        & \quad +
        2\beta
        \Bigg\|
        \frac{
        \big(
        \nabla Q_{1-t}^{\beta} g_{\beta}(x)
        \nabla Q_{1-t}^{\beta} g_{\beta}(x)^T v_1
        \nabla Q_{1-t}^{\beta} g_{\beta}(x)^T 
        \big)
        v_2}
        { {(Q^{\beta}_{1-t} g_{\beta}(x))^3}}
        \Bigg\|\\
        & \leq 
        \frac{
        \beta
        \big\|
        \nabla^3 Q_{1-t}^{\beta} g_{\beta}(x)
        \big)
        \big\|}
        {\rho}
        \cdot \|v_1\|
        \cdot \|v_2\|
        \\
        & \quad +
        \frac{
        \beta
        \big\|
        \nabla^2 Q_{1-t}^{\beta} g_{\beta}(x)
        \big\|
        \cdot
        \big\|
        \nabla Q_{1-t}^{\beta} g_{\beta}(x)
        \big\|}
        {\rho^2}
        \cdot \|v_1\|
        \cdot \|v_2\|
        \\
        & \quad +
        \frac{
        \beta
        \big\|
        \nabla 
        \big(
        \nabla Q_{1-t}^{\beta} g_{\beta}(x)
        \nabla 
        Q_{1-t}^{\beta}
        g_{\beta}(x)^T v_1
        \big)
        \big\|}
        {\rho^2}
        \cdot \|v_2\|\\
        & \quad +
        2 \beta
        \frac{
        \big\|
        \nabla Q_{1-t}^{\beta} g_{\beta}(x) 
        \big\|^3}
        {\rho^3}
        \cdot \|v_1\|
        \cdot \|v_2\|\\
        & \leq
        \left(
        1+
        \frac{3L_g}{\rho}
        +
        \frac{2 L_g^2}{\rho^2}
        \right)
        \frac{L_g \beta}{\rho}
        \|v_1\|
        \cdot\|v_2\|. 
\end{align*}
Regarding $\|\partial_t f_{\beta}(x,t)\|$,
using the fact that $g_{\beta}, \nabla g_{\beta}$ are Lipschitz and the definition of the operator norm \eqref{SFS-W2-eq:operator-norm-def}, one can similarly get 
\begin{align}
\label{SFS-W2-eq:g_nabla_g_bound}
    \|\nabla g_{\beta}(x)\|
    \leq 
    L_g,
    \quad 
    \|\nabla^2 g_{\beta}(x)\|
    \leq 
    L_g,
    \quad \forall x \in \mathbb{R}^d.
\end{align}
Combining the triangle inequality and \eqref{SFS-W2-eq:g_nabla_g_bound}, we arrive at
\begin{align*}
        \big\|
        \partial_t f_{\beta}(x, t)
        \big\|
        & =  
        \Bigg\|
        \frac{
        -
        \beta^2
        \mathbb{\,E}
        [\nabla^2 g_{\beta}(x+\sqrt{(1-t)\beta}\,\xi)\,\xi]
        \mathbb{E}
        [g_{\beta}(x+\sqrt{(1-t)\beta}\,\xi)]
        }
        {2
        \big(
        \mathbb{E}
        [g_{\beta}(x+\sqrt{(1-t)\beta}\,\xi)]
        \big)^2
        \sqrt{(1-t)\beta}
        }\\
        & \quad +
        \frac{
        \beta^2
        \mathbb{\,E}
        [\nabla g_{\beta}(x+\sqrt{(1-t)\beta}\,\xi)]
        \mathbb{E}
        [
        \langle
        \nabla g_{\beta}(x+\sqrt{(1-t)\beta}\,\xi),
        \xi
        \rangle]
        }
        {2
        \big(
        \mathbb{E}
        [g_{\beta}(x+\sqrt{(1-t)\beta}\,\xi)]
        \big)^2
        \sqrt{(1-t)\beta}}
        \Bigg\|\\
        & \leq  
        \frac{L_g \beta^2}
        {2 \rho \sqrt{(1-t)\beta}}
        \mathbb{E}
        \big[\|\xi\|\big]
        + 
        \frac{L_g^2
        \beta^2}
        {2 \rho^2 \sqrt{(1-t)\beta}}
        \mathbb{E}
        \big[\|\xi\|\big]\\
        & \leq
        \Big(
        1+
        \frac{L_g}
        {\rho}
        \Big)
        \,
        \frac{L_g\beta^{3/2}}
        {2 \rho }
        \,
        d^\frac{1}{2}
        \frac{1}{\sqrt{1-t}}.
\end{align*}
The proof is completed.
\end{proof}

\section{Proof of Proposition \ref{SFS-W2-prop:mixing-time}}
\label{SFS-W2-appendix:mixing-time}
\begin{proof}
\textbf{Proof.}
Given an error tolerance $\epsilon>0$, Theorem \ref{SFS-W2-thm:main-rerult-mc} { tells that, for $M$ being large enough and $h$ being small enough such that
\begin{align}
\label{SFS-W2-eq:mixing-time-two-term}
    C d h 
    \leq 
    \frac{\epsilon}{2}, 
    \quad 
    C \sqrt{\frac{d}{M}}
    \leq 
    \frac{\epsilon}{2},
\end{align}
one can arrive at}
\begin{align*}
        \mathcal{W}_2
        \big(
        \mathcal{L}aw
        (
        \widetilde{Y}^M_1
        ), 
        \mu
        \big)
        \leq 
        \epsilon.
\end{align*}
{ 
Rearranging the first inequality of \eqref{SFS-W2-eq:mixing-time-two-term} gives  
\begin{align*}
    N=\frac{1}{h}
    \geq 
    \frac{2Cd}{\epsilon}.
\end{align*}}
The second part of inequality \eqref{SFS-W2-eq:mixing-time-two-term} requires
\begin{align*}
    M 
    \geq 
    \frac{4C^2d}{\epsilon^2}.
\end{align*}
{Therefore, the number of evaluations of $g_{\beta}$ is of order $\mathcal{O}(\tfrac{d^2}{\epsilon^3}).$}
This completes the proof of the proposition.
\end{proof}

{ 
\section{Proof of Proposition \ref{sfs-w2-prop:ass-v}}
\label{sfs-w2-app:ass-v}
\begin{proof}
\textbf{Proof.}
Since $V \in \mathcal{C}^3$ satisfies \eqref{sff-w2-eq:v-condition} under the given assumptions, it follows that
\begin{align}
\big\|
g_{\beta}(x)
-
g_{\beta}(y)
\big\|
& =
\bigg\|
\int_0^1
\nabla 
g_{\beta}
\big(
y+r(x-y)
\big)
(x-y)
\mathrm{d} r
\bigg\|
\nonumber\\
& \leq
\big\|
\nabla 
g_{\beta}
\big(
y+r(x-y)
\big)
\big\|
\cdot
\|x-y\|
\nonumber\\
& \leq
\sup_{x\in\mathbb{R}^d}
\big\|
\beta^{-1}
x
-
\nabla V(x)
\big\|
\cdot
\exp
\big(
-V(x)
+
\tfrac{\|x\|^2}
{2\beta}
\big)
\cdot
\|x-y\|
\nonumber\\
& \leq
L_g \|x-y\|,
\quad
\forall \, x,y \in \mathbb{R}^d,
\end{align}
where the first equality is obtained by employing the Taylor expansion. In the same manner,
\begin{align}
\big\|
\nabla g_{\beta}(x)
-
\nabla
g_{\beta}(y)
\big\|
& =
\bigg\|
\int_0^1
\nabla^2 
g_{\beta}
\big(
y+r(x-y)
\big)
(x-y)
\mathrm{d} r
\bigg\|
\nonumber\\
& \leq
\big\|
\nabla^2 
g_{\beta}
\big(
y+r(x-y)
\big)
\big\|
\cdot
\|x-y\|
\nonumber\\
& \leq
\sup_{x\in\mathbb{R}^d}
\big\|
\beta^{-1}
x
-
\nabla V(x)
\big\|^2
\cdot
\exp
\big(
-V(x)
+
\tfrac{\|x\|^2}
{2\beta}
\big)
\cdot
\|x-y\|
\nonumber\\
& \quad +
\sup_{x\in\mathbb{R}^d}
\big\|
\beta^{-1}
\mathbf{I}_d
-
\nabla^2 V(x)
\big\|
\cdot
\exp
\big(
-V(x)
+
\tfrac{\|x\|^2}
{2\beta}
\big)
\cdot
\|x-y\|
\nonumber\\
& \leq
L_g \|x-y\|,
\quad
\forall \, x,y \in \mathbb{R}^d.
\end{align}
The Lipschitz continuity of $\nabla^2 
g_{\beta}$ is proved similarly and will not be repeated here for brevity.
The proof is completed.

\end{proof}}
\bibliography{reference}

\end{document}